%% file: conn_comp_streamlined_arxiv.tex
\documentclass[10pt]{amsart}
\usepackage[nohug,heads=littlevee]{diagrams}
\usepackage{mathrsfs}
\usepackage{bm}
\usepackage{stmaryrd}
\usepackage[bbgreekl]{mathbbol}
\usepackage{amssymb}
\usepackage[latin1]{inputenc}
\usepackage{xr-hyper}
\usepackage[plainpages=false,pdfpagelabels]{hyperref}
\usepackage[left=2cm,top=2.5cm,right=2cm,bottom=1.4in,asymmetric]{geometry}
\usepackage{mathtools}
\usepackage[backend=bibtex]{biblatex}
\usepackage{enumitem}
\usepackage{color}

\usepackage[all]{xy}

\input{commandssubsub}

\bibliography{irred}

\begin{document}
\title[Connected components of special cycles]{Connected components of special cycles on Shimura varieties}
\author{Keerthi Madapusi}
\address{Keerthi Madapusi\\%
Department of Mathematics\\%
Maloney Hall\\%
Boston College\\%
Chestnut Hill, MA 02467\\%
USA}
\email{madapusi@bc.edu}

\begin{abstract}
I use methods of Chai-Hida and ordinary $p$-Hecke correspondences to study the set of irreducible components of special fibers of special cycles of sufficiently low codimension in integral models of GSpin Shimura varieties, and apply this to prove irreducibility results for the special fibers of the moduli of polarized K3 surfaces. These results are also applied in joint work with Howard on the modularity of generating series of higher codimension cycles on GSpin Shimura varieties.
\end{abstract}

\maketitle

\section*{Introduction}

The goal of this paper is to use $p$-Hecke correspondences---following in spirit a paper of de Jong~\cite{De_Jong1993-al}---to obtain control over irreducible components of the special fibers of certain cycles on GSpin Shimura varieties. More precisely, I show that, under some numerical constraints, each such component is the specialization of a unique component in the generic fiber. This plays a key role in the proof of the main result in~\cite{HMP:mod_codim} on the modularity of generating series of special cycles of higher codimension. 

As a more immediate application, I give an application to the moduli of polarized K3 surfaces. Fix a positive integer $d\geq 1$, and consider the moduli stack $\mathsf{M}^{\circ}_{2d}$ over $\Int$ that parameterizes primitively polarized K3 surfaces of degree $2d$. I prove:

\begin{introthm}\label{introthm:k3s}
For every prime $p$, the fiber $\mathsf{M}^{\circ}_{2d,\Field_p}$ is geometrically irreducible.
\end{introthm}
When $p$ is odd and $p^2\nmid d$, this result was shown in~\cite{mp:tatek3}. In effect, what's happening in this paper is a reduction to this known case, by relating the ordinary locus of $\mathsf{M}^{\circ}_{2d/p^2}$ with that of $\mathsf{M}^{\circ}_{2d}$ using $p$-Hecke correspondences.

To state the somewhat technical main result I prove in this paper, we begin with a GSpin Shimura variety $\Sh_K$ associated with a quadratic lattice $V_{\Int}$ of signature $(n,2)$. On this variety, we have an infinite collection of special cycles $Z_K(\Lambda)\to \Sh_K$ associated with positive definite lattices $\Lambda$, which in the complex analytic fiber parameterize Noether-Lefschetz type loci where the canonical family of polarized $\Int$-Hodge structures $\bm{V}$ associated with $V_{\Int}$ pick up a collection of Hodge tensors generating a subspace isometric to $\Lambda$. As explained in the series of papers~\cites{mp:reg,Andreatta2017-gu,Andreatta2018-tt,Howard2020-fa,HMP:mod_codim}, the Shimura variety admits a natural integral model $\Ss_K$ over $\Int$, along with a cycle $\mathcal{Z}_K(\Lambda)$ extending $Z_K(\Lambda)$. Within $\mathcal{Z}_K(\Lambda)$, we have a certain open `non-degenerate' locus $\leftcirc \mathcal{Z}_K(\Lambda)$, whose complex fiber is the subspace on which $\Lambda$ spans a saturated submodule of the $\Int$-Hodge structure $\bm{V}$. Special cases of this construction include classical Noether-Lefschetz loci on the moduli of polarized K3 surfaces, where $\leftcirc \mathcal{Z}_K(\Lambda)$ parameterizes isometric embeddings of $\Lambda$ as a direct summand in the primitive Picard lattice.

In any case, we are now ready for the main result:
\begin{introthm}
	\label{introthm:main}
Let $p$ be a prime where $V_{\Int}$ is self-dual, and suppose that $\mathrm{rank}(\Lambda)\leq \frac{n-4}{2}$. Then:
\begin{enumerate}
		\item $\leftcirc\mathcal{Z}_K(\Lambda)_{\Int_{(p)}}$ is normal, flat, and equidimensional of dimension $n- \mathrm{rank}(\Lambda)+1 = \dim \Ss_K - \mathrm{rank}(\Lambda)$.
	\item The special fiber $\leftcirc\mathcal{Z}_K(\Lambda)_{\Field_p}$ is geometrically normal and equidimensional of dimension $n- \mathrm{rank}(\Lambda)$.
	\item The natural map
	\[
 \pi_{0}\bigl(\leftcirc \mathcal{Z}_K(\Lambda)_{\overline{\Field}_p}\bigr)\xrightarrow{\simeq}\pi_{0}\bigl(\leftcirc Z_K(\Lambda)_{\overline{\Rat}}\bigr).
\]
is a bijection.
\end{enumerate}
\end{introthm}
The first two assertions are already explained in~\cites{Howard2020-fa,HMP:mod_codim}. It is the last statement that is the focus of this paper.

The Hodge theoretic analogue of the essential idea of proof---which is by induction on the $p$-adic valuation of the discriminant of $\Lambda$---is easy to explain. Let's do so in the context of Theorem~\ref{introthm:k3s}. Suppose that $(X,\xi)$ is a primitively polarized K3 surface over $\Comp$ of degree $2d$. Then the Betti cohomology $H^2(X,\Int)$ is a pure Hodge structure of weight $2$, and the Poincar\'e pairing endows it with the structure of a quadratic space over $\Int$ that is isometric to $U = \bm{H}^{\oplus 3} \oplus E_8^{\oplus 2}$. Here, $\bm{H}$ is the hyperbolic plane, and $E_8$ is the root lattice corresponding to its eponymous Dynkin type. 

We will distinguish one hyperbolic plane $\bm{H}\subset U$, and choose a hyperbolic basis $e,f$ for it satisfying $e^2 = f^2 = 0, [e,f] = 1$. Let $U' = \bm{H}^\perp\subset U$ be its orthogonal complement, so that we can write $U$ as the orthogonal direct sum
\[
 U = \bm{H}\perp U'.
\]

We can choose the isometry 
\[
H^2(X,\Int)\xrightarrow{\simeq} U
\]
so that the Chern class of $\xi$ maps to the element $e+df\in \bm{H}$. Within $H^2(X,\Rat)$ we have a lattice corresponding on the right hand side to the subspace
\[
\langle p^{-1}e,pf\rangle \perp U'\subset U_{\Rat}.
\]
The basic point is that this lattice corresponds to an `isogenous' K3 surface $X'$ equipped with an isometry
\[
H^2(X',\Int)\xrightarrow{\simeq}\langle pe,p^{-1}f\rangle \perp U',
\]
and that it admits a canonical primitive polarization $\xi'$ of degree $2p^2d$, whose Chern class maps under the above isometry to the element $pe+p^2\cdot (p^{-1}f)$. 

If one does this carefully in families, one finds essentially a Hecke correspondence mapping the moduli of primitively polarized K3 surfaces of degree $2p^2d$ to that of degree $2d$, and this allows the direct comparison between their connected components.

Inspired by~\cite{De_Jong1993-al}, I show that a similar argument can be made in the special fiber if we work over the ordinary locus, and this gives us the proof of Theorem~\ref{introthm:main}. One can then easily deduce Theorem~\ref{introthm:k3s} as a special case where $\mathrm{rank}(\Lambda) = 1$ using the results of~\cite{mp:tatek3}. Here's the line of reasoning for the former theorem: 
\begin{itemize}
	\item The stack $\leftcirc \mathcal{Z}_K(\Lambda)_{\Int_{(p)}}$ is a disjoint union of open subspaces $\leftcirc\Ss_{K^\flat_g,(p)}$ of integral models $\Ss_{K^\flat_g,(p)}$ of GSpin Shimura varieties associated with certain direct summands $V^\flat_{g,\Int}\subset V_{\Int}$. When $\Lambda$ is \emph{maximal} at $p$, one can delete the qualifier `of open subspaces', and the compactifications from~\cite{mp:toroidal} now yield (3) of Theorem~\ref{introthm:main}. 

	\item When $\Lambda$ is contained in a positive definite lattice $\tilde{\Lambda}$, with $[\tilde{\Lambda}:\Lambda] = p$, by restricting to the ordinary loci in the mod-$p$ fiber, I use ordinary Hecke correpondences, combined with the Hodge-theoretic idea explained above, to show that there is a map 
  \[
	\pi:\leftcirc\mathcal{Z}^{\ord}_{K}(\Lambda)_{\Field_p}\to \leftcirc\mathcal{Z}^{\ord}_K(\tilde{\Lambda})_{\Field_p}
  \]

  \item I then use Serre-Tate deformation theory, a key argument of Chai-Hida for the irreducibility of monodromy of the ordinary Igusa tower, and the inductive hypothesis to show that $\pi$ induces a bijection on geometric connected components. 

  \item A comparison with the analogous construction in the generic fiber using classical Hecke correspondences now finishes the job.
\end{itemize}

Here is the organization of this paper:
\begin{itemize}
	\item The first two sections, which deal with ordinary loci in integral models of Shimura varieties and ordinary Hecke correspondences, are covering `well-known' ground, and are included for a lack of convenient reference that works in the generality that I need.

	\item Section~\ref{sec:monodromy} reviews the use of a slick argument of Chai and Hida for the control of the monodromy of Igusa towers, and applies it to the ordinary case we care about.

	\item In the first three sections, I work in the generality of Shimura varieties of Hodge type at primes with non-empty ordinary loci. In Section~\ref{sec:main}, I specialize to the case of GSpin Shimura varieties and---after some prerequisites on quadratic lattices and the group schemes associated with them covered in Section~\ref{sec:group_schemes}---apply the results from the three initial sections to prove the main results.
\end{itemize}

\section*{Acknowledgments}
This work was partially supported by NSF Postdoctoral Research Fellowship DMS-1204165 and NSF grants DMS-1502142/1802169/2200804. It is essentially a reworking of an old unpublished preprint~\cite{Madapusi_Pera2018-ia}, whose results will be redistributed between this article and the upcoming paper~\cite{serre-tate}. I thank the referees of that earlier preprint for many helpful comments and apologize to them for only now getting around to implementing some of those suggestions.

\section*{Conventions}

\begin{enumerate}
	\item We will fix a prime number $p$ for the entirety of this paper.

	\item Given a set $X$ and any Grothendieck site, we will write $\underline{X}$ for the locally constant sheaf over the site associated with the constant presheaf sending every object to $X$.

	\item Given a topos, a pro-finite group $H$ presented as an inverse limit of finite quotients $H_n$, and an object $S$ in the topos, an $\underline{H}$-torsor $\pi:\mathcal{P}\to S$ is an inverse system 
    \[
	\{\pi_n:\mathcal{P}_n\to S\}_{n\in\Int_{\geq 1}},
	\]
	 where $\pi_n$ is a torsor under $\underline{H_n)}$. In particular, this applies to the situation where $D$ is a smooth group scheme over $\Int_p$  with $H = D(\Int_p)$  and $H_n = D(\Int/p^n\Int)$.

	\item Given an $H$-torsor $\mathcal{P}$ as above and a finite $H_n$-set $Y$, we set
\[
\mathcal{P}_{Y} = \mathcal{P}\times^{H} Y \defn (\mathcal{P}_n\times Y)/H_n,
\]
where $H_n$ acts diagonally on the product. If $Y$ is a pro-finite set equipped with a continuous action of $H$, we can similarly define a contraction product $\mathcal{P}\times^{H} Y$ via an inverse limit of the construction applied at finite levels.

	\item For any finite set of primes $T$, we will set
	\[
     \Adele_f^T = \prod'_{\ell\notin T}\Rat_{\ell},
	\]
	the restricted product over all completions of $\Rat$ at finite places not in $T$. If $T = \{\ell\}$ is a singleton, we will write $\Adele_f^\ell$ instead.

	\item For any local or global field $F$ in characteristic $0$, we will write $\Reg{F}$ for its ring of integers.

	\item Suppose that $X$ is a scheme or stack over a localization $\Rg$ of $\Reg{F}$ in which $p$ is not invertible. For any place $v\vert p\nmid N$ of $F$, we will write $X_{(v)}$ (resp. $X_v$) for the base-change of $X$ over the localization $\Reg{F,(v)}$ (resp. the completion $\Reg{F,v}$) at $v$. For $M\in \Rg$, we will write $X[M^{-1}]$ for the base-change over $\Reg[M^{-1}]$.

   \item
  Suppose that $R$ is a ring and suppose that $\mathbf{C}$ is an $R$-linear tensor category that is a faithful tensor sub-category of $\operatorname{Mod}_R$, the category of $R$-modules. Suppose in addition that $\mathbf{C}$ is closed under taking duals, symmetric and exterior powers in $\operatorname{Mod}_R$. Then, for any object $D\in\Obj(\mathbf{C})$, we will denote by $D^\otimes$ the direct sum of the tensor, symmetric and exterior powers of $D$ and its dual.

  \item All formal schemes over $\Int_p$ will be $p$-adic: That is, they are functors on $p$-nilpotent rings.

  \item Given a connected reductive group $G$ over a field $F$,  we write $G^{\mathrm{der}}\subset G$ for its derived subgroup, $\rho_G:G^{\mathrm{\mathrm{sc}}}\to G$ for the simply connected cover of its derived group and $G^{\ad}$ for its adjoint quotient.

    \item Fix an algebraic closure $\bar{F}$ for $F$. For any torus $T$ over $F$, we set
\[
X_*(T)=\Hom(\Gmh{\bar{F}},T_{\bar{F}})\;;\; X^*(T)=\Hom(T_{\bar{F}},\Gmh{\bar{F}})
\] for the cocharacter and character groups of $T$, respectively.

\item For a perfect field $\kappa$, we will write $W(\kappa)$ for its ring of Witt vectors and denote by $\sigma:W(\kappa)\to W(\kappa)$ the canonical lift of Frobenius.

\item Given a product $X_1\times \cdots \times X_k$ in a category with finite products, we will write $\mathrm{pr}_i$ for the projection onto the $i$-th factor.
\end{enumerate}

\section{Ordinary loci of Shimura varieties}\label{sec:ord}

We will fix a Shimura datum $(G,X)$. The purpose of this section is to recall and situate some results of Noot~\cite{noot:ordinary} regarding the structure of the ordinary loci of integral models of Shimura varieties associated with this datum

\subsection{Background on Shimura varieties of Hodge type}
\label{subsec:background_on_shimura_varieties_of_hodge_type}

\subsubsection{}
Given $x\in X$, we have the associated homomorphism of $\Real$-groups:
\[
 h_x:\mathbb{S}=\Res_{\Comp/\Real}\Gmh{\Real}\to G_{\Real}.
\]
We also have the associated (minuscule) cocharacter:
\[
 \mu_x:\Gmh{\Comp}\xrightarrow{z\mapsto(z,1)}\Gmh{\Comp}\times\Gmh{\Comp}\xrightarrow{\simeq}\bb{S}_{\Comp}\xrightarrow{h_x}G_{\Comp}.
\]

The $G(\Real)$-conjugacy class of $h_x$, and hence the $G(\Comp)$-conjugacy class $\{\mu_X\}_{\infty}$ of $\mu_x$, is independent of the choice of $x$. Let $E\subset\Comp$ be the reflex field for $(G,X)$: This is the field of definition of $\{\mu_X\}_{\infty}$, and is a finite extension of $\Rat$. This gives us a geometric conjugacy class $[\mu]\in (\Hom(\Gm,G)/G)(E)$. Given a place $v\vert p$ of $E$ we obtain in turn a geometric conjugacy class $[\mu_v]$ defined over $E_v$

\begin{construction}
	\label{const:integral_models}
	 We will assume for the rest of the section that the Shimura datum is of Hodge type, so that it is equipped with an embedding 
\[
(G,X)\into (\GSp(H),\on{S}^{\pm}(H)) 
\]
into a Siegel Shimura datum.\footnote{One can relax this to asking for $(G,X)$ to be merely of abelian type, but I will restrict to the Hodge type case in this paper.} We fix a $\Int$-lattice $H_{\Int}\subset H$ on which the symplectic form is $\Int$-valued, and we take $K_p\subset G(\Rat_p)$ to be the stabilizer of $H_p = \Int_p\otimes_{\Int} H_{\Int}$. Let $K\subset G(\Adele_f)$ be a compact open subgroup of the form $K_pK^p$ for $K^p\subset G(\Adele_f^p)$ stabilizing the prime-to-$p$ lattice $\widehat{\Int}^p\otimes_{\Int}H_{\Int}$, and let $K^\sharp\subset \GSp(H)(\Adele_f)$ be the stabilizer of $\widehat{\Int}\otimes_{\Int}H_{\Int}$. Then we obtain a finite unramified map of Shimura varieties (or rather stacks)
\[
\Sh_K \defn \Sh_K(G,X)\to E\otimes_{\Rat}\Sh_{K^\sharp}\defn \Sh_{K^\sharp}(\GSp(H),\on{S}^{\pm}(H)).
\]
The Shimura variety $\Sh_{K^\sharp}$ has an integral model $\Ss_{K^\sharp}$ over $\Spec \Int$ as a certain moduli stack of polarized abelian varieties. Using the symplectic representation $H$ and the lattice $H_{\Int}$, we obtain a normal integral model $\Ss_K$ for $\Sh_K$ over $\Reg{E}$ by taking the normalization of $\Reg{E}\otimes_{\Int}\Ss_{K^\sharp}$ in $\Sh_K$; see for instance~\cite[(1.3.5)]{KMPS}. 
\end{construction}

\subsubsection{}\label{subsubsec:tensors}
Suppose that $G_{\Int_p}$ is a smooth group scheme over $\Int_p$ with connected special fiber and with generic fiber $G_{\Rat_p}$ such that $K_p = G_{\Int_p}(\Int_p)$. Suppose also that we have a collection of tensors $\{s_\alpha\}\subset H^\otimes_{\Int_p}$ such that their pointwise stabilizer in $\GL(H_{\Int_p})$ is $G_{\Int_p}$.
  
\begin{lemma}\label{lem:Hp_Gstruc}
The $p$-adic local system $\bm{H}_p = T_p(\mathcal{A})^\vee$ over $\Ss_K[1/p]$ is equipped with a canonical $G_{\Int_p}$-structure. More precisely, there exist canonical tensors $\{s_{\alpha,p}\}\subset H^0(\Ss_K[1/p],\bm{H}_p^\otimes)$ such that, at each geometric point $x:\Spec F\to \Ss_K[1/p]$, there is an isomorphism $H_{\Int_p}\xrightarrow{\simeq}\bm{H}_{p,x}$ carrying $\{s_{\alpha}\}$ to $\{s_{\alpha,p,x}\}$.
\end{lemma}
\begin{proof}
See the discussion in~\cite[(1.3.4)]{KMPS}.
\end{proof}

\begin{remark}
	\label{rem:de_rham_tensors}
The de Rham realization $\bm{H}_{\dR} = H^1_{\dR}(\mathcal{A}/\Ss_K)$, when restricted to the generic fiber $\Sh_K$, is also equipped with a canonical $G$-structure as a filtered vector bundle with integrable connection. That is, $\bm{H}_{\dR}\vert_{\Sh_K}$ is equipped with a two step Hodge filtration $\Fil^\bullet_{\mathrm{Hdg}}\bm{H}_{\dR}\vert_{\Sh_K}$ concentrated in degrees $0$ and $1$ and a Gauss-Manin connection, and we have tensors $\{s_{\alpha,\dR}\}\subset \Fil^0(\bm{H}^\otimes_{\dR}\vert_{\Sh_K})$ that are parallel for this connection. Moreover, for any point $x:\Spec L \to \Sh_K$ with $L$ algebraically closed, there exists an isomorphism $L\otimes_{\Rat}H\xrightarrow{\simeq}\bm{H}_{\dR,x}$ carrying $\{1\otimes s_\alpha\}$ to $\{s_{\alpha,\dR,x}\}$ such that the Hodge filtration pulls back to a filtration on $L\otimes_{\Rat}H$ that is split by a cocharacter in the conjugacy class $[\mu^{-1}]$.
\end{remark}

\begin{lemma}\label{lem:GIsoc_struc}
Fix a place $v\vert p$ of $E$. Suppose that we have a perfect field $\kappa$ of characteristic $p$ and a point $x_0\in \Ss_{K,v}(k)$. Associated with this is the abelian variety $\mathcal{A}_{x_0}$, as well as the $p$-divisible group $\mathcal{G}_{x_0} = \mathcal{A}_{x_0}[p^\infty]$. Then $\mathcal{G}_{x_0}$ is equipped with a canonical $G$-structure compatible with the $G_{\Int_p}$-structure over $\Ss_K[1/p]$. More precisely, the following holds: 
\begin{enumerate}
	\item If $\bm{H}_{\cris,x_0} = \Dieu(\mathcal{A}_{x_0})(W(\kappa))$ is the contravariant Dieudonn\'e module associated with $\mathcal{A}_{x_0}$, then there exist Frobenius invariant tensors
	\[
		 \{s_{\alpha,\cris,x_0}\}\subset \bm{H}_{\cris,x_0}^\otimes
	\]
	and an isomorphism $W(\kappa)[1/p]\otimes_{\Rat}H\xrightarrow{\simeq}\bm{H}_{\cris,x_0}$ carrying $\{1\otimes s_{\alpha}\}$ to $\{s_{\alpha,\cris,x_0}\}$.

	\item If $L/W(\kappa)[1/p]$ is a $p$-adically complete discrete valuation field with residue field $k$ and $x\in \Ss_K(\Reg{L})$ is a lift of $x_0$ with image $\overline{x}\in \Sh_K(\overline{L})$ for an algebraic closure $\overline{L}$ of $L$, then the $p$-adic comparison isomorphism
	\[
		B_{\cris}\otimes_{\Int_p}\bm{H}_{p,\overline{x}} \xrightarrow{\simeq}B_{\cris}\otimes_{W(\kappa)}\bm{H}_{\cris,x_0}
	\]
	carries $\{s_{\alpha,p,\overline{x}}\}$ to $\{s_{\alpha,\cris,x_0}\}$.
\end{enumerate}
\end{lemma}
\begin{proof}
This is shown in~\cite[Prop. 1.3.7]{KMPS}.
\end{proof}

\begin{remark}
\label{rem:bmu_adm}
In the situation of assertion (1) of Lemma~\ref{lem:GIsoc_struc}, via a choice of isomorphism $W(\kappa)[1/p]\otimes_{\Int_p}H_{\Int_p}\xrightarrow{\simeq}\bm{H}_{\cris,x_0}$, the $\sigma$-semilinear isomorphism $\bm{H}_{\cris,x_0}\xrightarrow{\simeq}\bm{H}_{\cris,x_0}$ underlying its structure of an $F$-isocrystal corresponds to an element $b_{x_0}\in G(W(\kappa)[1/p])$ whose $\sigma$-conjugacy class $[b_{x_0}]\in B(G_{\Rat_p})$ is independent of the choice of the isomorphism. Moreover, this class is in fact $[\mu_v^{-1}]$-admissible; see~\cite[Lemma 1.3.9]{KMPS}.
\end{remark}

\begin{remark}
\label{rem:hodge_cocharacter}
In the situation of assertion (2) of Lemma~\ref{lem:GIsoc_struc}, we also have the Berthelot-Ogus comparison isomorphism
\[
	L\otimes_{W(\kappa)}\bm{H}_{\cris,x_0}\xrightarrow{\simeq}\bm{H}_{\dR,x}[p^{-1}]
\]
This carries $\{1\otimes s_{\alpha,\cris,x_0}\}$ to $\{s_{\alpha,\dR,x[p^{-1}]}\}$, and is compatible with the combination of the crystalline comparison isomorphism with the de Rham comparison isomorphism
\[
B_{\dR}\otimes_{\Int_p}\bm{H}_{p,\overline{x}} \xrightarrow{\simeq}B_{\dR}\otimes_{\Reg{L}}\bm{H}_{\dR,x_0}.
\]
In particular, the de Rham comparison isomorphism carries $\{s_{\alpha,p,\overline{x}}\}$ to $\{s_{\alpha,\dR,x}\}\subset \bm{H}_{\dR,x}^\otimes[1/p]$. Further, for any isomorphism $L\otimes_{\Rat}H\xrightarrow{\simeq}\bm{H}_{\dR,x}$ carrying $\{1\otimes s_{\alpha}\}$ to $\{s_{\alpha,\dR,x}\}$, the Hodge filtration $\Fil^1_{\mathrm{Hdg}}\bm{H}_{\dR,x}$ pulls back to a filtration on $L\otimes_{\Rat}H$ that is split by a cocharacter in the geometric conjugacy class $[\mu_v^{-1}]$.
\end{remark}

\begin{definition}
\label{defn:G_structure}
In the situation of Lemma~\ref{lem:GIsoc_struc}, we will say that $\mathcal{G}_{x_0}$ has a \defnword{$G_{\Int_p}$-structure} if we have $\{s_{\alpha,\cris,x_0}\}\subset \bm{H}_{\cris,x_0}^\otimes$ and if there exists an isomorphism $W(\kappa)\otimes_{\Int_p}H_{\Int_p}\xrightarrow{\simeq}\bm{H}_{\cris,x_0}$ carrying $\{1\otimes s_{\alpha}\}$ to $\{s_{\alpha,\cris,x_0}\}$.
\end{definition}

\begin{remark}
	\label{rem:hyperspecial}
Suppose that $G_{\Int_p}$ is reductive, so that $K_p$ is hyperspecial. Then, for any place $v\vert p$ of $E$, $\Ss_{K,(v)}$ is the integral canonical model for $\Sh_K$ over $\Reg{E,(v)}$ constructed in~\cites{kis:abelian,Kim2016-fb}, and is in particular independent of the choice of symplectic representation $H$ and the lattice $H_{\Int}$. Moreover, in this case, for all $x_0\in \Ss_{K,(v)}(k)$ as above, $\mathcal{G}_{x_0}$ is equipped with a $G_{\Int_p}$-structure; see~\cite[Corollary (1.4.3)]{kis:abelian} and~\cite[Theorem 2.5]{Kim2016-fb}.
\end{remark}

\subsection{Deformation theory over the ordinary locus}

Fix a place $v\vert p$ of $E$ such that $\Rat_p=E_v$, so that $[\mu_v]$ is defined over $\Rat_p$. In this subsection, we will look at the deformation rings of $\Ss_{K,(v)}$ at its ordinary points, under some assumptions.

\begin{assumption}
\label{assump:ordinary_cochar}
The conjugacy class $[\mu_v]$ admits a representative $\mu_v:\Gmh{\Int_p}\to G_{\Int_p}$ whose centralizer $M_{\Int_p}\subset G_{\Int_p}$ is smooth with connected special fiber.
\end{assumption}

\begin{remark}
Under this assumption, we have a decomposition $H_{\Int_p} = H^1_{\Int_p}\oplus H^0_{\Int_p}$ where $H^i_{\Int_p}$ is the eigenspace on which $\mu_v(z)$ acts via $z^{-i}$. Moreover, for all $\alpha$, we have $\mu_v(z)\dot s_{\alpha} = s_{\alpha}$.
\end{remark}

\begin{remark}
	\label{rem:hyperspecial_mu_integral}
If $G_{\Int_p}$ is reductive, then a representative $\mu_v$ always exists. To see this, one chooses a maximal torus $T\subset G_{\Int_p}$ contained in a Borel subgroup of $G_{\Int_p}$. There is a unique representative $\mu_v$ of the conjugacy class $[\mu_v]$ factoring through $T$ and dominant with respect to the choice of Borel, and this does the job for us. The centralizer of $\mu_v$ is a Levi subgroup of $G_{\Int_p}$ and is thus reductive with connected fibers.
\end{remark}

\begin{definition}
Given an algebraically closed field $\kappa$ in characteristic $p$, we will say that a point $x_0\in \Ss_{K,v}(\kappa)$ is \defnword{$\mu_v$-ordinary} or simply \defnword{ordinary} if $\mathcal{G}_{x_0}$ admits a $G_{\Int_p}$-structure, and if there exists an isomorphism $W(\kappa)\otimes_{\Int_p}H_{\Int_p}\xrightarrow{\simeq}\bm{H}_{\cris,x_0}$ with the following properties:
\begin{enumerate}
	\item It carries $\{1\otimes s_{\alpha}\}$ to $\{s_{\alpha,\cris,x_0}\}$.
	\item The $\sigma$-semilinear endomorphism of $\bm{H}_{\cris,x_0}$ arising from its structure of a Dieudonn\'e module conjugates under the isomorphism to the endomorphism $1\otimes\mu_v(p)^{-1}$ of $W(\kappa)\otimes_{\Int_p}H_p$. In other words, it acts as $p^i$ on $W(\kappa)\otimes_{\Int_p}H_{\Int_p}^i$, for $i=0,1$.
\end{enumerate}
\end{definition}

\begin{remark}
\label{rem:our_ordinary_is_ordinary}
	If $x_0$ is an ordinary point, the abelian variety $\mathcal{A}_{x_0}$ is ordinary in the usual sense: We have $\mathcal{G}_{x_0} \simeq \mathcal{G}_{x_0}^{\et}\times \mathcal{G}_{x_0}^{\mult}$, where $\mathcal{G}_{x_0}^{\et}$ (resp. $\mathcal{G}_{x_0}^{\mult}$) is an \'etale (resp. a multiplicative) $p$-divisible group. Under the isomorphism $W(\kappa)\otimes_{\Int_p}H_{\Int_p}\xrightarrow{\simeq}\bm{H}_{\cris,x_0}$ as in the definition, $W(\kappa)\otimes_{\Int_p}H^0_{\Int_p}$ (resp. $W(\kappa)\otimes_{\Int_p}H^1_{\Int_p}$) maps isomorphically onto the Dieudonn\'e module $\bm{H}_{\cris,x_0}^{\et}$ for $\mathcal{G}_{x_0}^{\et}$ (resp. $\bm{H}_{\cris,x_0}^{\mult}$ for $\mathcal{G}_{x_0}^{\mult}$).
\end{remark}

\begin{remark}
	\label{rem:unramified_ordinary}
If $G_{\Int_p}$ is reductive, then the converse of Remark~\ref{rem:our_ordinary_is_ordinary} holds: If $\mathcal{A}_{x_0}$ is ordinary in the usual sense, then $x_0$ is ordinary in our sense here. Indeed, by Remark~\ref{rem:hyperspecial}, $\mathcal{G}_{x_0}$ is equipped with a $G_{\Int_p}$-structure. Moreover, there exists a choice of isomorphism witnessing the $G_{\Int_p}$-structure such that the element $b_{x_0}$ from Remark~\ref{rem:bmu_adm} is of the form $\nu(p)$ for some cocharacter $\nu:\Gmh{W(\kappa)}\to G_{W(\kappa)}$ acting with weights $0,1$ on $W(\kappa)\otimes_{\Int_p}H_{\Int_p}$. We can further assume that $\nu$ factors through $T_{W(\kappa)}$, where $T\subset G_{\Int_p}$ is a maximal torus as in Remark~\ref{rem:hyperspecial_mu_integral}, and that it is dominant with respect to the choice of Borel made there. Now the condition that the $\sigma$-conjugacy class of $\nu(p)$ is $[\mu_v^{-1}]$-admissible forces the equality $\nu = \mu_v^{-1}$.
\end{remark}

\begin{notation}
	Given an algebraically closed field $\kappa$ in characteristic $p$, write $\mathrm{Art}_{W(\kappa)}$ for the category of Artin local $W(\kappa)$-algebras with residue field $\kappa$. Given $C\in \mathrm{Art}_{W(\kappa)}$, write $\mx_C$ for its maximal ideal. Note that $1+\mx_C$ as a group under multiplication is $p^n$-torsion for $n$ sufficiently large and so can be viewed as a $\Int_p$-module.
\end{notation}

\begin{definition}
Given a $\Int_p$-module $M$, the associated \defnword{formal torus with character group $M$} is the functor $\widehat{T}_M$ on $\mathrm{Art}_{W(\kappa)}$ given by
\[
	\widehat{T}_M(C) = \Hom_{\Int_p}(M,1+\mx_C).
\]
We will call $M^\vee = \Hom_{\Int_p}(M,\Int_p)$ the \defnword{cocharacter group} for $\widehat{T}_M$.
\end{definition}

\begin{notation}
	Set 
\[
\mathcal{G}_0^{\et} = \underline{\Hom}(H_{\Int_p}^0,\Rat_p/\Int_p)\;;\; \mathcal{G}_0^{\mult} = \underline{\Hom}(H_{\Int_p}^1,\mu_{p^\infty}).
\]
These are $p$-divisible groups over $\Int_p$. 
\end{notation}

\begin{remark}
	\label{rem:classical_serre-tate}
Set 
\[
\mathcal{G}_0^{\et} = \underline{\Hom}(H_{\Int_p}^0,\Rat_p/\Int_p)\;;\; \mathcal{G}_0^{\mult} = \underline{\Hom}(H_{\Int_p}^1,\mu_{p^\infty}).
\]
These are $p$-divisible groups over $\Int_p$. 

Suppose that we have an ordinary point $x_0\in \Ss_{K,v}(\kappa)$. Fix a choice of isomorphism witnessing the ordinariness of $x_0$: This gives rise via Dieudonn\'e theory to isomorphisms of $p$-divisible groups
\[
\kappa\otimes_{\Int_p}\mathcal{G}_0^{\et}\xrightarrow{\simeq} \mathcal{G}_{x_0}^{\et}\;;\; \kappa\otimes_{\Int_p}\mathcal{G}_0^{\mult}\xrightarrow{\simeq} \mathcal{G}_{x_0}^{\mult}.
\]
By classical Serre-Tate ordinary theory~\cite{katz:serre-tate}, the deformation functor $\mathrm{Def}_{\mathcal{G}_{x_0}}$ on $\mathrm{Art}_{W(\kappa)}$ is now isomorphic to the functor
\begin{align*}
\widehat{\Ext}(\mathcal{G}_0^{\et},\mathcal{G}_0^{\mult}): C&\mapsto \widehat{\Ext}^1(C\otimes_{\Int_p}\mathcal{G}_0^{\et},C\otimes_{\Int_p}\mathcal{G}_0^{\mult})\\
&\xrightarrow{\simeq}\Hom(H^1_{\Int_p},H^0_{\Int_p})\otimes_{\Int_p}\widehat{H}^1(\Spec C,\mu_{p^\infty})\\
&\xrightarrow{\simeq}\Hom(H^1_{\Int_p},H^0_{\Int_p})\otimes_{\Int_p}(1+ \mathfrak{m}_C).
\end{align*}
Here, $\widehat{H}^1(\Spec C,\mu_{p^\infty})$ is the set of isomorphism classes of extensions of $\mu_{p^\infty}$ by $\Rat_p/\Int_p$ as fppf sheaves over $\Spec C$ equipped with a trivialization over $\kappa$. The last isomorphism is obtained from Kummer theory. 

In particular, $\mathrm{Def}_{\mathcal{G}_{x_0}}$ is isomorphic to the formal torus $\widehat{T}$ over $W(\kappa)$ with cocharacter group $\Hom(H^1_{\Int_p},H^0_{\Int_p})$. This isomorphism is well-defined up to the action of $M_{\Int_p}(\Int_p)$ on the cocharacter group.
\end{remark}


\begin{remark}
	\label{rem:crystalline_realization_serre-tate}
Let  $R$ be the coordinate ring of $\widehat{T}$ and let $\mathcal{G}^{\mathrm{univ}}$ be the universal deformation of $\mathcal{G}_{x_0}$ over $R$. Dieudonn\'e theory over such formally smooth rings associates with $\mathcal{G}^{\mathrm{univ}}$ its Dieudonn\'e module $M = \Dieu(\mathcal{G}^{\mathrm{univ}})(R)$, which is equipped with a topologically locally nilpotent integrable connection $\nabla$ and a $\varphi$-semilinear map $F:M\to M$ that is horizontal for the connection. This can be described explicitly following de Jong~\cite[\S 4.3]{dejong:formal_rigid}: We have the tautological element
\[
	q^{\mathrm{univ}}\in \Hom(H_{\Int_p}^1,H_{\Int_p}^0)\otimes_{\Int_p}(1+\mx_R) 
\]
mapping to
\[
	\dlog(q^{\mathrm{univ}})\in \Hom(H_{\Int_p}^1,H_{\Int_p}^0)\otimes_{\Int_p}\hat{\Omega}^1_{R/W(\kappa)}\subset \End(H_{\Int_p})\otimes_{\Int_p}\hat{\Omega}^1_{R/W(\kappa)}
\]
We now have $M = R\otimes_{\Int_p}H_{\Int_p}$. The connection $\nabla$ is given by $d\otimes 1+\dlog(q^{\mathrm{univ}})$, and we have $F = \varphi\otimes \mu_v(p)^{-1}$.
\end{remark}

\begin{definition}
	The \defnword{canonical lift} of $\mathcal{G}_{x_0}$ is the lift $\mathcal{G}^{\mathrm{can}}_{x_0}$ over $W(\kappa)$ corresponding to the identity section of the formal torus $\widehat{T}$. The corresponding 
\end{definition}

\begin{definition}
	Let $\Lie U^-_{\mu_v}\subset \Lie G_{\Int_p}$ be the eigenspace on which the adjoint action of $\mu_p(z)$ is via $z$. We have
	\[
		\Lie U^-_{\mu_v} = \Lie G \cap \Hom(H^1_{\Int_p},H^0_{\Int_p})\subset \End(H_{\Int_p})
	\]
	where we are viewing $\Hom(H^1_{\Int_p},H^0_{\Int_p})$ as the space of endomorphisms of $H_{\Int_p}$ whose kernel contains $H^0_{\Int_p}$ and whose image is contained in $H^0_{\Int_p}$. Let $\widehat{T}_G$ be the formal torus with cocharacter group $\Lie U^-_{\mu_v}$: This is a sub-formal torus of $\widehat{T}$.
\end{definition}

\begin{remark}
	\label{rem:crystalline_tensors_families}
Let $R_G$ be the ring of functions of $\widehat{T}_G$: it is equipped with a canonical Frobenius lift $\varphi_G$ given by the $p$-power map. The map $f:R\to R_G$ corresponding to $\widehat{T}_G\hookrightarrow \widehat{T}$ is compatible with Frobenius lifts. Via base-change, we now see from Remark~\ref{rem:crystalline_realization_serre-tate} that the Dieudonn\'e module for the universal $p$-divisible group over $R_G$ is given by $M_G = R_G\otimes_{\Int_p}H_{\Int_p}$ with connection given by $\nabla_G = d\otimes 1 + \dlog(f(q_{\mathrm{univ}}))$ and $\varphi_G$-semilinear endomorphism $F_G = \varphi_G\otimes \mu_v(p)^{-1}$. Now, we have
\[
	f(q_{\mathrm{univ}})\in \Lie U_{\mu_v}^-\otimes_{\Int_p}(1+\mx_{R_G})\subset \Lie G\otimes_{\Int_p}(1+\mx_{R_G}).
\]
This shows that the tensors $s_{\alpha,R_G} = 1\otimes s_{\alpha,\cris,x_0}\in M_G^{\otimes}$ are parallel for the connection. They are also invariant under $F_G$.
\end{remark}

\begin{proposition}
	\label{prop:ordinary_deformations}
Let $x_0\in \Ss_{K,v}(\kappa)$ be ordinary, and let $\widehat{U}_{x_0}$ be the deformation space for $\Ss_{K,v}$ at $x_0$. If the image of $\widehat{U}_{x_0}$ in $\mathrm{Def}_{\mathcal{G}_{x_0}}\xrightarrow{\simeq}\widehat{T}$ contains the canonical lift, then $\widehat{U}_{x_0}$ is isomorphic to $\widehat{T}_G\subset \widehat{T}$.
\end{proposition}
\begin{proof}
First, note that by construction and the Serre-Tate theorem for deformations of abelian varieties, the map $\widehat{U}_{x_0}\to \widehat{T}$ is finite. Moreover, $\widehat{U}_{x_0} = \Spf R_{x_0}$ where $R_{x_0}$ is a complete Noetherian local $W(\kappa)$-algebra of dimension $\dim \Ss_K = \dim \widehat{T}_G$. Moreover, by a theorem of Noot~\cite[Theorem 3.7]{noot:ordinary}, there is a finite extension $L/W(\kappa)[1/p]$ such that each irreducible component of $\widehat{U}_{x_0,\Reg{L}} \defn \Spf(\Reg{L}\otimes_{W(\kappa)}R_{x_0})$ is isomorphic to the translate by a torsion point of a formal sub-torus of $\widehat{T}_{\Reg{L}}$.

Consider the logarithm map
\[
\ell:\widehat{T}(\Reg{L}) = \Hom(H_{\Int_p}^1,H_{\Int_p}^0)\otimes_{\Int_p}(1+\mx_L)\xrightarrow{1\otimes \log}\Hom(H_{\Int_p}^1,H_{\Int_p}^0)\otimes_{\Int_p} L.
\]
I claim that we have
\[
\ell(\widehat{U}_{x_0}(\Reg{L})) \subset \Lie U^-_{\mu_v} \otimes_{\Int_p} L.
\] 
For this, we will need the following interpretation of the map $\ell$: Given a lift $x\in \widehat{T}(\Reg{L})$, we obtain a $p$-divisible group $\mathcal{G}_x$. As in Remark~\ref{rem:hodge_cocharacter}, we have a canonical isomorphism $L\otimes_{W(\kappa)}\bm{H}_{\cris,x_0}\xrightarrow{\simeq}\bm{H}_{\dR,x}[1/p]$, and we therefore have an isomorphism
\[
L\otimes_{\Int_p}H_{\Int_p}\xrightarrow{\simeq}\bm{H}_{\dR,x}[1/p]
\]
by our choice of fixed isomorphism. The Hodge filtration arising from this identification is of the form
\[
\Fil^1_x(L\otimes_{\Int_p}H_{\Int_p}) = \exp(N_x)(L\otimes_{\Int_p}H_{\Int_p}^1),
\]
for some $N_x\in L\otimes_{\Int_p}\Hom(H_{\Int_p}^1,H_{\Int_p}^0)$. It follows from a computation of Katz~\cite[A.3]{Deligne1981-xy} that, at least up to sign, we have
\[
N_x = \ell(x)\in \Hom(H_{\Int_p}^1,H_{\Int_p}^0)\otimes_{\Int_p}L.
\]

Now, by Remark~\ref{rem:hodge_cocharacter}, if $x\in \widehat{U}_{x_0}(\Reg{L})$, then $\Fil^1_x(L\otimes_{\Int_p}H_{\Int_p})$ is split by a cocharacter conjugate to $\mu_v^{-1}$. In particular, the parabolic subgroup $P_x\subset G_L$ stabilizing $\Fil^1_x(L\otimes_{\Int_p}H_{\Int_p})$ is $G(L)$-conjugate to the subgroup $P_L\subset G_L$ stabilizing $L\otimes_{\Int_p}H^1_{\Int_p}$. That is, within the set $\mathrm{Gr}(L)$ of subspaces of middle dimension in $H_L$, $\Fil^1_x(L\otimes_{\Int_p}H_{\Int_p})$ is in the image of $G(L)$. However, the map 
\[
	L\otimes_{\Int_p}\Hom(H^1_{\Int_p},H^0_{\Int_p})\xrightarrow{N\mapsto \exp(N)(L\otimes_{\Int_p}H^1_{\Int_p})}\mathrm{Gr}(L)
\]
is injective, with image isomorphic to the $L$-points of the open Bruhat cell, and the pre-image of the image of $G(L)$ is precisely $L\otimes_{\Int_p}\Lie U^-_{\mu_v}$. This proves the claim.

Now, it follows that the irreducible components of $\widehat{U}_{x_0,\Reg{L}}$ are all translates by torsion points of $\widehat{T}_{G,\Reg{L}}$. By our hypothesis, one of these components must contain the identity section, and must therefore be isomorphic to $\widehat{T}_{G,\Reg{L}}$. This shows that $\widehat{T}_G$ is in the iamge of $\widehat{U}_{x_0}$, and by dimension considerations we see that we must in fact have $\widehat{U}_{x_0}\xrightarrow{\simeq}\widehat{T}_G$.
\end{proof}

\begin{remark}
	\label{rem:hodge_tensors_integral}
Maintain the hypotheses of Proposition~\ref{prop:ordinary_deformations}. Let $\widehat{U}^{\an}_{x_0}$ be the rigid analytic space over $W(\kappa)[1/p]$ associated with $\widehat{U}_{x_0}$.  The proposition shows that this is isomorphic to an open unit ball of dimension $\dim \Sh_K$ contained in the rigid analytic space $\Sh_{K,W(\kappa)[1/p]}^{\an}$. Over the latter, the restriction of the vector bundle $\bm{H}_{\dR}$ is equipped with an integrable connection and parallel tensors $\{s_{\alpha,\dR}\}$. On the other hand, under the isomorphism $\widehat{U}_{x_0}\xrightarrow{\simeq}\widehat{T}_G$, the vector bundle with integrable connection $M_G$ over the latter constructed in Remark~\ref{rem:crystalline_tensors_families} pulls back to the restriction of $\bm{H}_{\dR}$, and the tensors $\{s_{\alpha,R_G}\}$ constructed in \emph{loc. cit.} pull back to horizontal tensors in $H^0(\widehat{U}^{\an}_{x_0},\bm{H}^\otimes_{\dR})$. These tensors agree with the restrictions of $\{s_{\alpha,\dR}\}$: Indeed, since both collections are parallel for the connection, it suffices to verify that they agree at one point. One can do this at the point corresponding to the canonical lift via Remark~\ref{rem:hodge_cocharacter}.
\end{remark}

\emph{The following assumption will be in force for the rest of this section.}
\begin{assumption}
\label{assump:canonical_lift}
For all $\kappa$ algebraically closed and all ordinary points $x_0\in \Ss_{K,v}(\kappa)$, the image of $\widehat{U}_{x_0}$ in $\mathrm{Def}_{\mathcal{G}_{x_0}}$ contains the canonical lift.
\end{assumption}

\begin{remark}
	\label{rem:hyperspecial_canonical_lift}
The above assumption---and hence the conclusion of Proposition~\ref{prop:ordinary_deformations}---is always valid when $G_{\Int_p}$ is reductive. See for instance~\cite[Theorem 6.5]{Shankar2021-ik}.
\end{remark}

\begin{corollary}
	\label{cor:ordinary_openness}
The ordinary points form an open substack $\Ss^{\ord}_{K,k(v)}$ of the special fiber $\Ss_{K,k(v)} = k(v)\otimes_{\Reg{E,v}}\Ss_{K,v}$.
\end{corollary}
\begin{proof}
	With the notation of Proposition~\ref{prop:ordinary_deformations} and its proof, it is enough to show that a geometric generic point of $\Spec R_{x_0}/(p)$ is ordinary in our sense. Write $\bm{H}_{\dR,R_{x_0}}$ for the finite free $R_{x_0}$-module associated with $\bm{H}_{\dR}^\otimes\vert_{\Spec R_{x_0}}$. Via the isomorphism $R_{x_0}\xrightarrow{\simeq}R_G$, $R_{x_0}$ is equipped with a Frobenius lift $\varphi$, and we have an isomorphism
	\[
		F:\varphi^*\bm{H}_{\dR,R_{x_0}}[p^{-1}]\xrightarrow{\simeq}\bm{H}_{\dR,R_{x_0}}[p^{-1}].
	\]

	By Remark~\ref{rem:hodge_tensors_integral}, the restriction of $\{s_{\alpha,\dR}\}$ over $\Spec R_{x_0}[p^{-1}]$ gives a collection 
  \[
	\{s_{\alpha,\dR,R_{x_0}}\}\subset \bm{H}^\otimes_{\dR,R_{x_0}}
  \]
	that is invariant under $F$, and moreover there is an isomorphism
	\begin{align}\label{eqn:trivialization_crystalline}
		R_{x_0}\otimes_{\Int_p}H_{\Int_p}\xrightarrow{\simeq}\bm{H}_{\dR,R_{x_0}}
	\end{align}
	carrying $\{1\otimes s_\alpha\}$ to $\{s_{\alpha,\dR,R_{x_0}}\}$, and is such that $F$ pulls back to the automorphism $1\otimes \mu_v(p)^{-1}$ of $R_{x_0}\otimes_{\Int_p}H_{\Int_p}$. 

	If $F$ is an algebraic closure of the fraction field of $R_{x_0}/(p)$, we can lift $R_{x_0}\to F$ uniquely to a map $R_{x_0}\to W(F)$ that respects Frobenius lifts, and the base-change over $W(F)$ of~\eqref{eqn:trivialization_crystalline} now witnesses the ordinariness of the corresponding $F$-valued point of $\Ss_{K,(v)}$.
\end{proof}

\begin{corollary}
	\label{cor:integral_parallel_tensors}
Let $\widehat{\Ss}^{\ord}_{K,v}$ be the formal completion of $\Ss_{K,(v)}$ along the (open) ordinary locus $\Ss^{\ord}_{K,k(v)}$, and let $\widehat{\Ss}^{\ord,\an}_{K,v}\subset \Sh^{\an}_{K,E_v}$ be the corresponding rigid analytic tube. Then the restriction of $\{s_{\alpha,\dR}\}$ over $\widehat{\Ss}^{\ord,\an}_{K,v}$ extends canonically to a collection of parallel tensors $\{s_{\alpha,\dR}\}\subset H^0(\widehat{\Ss}^{\ord}_{K,v},\Fil^0(\bm{H}^\otimes_{\dR}))$.
\end{corollary}
\begin{proof}
Immediate from Remark~\ref{rem:hodge_tensors_integral}.
\end{proof}

\subsection{The ordinary Igusa tower}
\label{subsec:the_ordinary_igusa_tower}

Here, we review the story of the Igusa tower over the ordinary locus.

\begin{remark}
	\label{rem:G0_G_structure}
	Set $\mathcal{G}_0 \defn \mathcal{G}_0^{\et}\times \mathcal{G}_0^{\mult}$, and note that it can be equipped with a canonical $G_{\Int_p}$-structure in the following sense: If we set $\bm{H}_0 = \Dieu(\mathcal{G}_0)(\Int_p)$, then there is an isomorphism
	\[
		\bm{H}_0 \xrightarrow{\simeq} H^0_{\Int_p}\oplus H^1_{\Int_p} = H_{\Int_p}
	\]
	well-defined up to the action of $M_{\Int_p}(\Int_p)$, and we can use this to obtain $F$-invariant tensors $\{s_{\alpha,0}\}\subset \bm{H}^\otimes_0$, which are carried to $\{s_{\alpha}\}$ under any choice of such isomorphism. 
\end{remark}

\begin{definition}
	For an algebraically closed field $\kappa$ in characteristic $p$, an automorphism of $\kappa\otimes_{\Int_p}\mathcal{G}_0$ is $G_{\Int_p}$\defnword{-structure preserving} if the induced automorphism of $W(\kappa)\otimes_{\Int_p}\bm{H}_0$ fixes the tensors $\{s_{\alpha,0}\}$. For a $p$-complete $\Int_p$-algebra $R$, an automorphism $\alpha$ of $R\otimes_{\Int_p}\mathcal{G}_0$ is $G_{\Int_p}$\defnword{-structure preserving} its restriction along any geometric point of $\Spf R$ is $G_{\Int_p}$-structure preserving.
\end{definition}

\begin{notation}
For a $p$-complete ring $R$, write $\Aut_G(R\otimes_{\Int_p}\mathcal{G}_0)$ for the group of $G_{\Int_p}$-structure preserving automorphisms.
\end{notation}

\begin{lemma}
	\label{lem:aut_g_loc_const}
Consider the formal functor 
\[
\underline{\Aut}_G(\mathcal{G}_0):\Spf R \mapsto \Aut_G(R\otimes_{\Int_p}\mathcal{G}_0).
\]
There is an isomorphism $\underline{\Aut}_G(\mathcal{G}_0)\xrightarrow{\simeq}\underline{M_{\Int_p}(\Int_p)}$ well-defined up to conjugation by an element of $M_{\Int_p}(\Int_p)$.
\end{lemma}
\begin{proof}
	This is immediate from the definitions and the fact that the automorphism sheaf $\underline{\Aut}(\mathcal{G}_0)$ is isomorphic up to conjugation by $M_{\Int_p}(\Int_p)$ to the locally constant pro-finite sheaf of groups $\underline{\GL(H^0_{\Int_p})}\times \underline{\GL(H^1_{\Int_p})}$. 
\end{proof}

\begin{remark}
	\label{rem:slope_filtration_ordinary}
The restriction of the $p$-divisible group $\mathcal{G} = \mathcal{A}[p^\infty]$ over $\widehat{\Ss}_{K,v}^{\ord}$ is canonically an extension
\[
	0\to \mathcal{G}^{\mult}\to \mathcal{G}\vert_{\widehat{\Ss^{\ord}_{K,v}}}\to \mathcal{G}^{\et}\to 0,
\]
where $\mathcal{G}^{\mult}$ is of multiplicative type and $\mathcal{G}^{\et}$ is \'etale. If $x_0\in \Ss^{\ord}_{K,k(v)}(\kappa)$ is an algebraically closed point, then we have a canonical splitting $\mathcal{G}_{x_0}\xrightarrow{\simeq}\mathcal{G}_{x_0}^{\et}\times \mathcal{G}_{x_0}^{\mult}$.
\end{remark}

\begin{definition}
For $\kappa$ algebraically closed, an isomorphism $\kappa\otimes_{\Int_p}\mathcal{G}_0\xrightarrow{\simeq}\mathcal{G}_{x_0}^{\et}\times \mathcal{G}_{x_0}^{\mult}\simeq \mathcal{G}_{x_0}$ is $G_{\Int_p}$\defnword{-structure preserving} if the associated isomorphism $W(\kappa)\otimes_{\Int_p}\bm{H}_0\xrightarrow{\simeq}\bm{H}_{\cris,x_0}$ carries $\{s_{\alpha,0}\}$ to $\{s_{\alpha,\cris,x_0}\}$
\end{definition}

\begin{definition}
We define $\widehat{\mathrm{Ig}}^{\ord}_{K,v}\to \widehat{\Ss}^{\ord}_{K,v}$ to be the formal scheme parameterizing, for each $x:\Spf R\to \widehat{\Ss}^{\ord}_{K,v}$, the set of isomorphisms $R\otimes_{\Int_p}\mathcal{G}_0 \xrightarrow{\simeq}\mathcal{G}^{\et}_x\times \mathcal{G}^{\mult}_x$ whose restriction over any algebraically closed point is $G_{\Int_p}$-structure preserving.
\end{definition}

\begin{proposition}
\label{prop:igusa_tower}
The map $\widehat{\mathrm{Ig}}^{\ord}_{K,v}\to \widehat{\Ss}^{\ord}_{K,v}$ is a torsor under the locally constant pro-finite sheaf of groups $\underline{\Aut}_G(\mathcal{G}_0)\simeq \underline{M_{\Int_p}(\Int_p)}$.
\end{proposition}
\begin{proof}
Given the natural action of $\underline{\Aut}_G(\mathcal{G}_0)$ on $\widehat{\mathrm{Ig}}^{\ord}_{K,v}$, it suffices to check the assertion over the complete local ring $R_{x_0}$ at any $\overline{\Field}_p$-point $x_0$. In fact, it suffices to show that $\widehat{\mathrm{Ig}}^{\ord}_{K,v}(R_{x_0})$ is non-empty, and this can be deduced from Proposition~\ref{prop:ordinary_deformations} and Remark~\ref{rem:crystalline_tensors_families}.
\end{proof}

\subsection{The generic fiber of the Igusa tower}
\label{subsec:the_generic_fiber_of_the_igusa_tower}

\begin{notation}
	Let $P^-_{\Int_p}\subset G_{\Int_p}$ be the subgroup stabilizing the subspace $H^0_{\Int_p}\subset H_{\Int_p}$: its generic fiber is a parabolic subgroup of $G_{\Rat_p}$. We have a natural quotient map $P^-_{\Int_p}\to M_{\Int_p}$ whose kernel is the commutative unipotent subgroup $U^-_{\Int_p}$with Lie algebra $\Lie U^-_{\mu_v}$. Write $I_{G,p}\to \Sh_K$ for the $G_{\Int_p}(\Int_p)$-torsor $\Sh_{K^p}$.
\end{notation}

\begin{construction}
\label{const:two_M_torsors}
		We have two $M_{\Int_p}(\Int_p)$-torsors over the analytic space  $\widehat{\Ss}^{\ord,\an}_{K,v}$.  First, we can take the analytification $\widehat{\mathrm{Ig}}^{\ord,\an}_{K,v}$ of the Igusa tower. The second $M_{\Int_p}(\Int_p)$-torsor is obtained as follows: The filtration on $\mathcal{G}\vert_{\widehat{\Ss}^{\ord}_{K,v}}$ from Remark~\ref{rem:slope_filtration_ordinary} gives a corresponding filtration on the dual Tate module 
\[
  T_p\left(\mathcal{G}\vert_{\widehat{\Ss}^{\ord,\an}_{K,v}}\right)^\vee\simeq \bm{H}_p\vert_{\widehat{\Ss}^{\ord,\an}_{K,v}}
 \]
 yielding a short exact sequence $0\to \bm{H}_p^{\et}\to \bm{H}_p\vert_{\widehat{\Ss}^{\ord,\an}_{K,v}}\to \bm{H}_p^{\mult}\to 0$ of pro-finite \'etale sheaves over $\widehat{\Ss}^{\ord,\an}_{K,v}$. This yields a reduction of structure group of $I_{G,v}\vert_{\widehat{\Ss}^{\ord,\an}_{K,v}}$ to a $P^-_{\Int_p}(\Int_p)$-torsor, which we denote by $I^{\ord,\an}_{P^-,v}$. Pushing out via the quotient map $P^-_{\Int_p}(\Int_p)\to M_{\Int_p}(\Int_p)$ now yields an $M_{\Int_p}(\Int_p)$-torsor $I^{\ord,\an}_{M,v}\to \widehat{\Ss}^{\ord,\an}_{K,v}$.
\end{construction}


\begin{remark}
	\label{rem:ig_generic_desc}
	By construction, $\widehat{\mathrm{Ig}}^{\ord,\an}_{K,v}$ is a subsheaf of the analytification $\widehat{\mathrm{Ig}}^{\mathrm{big},\an}_{v}$ of the sheaf $\widehat{\mathrm{Ig}}^{\mathrm{big}}_{v}$ over $\widehat{\Ss}_{K,v}^{\ord}$ parameterizing for $x:\Spf R\to \widehat{\Ss}_{K,v}^{\ord}$ pairs of isomorphisms $R\otimes_{\Int_p}\mathcal{G}_0^{\et}\xrightarrow{\simeq}\mathcal{G}^{\et}_x$ and $R\otimes_{\Int_p}\mathcal{G}_0^{\mult}\xrightarrow{\simeq}\mathcal{G}^{\mult}_x$. This analytification is simply the sheaf parameterizing pairs $(\alpha,\beta)$, where $\alpha$ (resp. $\beta$) is an isomorphism of pro-finite \'etale $\Int_p$-local systems $\underline{H}^0_{\Int_p}\xrightarrow{\simeq}\bm{H}_p^{\et}$ (resp. $\underline{H}^1_{\Int_p}(1)=\underline{H}^1_{\Int_p}\otimes_{\Int_p}\underline{\Int}_p(1)\xrightarrow{\simeq}\bm{H}_p^{\mult}$). Here, $\underline{\Int}_p(1) = T_p(\Gm)^\vee$ is the inverse cyclotomic tower. 

	Note that $\widehat{\mathrm{Ig}}^{\mathrm{big},\an}_{v}$ is a $\GL(H^0_{\Int_p})\times \GL(H^1_{\Int_p})$-torsor over $\widehat{\Ss}^{\ord,\an}_{K,v}$.
\end{remark}

\begin{remark}
	\label{rem:M_reduction_over_canonical_lift}
Suppose that we have an algebraically closed point $x_0\in \Ss^{\ord}_{K,k(v)}(\kappa)$ and that $x\in \Ss^{\ord}_{K,v}(W(\kappa))$ is its canonical lift. Write $x$ also for the associated $\Spec W(\kappa)[p^{-1}]$-point of $\Sh_{K,E_v}$. We have a canonical splitting $\bm{H}_{p,x}\xrightarrow{\simeq}\bm{H}_{p,x}^{\et}\oplus \bm{H}_{p,x}^{\mult}$, which gives a reduction of structure for $I_{G,v,x}$ to an $M_{\Int_p}(\Int_p)$-torsor $I^{\mathrm{can}}_{M,v,x}\subset I^{\mathrm{ord},\an}_{P,v,x}$ that maps isomorphically onto $I^{\mathrm{ord},\an}_{M,v,x}$. Explicitly, if $\overline{L}$ is an algebraic closure of $L = W(\kappa)[1/p]$ and $\overline{x}\in \Sh_K(\overline{L})$ is the associated geometric point, then $I^{\mathrm{can}}_{M,v,x}(\overline{L})$ is the set of $G_{\Int_p}$-structure preserving isomorphisms $H_{\Int_p}\xrightarrow{\simeq}\bm{H}_{p,\overline{x}}$ that carry $H^0_{\Int_p}$ onto $\bm{H}^{\et}_{p,\overline{x}}$ and $H^1_{\Int_p}$ onto $\bm{H}^{\mult}_{p,\overline{x}}$.
\end{remark}

\begin{remark}
	\label{rem:igusa_tower_over_canonical_lift}
 Similarly, one sees that $\widehat{\mathrm{Ig}}^{\ord,\an}_{K,v,x}(\overline{L})$ is also isomorphic to the same set, except that its $\Gal(\overline{L}/L)$-structure is twisted by having the Galois group act on $H^1_{\Int_p}$ via the inverse cyclotomic character. Indeed, a section of (the non-empty set) $\widehat{\mathrm{Ig}}^{\ord}_{K,v,x}(W(\kappa))$ is an isomorphism $W(\kappa)\otimes_{\Int_p}\mathcal{G}_0\xrightarrow{\simeq}\mathcal{G}_x$ that preserves $G_{\Int_p}$-structure over $\kappa$. Via the crystalline comparison isomorphism and assertion (2) of Lemma~\ref{lem:GIsoc_struc}, the associated isomorphism 
\[
 \underline{H}^0_{\Int_p}\oplus \underline{H}^1_{\Int_p}(1) = T_p(\mathcal{G}_0)^\vee\xrightarrow{\simeq}\bm{H}_{p,x}
 \]
 of $p$-adic local systems over $\Spec L$ also preserves $G_{\Int_p}$-structure. We now recover our desired assertion by restricting this isomorphism over $\Spec\overline{L}$.
\end{remark}

\begin{construction}
\label{const:igusa_tower_comparison}
Let $\Int_p^{\mathrm{cycl}}$ be the $\Int_p^\times$-torsor over $\Spec \Rat_p$ corresponding to the $p$-adic cyclotomic tower: this parameterizes isomorphisms of $p$-adic local systems $\underline{\Int}_p\xrightarrow{\simeq}\underline{\Int}_p(1)$. There is a canonical $M_{\Int_p}(\Int_p)\times \Int_p^\times$-equivariant map
\begin{align}\label{eqn:igusa_tower_comp_map}
	I^{\ord,\an}_{M,v}\times \Int_p^{\mathrm{cycl}}\to \widehat{\mathrm{Ig}}^{\mathrm{big},\an}_{v}
\end{align}
over $\widehat{\Ss}^{\ord,\an}_{K,v}$, where $\Int_p^\times$ acts trivially on the right hand side. This is constructed as follows: A section on the left over $x:\mathrm{Spa}(R,R^+)\to \widehat{\Ss}^{\ord,\an}_{K,v}$ yields isomorphisms
\[
	\alpha:\underline{H}_p^0\xrightarrow{\simeq}\bm{H}^{\et}_{p,x}\;;\; \beta':\underline{H}_p^1\xrightarrow{\simeq}\bm{H}^{\mult}_{p,x}\;;\; \zeta:\underline{\Int}_p\xrightarrow{\simeq}\underline{\Int}_p(1)
\]
of $p$-adic sheaves over $\mathrm{Spa}(R,R^+)$. We can now combine the last two to get an isomorphism $\beta'\otimes\zeta^{-1}:\underline{H}_p^1(1)\xrightarrow{\simeq}\bm{H}^{\mult}_{p,x}$. 
\end{construction}

\begin{proposition}
	\label{prop:igusa_tower_generic_fiber}
The map~\eqref{eqn:igusa_tower_comp_map} factors through an isomorphism of $M_{\Int_p}(\Int_p)$-torsors
\[
	I^{\ord,\an}_{M,v}\times^{\Int_p^\times} \Int_p^{\mathrm{cycl}}\xrightarrow{\simeq} \widehat{\mathrm{Ig}}^{\mathrm{ord},\an}_{K,v}
\]
where on the left we are using the action of $\Int_p^\times = \Gm(\Int_p)$ on $M_{\Int_p}(\Int_p)$ via the central cocharacter $\mu_v^{-1}:\Gmh{\Int_p}\to M_{\Int_p}$.
\end{proposition}
\begin{proof}
	First, note that, because of the centrality of $\mu_v^{-1}$, the quotient on the left inherits the structure of an $M_{\Int_p}(\Int_p)$-torsor. The fact that~\eqref{eqn:igusa_tower_comp_map} factors through this quotient is immediate from the construction, and amounts to the observation that $\mu_v(z)$ acts on $H^1_{\Int_p}$ via multiplication by $z^{-1}$. To see that it maps isomorphically onto the other $M_{\Int_p}(\Int_p)$-torsor $\widehat{\mathrm{Ig}}^{\mathrm{ord},\an}_{K,v}$, it suffices to test over a classical point in each connected component of $\widehat{\Ss}^{\ord,\an}_{K,v}$. We can do so over the canonical lifts of the $\overline{\Field}_p$-points of $\Ss^{\ord}_{K,k(v)}$ where Remarks~\ref{rem:M_reduction_over_canonical_lift} and~\ref{rem:igusa_tower_over_canonical_lift} do the job.
\end{proof}

\section{Hecke correspondences}\label{sec:hecke}

This section is a review of $p$-Hecke correspondences on Shimura varieties via isogenies of abelian varieties. We will consider both the generic fiber and the ordinary locus. The setup will be as in \S~\ref{subsec:background_on_shimura_varieties_of_hodge_type}.

\subsection{Isogenies and $p$-Hecke correspondences in the generic fiber}
\label{subsec:isogenies_and_p_hecke_correspondences_in_the_generic_fiber}

\begin{definition}
\label{defn:isogG_generic_pointwise}
Suppose that $L$ is an algebraically closed field in characteristic $0$ and that we have $s,t\in \Sh_K(L)$. A $p$-quasi-isogeny \footnote{By this, we mean an element $f\in \Hom(\mathcal{A}_{s},\mathcal{A}_{t})\otimes\Rat$ such that, for some $m,n\ge 0$, $p^mf\in \Hom(\mathcal{A}_{s},\mathcal{A}_{t})$ is an isogeny of degree $p^n$.} $f:\mathcal{A}_s\dashrightarrow \mathcal{A}_t$ \defnword{preserves $G$-structure} if the associated isomorphism $f^*:\bm{H}_{p,t}[p^{-1}]\xrightarrow{\simeq}\bm{H}_{p,t}[p^{-1}]$ carries $\{s_{\alpha,p,t}\}$ to $\{s_{\alpha,p,s}\}$. The \defnword{type} of such a $p$-quasi-isogeny is the unique class 
\[
\pow{g}\in G_{\Int_p}(\Int_p)\backslash G(\Rat_p)/G_{\Int_p}(\Int_p) = K_p\backslash G(\Rat_p)/K_p
\]
such that, for any choice of trivialization $\alpha:H_{\Int_p}\xrightarrow{\simeq}\bm{H}_{p,s}$ (resp. $\beta:H_{\Int_p}\xrightarrow{\simeq}\bm{H}_{p,t}$) carrying $\{s_{\alpha}\}$ to $\{s_{\alpha,p,s}\}$ (resp. to $\{s_{\alpha,p,t}\}$), we have $\beta^{-1}\circ (f^*)^{-1}\circ \alpha\in \pow{g}\subset G(\Rat_p)$.
\end{definition}

\begin{definition}
	\label{defn:isogG_generic}
For $s,t:\Spec R\to \Sh_K$ arbitrary, a $p$-quasi-isogeny $\mathcal{A}_s\dashrightarrow \mathcal{A}_t$ \defnword{preserves $G$-structure} (\defnword{and has type $\pow{g}$}) if its restriction over every geometric point of $\Spec R$ does so.\footnote{For connected $\Spec R$ it is enough to check at \emph{some} geometric point.}
\end{definition}

\begin{notation}
	With $s,y:\Spec R\to \Sh_K$ as above, write $\mathrm{QIsog}_G(s,t)$ for the space of $p$-quasi-isogenies from $\mathcal{A}_s$ to $\mathcal{A}_t$ preserving $G$-structure. Write $\mathrm{Isog}_G(s,t)$ for its subspace consisting of honest isogenies $\mathcal{A}_s\to \mathcal{A}_t$ and $\mathrm{QIsog}_{G,\pow{g}}(s,t)$ for the subspace consisting of quasi-isogenies of type $\pow{g}$. If $s=t$, we will write $\Aut^\circ_G(s)$ instead of $\mathrm{QIsog}_G(s,s)$. We will view each of $\mathrm{QIsog}_G,\mathrm{QIsog}_{G,\pow{g}},\mathrm{Isog}_G$ as presheaves over $\Sh_K\times \Sh_K$\footnote{The product is over $\Spec E$.} via the pair $(s,t)$.
\end{notation}

\begin{remark}
	\label{rem:isogG_generic_type_hyperspecial}
If $G_{\Int_p}$ is reductive, then, for a choice of maximal torus $T\subset G_{\Int_p}$ as in Remark~\ref{rem:hyperspecial_mu_integral}, the Cartan decomposition tells us that every double coset $\pow{g}$ admits a unique representative of the form $\lambda(p)$ for some dominant cocharacter $\lambda\in X_*(T)$ defined over $\Int_p$. In this situation, we will also write $\mathrm{QIsog}_{G,\lambda}(s,t)$ instead of  $\mathrm{QIsog}_{G,\pow{g}}(s,t)$. 
\end{remark}

\begin{construction}
	For any double coset $\pow{g}\in  K_p\backslash G(\Rat_p)/K_p$ and any choice of $h\in \pow{g}\subset G(\Rat_p)$, set $K_g = K \cap gKg^{-1}$. We have two maps $s_g,t_g:\Sh_{K_g}\to \Sh_K$ where $s_g$ arises from conjugation by $g^{-1}$ and $t_g$ from the natural inclusion, using which we can view $\Sh_{K_g}$ as a scheme over $\Sh_K\times \Sh_K$.
\end{construction}

\begin{proposition}
\label{prop:isogG_generic_desc}
There is an isomorphism $\mathrm{QIsog}_{G,\pow{g}}\xrightarrow{\simeq}\Sh_{K_g}$ over $\Sh_K\times \Sh_K$.
\end{proposition}
\begin{proof}
	Let's begin by constructing an isomorphism such that we have a commuting diagram
	\[
	\begin{diagram}
		\mathrm{QIsog}_{G,\pow{g}}&\rTo^{\simeq}&\Sh_{K_g}\\
		&\rdTo(1,2)_t\ldTo(1,2)_{t_g}\\
		&\Sh_K
	\end{diagram}
	\]
	For this, set $K_{g,p} = K_p\cap gK_pg^{-1}$ and note that we have canonical bijections
	\[
		K/K_g \xrightarrow{\simeq}K_p/K_{g,p}\xrightarrow[\simeq]{kK_{g,p}\mapsto kgK_p}K_pgK_p/K_p.
	\]
	From this one deduces: If $\Sh_{K^p}\to \Sh_K$ is the $K_p$-torsor parameterizing trivializations of $\bm{H}_p$ with its $G_{\Int_p}$-structure, then we have
	\[
		\Sh_{K^p}\times^{G_{\Int_p}(\Int_p)}[K_pg K_p/K_p]\xrightarrow{\simeq}\Sh_{K_g},
	\]
	so that we have a canonical isomorphism
	\[
		\Sh_{K^p}\times_{\Sh_K,t_g}\Sh_{K_g}\xrightarrow{\simeq}\Sh_{K^p}\times [K_pg K_p/K_p].
	\]

	On the other hand, suppose that we have $s,t:\Spec R\to \Sh_K$, a $G_{\Int_p}$-structure preserving trivialization $\beta:\underline{H}_{\Int_p}\xrightarrow{\simeq}\bm{H}_{p,t}$ of $p$-adic local systems over $\Spec R$, and a $G_{\Int_p}$-structure preserving $p$-quasi-isogeny $f:\mathcal{A}_s\dashrightarrow \mathcal{A}_t$ of type $\pow{g}$. Consider the composition
	\[
		\gamma:\bm{H}_{p,s}\xrightarrow{(f^*)^{-1}}\bm{H}_{p,t}[p^{-1}]\xrightarrow[\simeq]{\beta^{-1}}\underline{H}_{\Int_p}[p^{-1}].
	\]
	For any geometric point $y:\Spec L\to \Spec R$, since $f$ has type $\pow{g}$, we have
	\[
	\gamma_y(\bm{H}_{p,s\circ y}) = g_1^{-1}H_{\Int_p}\subset H_{\Rat_p}
	\]
	for some coset $g_1K_p\subset \pow{g}$, which is locally constant on $\Spec R$.
	In this way, we get a map
	\[
	\Sh_{K^p}\times_{\Sh_K,t}\mathrm{QIsog}_{G,\pow{g}}\to \Sh_{K^p}\times [K_pg K_p/K_p]\xrightarrow{\simeq}\Sh_{K^p}\times_{\Sh_K,t_g}\Sh_{K_g}.
	\]
  One checks that it is $K_p$-equivariant and so descends to a map $\mathrm{QIsog}_{G,\pow{g}}\to \Sh_{K_g}$.

  To construct the inverse, suppose now that we have $t:\Spec R\to \Sh_K$, $\beta:\underline{H}_{\Int_p}\xrightarrow{\simeq}\bm{H}_{p,t}$ and a locally constant map $\gamma:\Spec R\to K_pg^{-1} K_p/K_p$. We want to show that this arises from a pair $(s,f)$, where $s\in \Sh_K(R)$ and $f\in \mathrm{QIsog}_{G,\pow{g}}(s,t)$. Without loss of generality, we can assume that $\gamma$ is constant and corresponds to a coset $g_1K_p\subset \pow{g}$. We now get a $p$-adic local system
  \[
  	\beta[p^{-1}](\underline{g_1K_p})\subset \bm{H}_{p,t}[p^{-1}],
  \]
  which is of the form $(f^*)^{-1}(T_p(\mathcal{A})^\vee)$ for a $p$-quasi-isogeny $\mathcal{A}\dashrightarrow \mathcal{A}_t$ of abelian schemes over $\Spec R$. One checks now that we have $\mathcal{A}\simeq \mathcal{A}_s$ where $s\in \Sh_K(R)$ is the image of $(t,\beta,\gamma)$ under the map
  \[
  	\Sh_{K^p}\times [K_pg K_p/K_p]\xrightarrow{\simeq}\Sh_{K^p}\times_{\Sh_K,t_g}\Sh_{K_g}\xrightarrow{\mathrm{pr}_2}\Sh_{K_g}\xrightarrow{s_g}\Sh_K.
  \]

  This both completes the construction of the desired isomorphism and also shows that it is compatible with the maps $s$ and $s_g$.
\end{proof}

\subsection{Ordinary isogenies and $p$-Hecke correspondences}
\label{sub:ordinary_isogenies_and_p_hecke_correspondences}

Fix a place $v\vert p$ of $E$.

\begin{definition}
\label{defn:isogG_special_fiber}
	Suppose that $\kappa$ is algebraically closed in characteristic $p$ and that we have $s_0,t_0\in \Ss_{K,(v)}(\kappa)$. A $p$-quasi-isogeny  $f:\mathcal{A}_{s_0}\dashrightarrow \mathcal{A}_{t_0}$ \defnword{preserves $G$-structure} if the associated isomorphism $f^*:\bm{H}_{\cris,t_0}[p^{-1}]\xrightarrow{\simeq} \bm{H}_{\cris,s_0}[p^{-1}]$ carries $\{s_{\alpha,\cris,t_0}\}$ to $\{s_{\alpha,\cris,s_0}\}$. In general, if we have $s,t:\Spf R\to \Ss_{K,v}$ for some $p$-complete ring $R$, then a $p$-quasi-isogeny $\mathcal{A}_s\dashrightarrow\mathcal{A}_t$ \defnword{preserves $G$-structure} if its restriction over every algebraically closed point of $\Spf R$ preserves $G$-structure. 
\end{definition}

\begin{notation}
	With $s,t:\Spf R\to \Ss_{K,v}$ as above, write $\mathrm{QIsog}_G(s,t)$ for the space of $p$-quasi-isogenies from $\mathcal{A}_s$ to $\mathcal{A}_t$ preserving $G$-structure, and $\mathrm{Isog}_G(s,t)$ for its subspace consisting of honest isogenies $\mathcal{A}_s\to \mathcal{A}_t$. If $s=t$, we will write $\Aut^\circ_G(s)$ instead of $\mathrm{QIsog}_G(s,s)$.
\end{notation}

From here on, we will restrict ourselves to the completion along the ordinary locus $\widehat{\Ss}_{K,v}^{\ord}$ with Assumption~\ref{assump:canonical_lift} in force.

\begin{definition}
		If we have $s,t:\Spf R\to \widehat{\Ss}^{\ord}_{K,v}$ for some $p$-complete ring $R$, a quasi-isogeny $\mathcal{G}^{\et}_s\times \mathcal{G}^{\mult}_s\dashrightarrow\mathcal{G}^{\et}_t\times \mathcal{G}^{\mult}_t$ \defnword{preserves $M$-structure} if its restriction over every algebraically closed point of $\Spf R$ preserves $G$-structure. Here, we are using the property that, for all algebraically closed points $x_0\in \Ss^{\ord}_{K,k(v)}(\kappa)$, we have a canonical isomorphism $\mathcal{G}_{x_0}\xrightarrow{\simeq}\mathcal{G}^{\et}_{x_0}\times \mathcal{G}^{\mult}_{x_0}$.
\end{definition}

\begin{notation}
	With $s,t:\Spf R\to \Ss_{K,v}$ as above, write $\mathrm{QIsog}_M(s,t)$ for the space of quasi-isogenies from $\mathcal{G}^{\et}_s\times \mathcal{G}^{\mult}_s$ to $\mathcal{G}^{\et}_t\times \mathcal{G}^{\mult}_t$ preserving $M$-structure, and $\mathrm{Isog}_M(s,t)$ for its subspace consisting of honest isogenies $\mathcal{A}_s\to \mathcal{A}_t$.
\end{notation}

\begin{definition}
	\label{defn:isogG_ordinary_type_pointwise}
If $\kappa$ is algebraically closed and we have $f\in \mathrm{QIsog}_G(s_0,t_0)$ for $s_0,t_0\in \Ss^{\ord}_{K,k(v)}(\kappa)$, the \defnword{type} of $f$ is the double coset $\pow{m(f)}\in M_{\Int_p}(\Int_p)\backslash M(\Rat_p)/M_{\Int_p}(\Int_p)$ determined as follows. Choose trivializations 
\[
	\alpha:W(\kappa)\otimes_{\Int_p}H_{\Int_p}\xrightarrow{\simeq}\bm{H}_{\cris,s_0}\;;\; \beta:W(\kappa)\otimes_{\Int_p}H_{\Int_p}\xrightarrow{\simeq}\bm{H}_{\cris,t_0}
\]
witnessing the ordinariness of $s_0$ and $t_0$: these are each well-defined up to multiplication by $M_{\Int_p}(\Int_p)$. Then $\pow{m(f)}$ is the class of $\beta\circ (f^*)^{-1}\circ\alpha^{-1}\in M(\Rat_p)$.
\end{definition}

\begin{definition}
\label{defn:isogG_ordinary_type}
	If we have $s,t:\Spf R\to \widehat{\Ss}^{\ord}_{K,v}$, $f\in \mathrm{Isog}_G(s,t)$\footnote{It is important to work with honest isogenies here, since we will have use for their deformation theory. Quasi-isogenies on the other hand always deform uniquely.} and $\pow{m}\in M_{\Int_p}(\Int_p)\backslash M(\Rat_p)/M_{\Int_p}(\Rat_p)$ then we will say that $f$ \defnword{has type }$\pow{m}$ if its restriction to each algebraically closed point of $\Spf R$ has this property. We will write $\mathrm{QIsog}_{G,\pow{m}}(s,t)$ for the subspace consisting of such $f$.
\end{definition}

\begin{example}
\label{ex:frobenius_type}
	Suppose that $R$ is an $\Field_p$-algebra and that we have $s:\Spec R\to \Ss_{K,k(v)}^{\ord}$. Hitting $s$ with the absolute Frobenius of $R$ gives another point $s^{(1)}\in  \Ss_{K,k(v)}^{\ord}(R)$ and we have the relative Frobenius map $F_s:\mathcal{A}_s\to \mathcal{A}_{s^{(1)}}$. This is a $p$-isogeny, and is in fact $G$-structure preserving with type $[\mu_v(p)^{-1}]$.
\end{example}

\begin{remark}
Definition~\ref{defn:isogG_ordinary_type} gives something non-empty only when $\pow{m}$ is contained in $M(\Rat_p)\cap \End(H_{\Int_p})$. This will suffice for our purposes, but one can extend the definition usefully to all double cosets $\pow{m}$ in the following way: Choose $r\ge 0$ such that $\pow{p^rm}$ is contained in $M(\Rat_p)\cap \End(H_{\Int_p})$, and define $\mathrm{QIsog}_{G,\pow{m}}(s,t)$ to be the subspace of $\mathrm{QIsog}_G(s,t)$ consisting of $f$ such that $p^r$ lies in $\mathrm{QIsog}_{G,\pow{p^rm}}(s,t)$. By definition, multiplication-by-$p^r$, gives an isomorphism $\mathrm{QIsog}_{G,\pow{m}}(s,t)\xrightarrow{\simeq}\mathrm{QIsog}_{G,\pow{p^rm}}(s,t)$.
\end{remark}

\begin{definition}
	\label{defn:isogM_ordinary_type}
If we have $s,t:\Spf R\to \widehat{\Ss}^{\ord}_{K,v}$, $f\in \mathrm{QIsog}_M(s,t)$ and $\pow{m}\in M_{\Int_p}(\Int_p)\backslash M(\Rat_p)/M_{\Int_p}(\Int_p)$ then we will say that $f$ \defnword{has type }$\pow{m}$ if its restriction to each algebraically closed point of $\Spf R$ has this property. We will write $\mathrm{QIsog}_{M,\pow{m}}(s,t)$ for the subspace consisting of such $f$.
\end{definition}

\begin{notation}
	Fix a double coset $\pow{m}\in M_{\Int_p}(\Int_p)\backslash M(\Rat_p)/M_{\Int_p}(\Int_p)$. Let
\[
	\widehat{\mathrm{QIsog}}^{\ord}_{G,\pow{m}}\xrightarrow{(s,t)} \widehat{\Ss}^{\ord}_{K,v}\times \widehat{\Ss}^{\ord}_{K,v}\footnote{The product is of formal stacks over $\Spf \Reg{E,v}$.}\;;\;	\widehat{\mathrm{QIsog}}^{\ord}_{M,\pow{m}}\xrightarrow{(s,t)} \widehat{\Ss}^{\ord}_{K,v}\times \widehat{\Ss}^{\ord}_{K,v}
\]
be the formal presheaves given by $\mathrm{QIsog}_{G,\pow{m}}(s,t)$ and $\mathrm{QIsog}_{M,\pow{m}}(s,t)$, respectively. 
\end{notation}

\begin{remark}
	\label{rem:isogM_ordinary_type_hyperspecial}
If $G_{\Int_p}$ is reductive, then so is $M_{\Int_p}$, and for the choice of maximal torus $T\subset G_{\Int_p}$ as in Remark~\ref{rem:hyperspecial_mu_integral}, the Cartan decomposition tells us that every double coset $\pow{m}$ admits a unique representative of the form $\lambda(p)$ for some cocharacter $\lambda\in X_*(T)$ defined over $\Int_p$ and dominant \emph{with respect to} $M_{\Int_p}$. In this situation, we will replace the $\pow{m}$ in the subscript with $\lambda$, and write $\mathrm{QIsog}_{G,\lambda}(s,t)$ and $\mathrm{QIsog}_{M,\lambda}(s,t)$ instead. This notation will also be used with $\lambda$ replaced by any other cocharacter in its $M_{\Int_p}$-conjugacy class. 
\end{remark}

\begin{construction}
Suppose that we have $s,t:\Spf R\to \widehat{\Ss}^{\ord}_{K,v}$, $f\in \mathrm{QIsog}_{M,\pow{m}}(s,t)$ and a section
\[
\beta:R\otimes_{\Int_p}\mathcal{G}_0\xrightarrow{\simeq}\mathcal{G}^{\et}_{t}\times \mathcal{G}^{\mult}_t 
\]
of $\widehat{\mathrm{Ig}}^{\ord}_{K,v}(t)$. For any section $\alpha\in\widehat{\mathrm{Ig}}^{\ord}_{K,v}(s)$, the composition
\[
	R\otimes_{\Int_p}\mathcal{G}_0\xrightarrow[\simeq]{\alpha}\mathcal{G}^{\et}_s\times \mathcal{G}^{\mult}_s\dashrightarrow \mathcal{G}^{\et}_t\times \mathcal{G}^{\mult}_t\xrightarrow[\simeq]{\beta^{-1}}R\otimes_{\Int_p}\mathcal{G}_0
\]
is a self-quasi-isogeny that preserves $G$-structure, and so yields a locally constant map $\Spf R\to M(\Rat_p)$. The induced map to $M(\Rat_p)/M_{\Int_p}(\Int_p)$ is independent of the choice of $\alpha$ and lands in $M_{\Int_p}(\Int_p) mM_{\Int_p}(\Int_p)/M_{\Int_p}(\Int_p)$. In this way, we obtain a canonical map
\begin{align}\label{eqn:isogM_target}
\widehat{\mathrm{Ig}}^{\ord}_{K,v}\times_{\widehat{\Ss}^{\ord}_{K,v},t}\widehat{\mathrm{QIsog}}^{\ord}_{M,\pow{m}}&\to \widehat{\mathrm{Ig}}^{\ord}_{K,v}\times \underline{M_{\Int_p}(\Int_p) mM_{\Int_p}(\Int_p)/M_{\Int_p}(\Int_p)}.
\end{align}
By noting that taking inverses of quasi-isogenies flips source and target and switches a type $\pow{m}$ to its inverse $\pow{m^{-1}}$, we also obtain a map
\begin{align}\label{eqn:isogM_source}
\widehat{\mathrm{Ig}}^{\ord}_{K,v}\times_{\widehat{\Ss}^{\ord}_{K,v},s}\widehat{\mathrm{QIsog}}^{\ord}_{M,\pow{m}}&\to \widehat{\mathrm{Ig}}^{\ord}_{K,v}\times \underline{M_{\Int_p}(\Int_p) m^{-1}M_{\Int_p}(\Int_p)/M_{\Int_p}(\Int_p)}.
\end{align}
\end{construction}

\begin{notation}
	For a double coset $\pow{m} = M_{\Int_p}(\Int_p)mM_{\Int_p}(\Int_p)$, set
	\[
		\widehat{\mathrm{Ig}}^{\ord}_{\pow{m}} \defn \widehat{\mathrm{Ig}}^{\ord}_{K,v}\times^{M_{\Int_p}(\Int_p)}\underline{M_{\Int_p}(\Int_p)mM_{\Int_p}(\Int_p)/M_{\Int_p}(\Int_p)}.
	\]
\end{notation}

\begin{proposition}
	\label{prop:isogM_desc_without_isoms}
The maps~\eqref{eqn:isogM_target} and~\eqref{eqn:isogM_source} descend to maps sitting in commuting diagrams of formal stacks 
	\begin{align*}
			\begin{diagram}
		\widehat{\mathrm{QIsog}}^{\ord}_{M,\pow{m}}&\rTo&\widehat{\mathrm{Ig}}^{\ord}_{\pow{m}^{-1}}\\
		&\rdTo(1,2)_s\ldTo(1,2)\\
		&\widehat{\Ss}^{\ord}_{K,v}
	\end{diagram}\qquad;\qquad
	\begin{diagram}
		\widehat{\mathrm{QIsog}}^{\ord}_{M,\pow{m}}&\rTo&\widehat{\mathrm{Ig}}^{\ord}_{\pow{m}}\\
		&\rdTo(1,2)_t\ldTo(1,2)\\
		&\widehat{\Ss}^{\ord}_{K,v}
	\end{diagram}
	\end{align*}
where the diagonal maps are finite \'etale over $\widehat{\Ss}_{K,v}^{\ord}$.\footnote{We will see in the next subsection that the top arrows are in fact isomorphisms.}
\end{proposition}
\begin{proof}
 That the map descends amounts to checking $K_p$-equivariance, which is straightforward.

 By construction, $\widehat{\mathrm{QIsog}}^{\ord}_{M,\pow{m}}$ is isomorphic---for some $k\ge 1$ sufficiently large---to an open and closed substack of the formal stack $\widehat{\mathrm{Isog}}_{p^k}$ over $\widehat{\Ss}^{\ord}_{K,v}\times \widehat{\Ss}^{\ord}_{K,v}$ whose fiber over a point $(s,t)$ parameterizes $p$-isogenies
 \[
 	\mathcal{G}^{\et}_s\times \mathcal{G}^{\mult}_s \to \mathcal{G}^{\et}_t\times \mathcal{G}^{\mult}_t
 \]
 of degree bounded by $p^k$. Via the projection onto the factor via $s$ (resp. $t$) this formal stack is simply parameterizing finite flat subgroup schemes of $\mathcal{G}^{\et}_s\times \mathcal{G}^{\mult}_s$ (resp. $\mathcal{G}^{\et}_t\times \mathcal{G}^{\mult}_t$) of rank at most $p^k$. This is represented by a finite \'etale stack over $\widehat{\Ss}^{\ord}_{K,v}$.
\end{proof}

\begin{construction}
\label{const:torus_for_isogenies}
Fix an algebraically closed field $\kappa$ of characteristic $p$. For $m\in M(\Rat_p)\cap \End(H_{\Int_p})$, write $m^0,m^1$ for the induced endomorphisms of $H^0_{\Int_p}$ and $H^1_{\Int_p}$. We then get two maps
\[
\_\circ m^1, m^0\circ \_:\Lie U^-_{\mu_v}\to \Hom(H^1_{\Int_p},H^0_{\Int_p})
\]
which in turn give us maps $\psi^1_m,\psi^0_m:\widehat{T}_G \to \widehat{T}$ of formal tori over $W(\kappa)$.

Let $\widehat{T}_{G,m}\subset \widehat{T}_{G}\times \widehat{T}_G$\footnote{The product here is over $\Spf W(\kappa)$.} be the diagonalizable formal group sitting in an exact sequence
\[
1\to \widehat{T}_{G,m}\xrightarrow{(s_m,t_m)} \widehat{T}_{G}\times \widehat{T}_G\xrightarrow{\psi^1_m\circ \mathrm{pr}_1-\psi^0_m\circ \mathrm{pr}_2}\widehat{T}.
\]
\end{construction}

\begin{remark}
	\label{rem:TGm_finite_flat}
The maps $\_\circ m^1, m^0\circ \_$ are injective, and after inverting $p$ one differs from the other by conjugation by $m$. This implies that the maps $s_m$ and $t_m$ are both finite flat maps that are purely inseparable mod-$p$: Indeed, the corresponding maps of character groups are isomorphisms of $\Rat_p$-vector spaces after inverting $p$.
\end{remark}

\begin{proposition}
	\label{prop:isogG_ordinary_deformation_rings}
Suppose that we have $(s_0,t_0,f)\in \widehat{\mathrm{QIsog}}^{\ord}_{G,\pow{m}}(\kappa)$, and let $\widehat{U}_{(s_0,t_0,f)}$ be the deformation functor for $\widehat{\mathrm{QIsog}}^{\ord}_{G,\pow{m}}$ on $\mathrm{Art}_{W(\kappa)}$. Then we can choose isomorphisms $\widehat{T}_G\xrightarrow{\simeq}\widehat{U}_{s_0}$ and $\widehat{T}_G\xrightarrow{\simeq}\widehat{U}_{t_0}$ as in Proposition~\ref{prop:ordinary_deformations} such that we have a commuting diagram
\[
	\begin{diagram}
		\widehat{T}_{G,m}&\rTo^{(s_m,t_m)}&\widehat{T}_G\times \widehat{T}_G\\
		\dTo^{\simeq}&&\dTo_{\simeq}\\
		\widehat{U}_{(s_0,t_0,f)}&\rTo_{(s,t)}&\widehat{U}_{s_0}\times \widehat{U}_{t_0}
	\end{diagram}
\]
where the left vertical arrow is also an isomorphism.
\end{proposition}
\begin{proof}
  Choose isomorphisms 
\[
  \mathcal{G}_{t_0} \xleftarrow[\simeq]{\beta}\kappa\otimes_{\Int_p}\mathcal{G}_0\xrightarrow[\simeq]{\alpha}\mathcal{G}_{s_0}
 \]
 witnessing the ordinariness of $t_0$ and $s_0$ such that the $G$-structure preserving self-isogeny
 \[
 \beta^{-1}\circ	f\circ \alpha: \kappa\otimes_{\Int_p}\mathcal{G}_0\to \kappa\otimes_{\Int_p}\mathcal{G}_0
 \]
 corresponds to $m\in M(\Rat_p) = \Aut^{\circ}_G(\mathcal{G}_0)$. By Proposition~\ref{prop:ordinary_deformations}, we can use $\alpha$ and $\beta$ to get isomorphisms $\widehat{T}_G\xrightarrow{\simeq}\widehat{U}_{s_0}$ and $\widehat{T}_G\xrightarrow{\simeq}\widehat{U}_{t_0}$

	Over $\widehat{T}_G\times \widehat{T}_G\simeq \widehat{U}_{s_0}\times \widehat{U}_{t_0}$, we have two extensions $\mathcal{G}_s$ and $\mathcal{G}_t$ of $\mathcal{G}_0^{\et}$ by $\mathcal{G}_0^{\mult}$, and the locus $\widehat{U}_{(s_0,t_0,f)}$ is where the isogeny $f$ lifts to an isogeny between these extensions. Now, the \'etale  part $f^{\et}$ (resp. multiplicative part $f^{\mult}$) of $f$ lifts uniquely, and is given by
	\begin{align*}
		\mathcal{G}_0^{\et} = \underline{\Hom}(H^0_{\Int_p},\Rat_p/\Int_p)&\xrightarrow{\_\circ m^0}\underline{\Hom}(H^0_{\Int_p},\Rat_p/\Int_p) = \mathcal{G}_0^{\et}\\
		\text{(resp. }	\mathcal{G}_0^{\mult} = \underline{\Hom}(H^1_{\Int_p},\mu_{p^\infty})&\xrightarrow{\_\circ m^1}\underline{\Hom}(H^1_{\Int_p},\mu_{p^\infty}) = \mathcal{G}_0^{\mult}).
	\end{align*}

	We now obtain two further extensions over $\widehat{T}_G\times \widehat{T}_G$: One by pulling back $\mathcal{G}_t$ along $f^{\et}$ and the other by pushing forward along $f^{\mult}$. The locus where $f$ lifts is the same as the locus where these two extensions are isomorphic. To finish, one just checks that the first extension is obtained from the map $\psi^0_m:\widehat{T}_G\to \widehat{T}$, and the second from the map $\psi^1_m:\widehat{T}_G\to \widehat{T}$, where we are using Serre-Tate ordinary theory to identify $\widehat{T}$ with the deformation space of such extensions.
\end{proof}

\begin{corollary}
	\label{cor:isogG_ordinary_deformation_desc}
The map $\widehat{\mathrm{QIsog}}^{\mathrm{ord}}_{G,\pow{m}}\to \widehat{\mathrm{QIsog}}^{\mathrm{ord}}_{M,\pow{m}}$ is a finite flat homeomorphism.
\end{corollary}
\begin{proof}
	By construction, the map is an isomorphism on algebraically closed points. Moreover, by definition, $\widehat{\mathrm{QIsog}}^{\mathrm{ord}}_{G,\pow{m}}$ is an open and closed substack of the formal stack over $\widehat{\Ss}^{\ord}_{K,v}\times \widehat{\Ss}^{\ord}_{K,v}$ parameterizing isogenies between $\mathrm{pr}_1^* \mathcal{A}$ and $\mathrm{pr}_2^*\mathcal{A}$. Therefore, it is enough now to know that the map is a finite flat homeomorphism on completions at algebraically closed points. This follows from Proposition~\ref{prop:isogG_ordinary_deformation_rings} and Remark~\ref{rem:TGm_finite_flat}.
\end{proof}



\subsection{Comparison between the two notions of $p$-Hecke correspondences}

Here, we will look at the relationship between $p$-Hecke correspondences over the generic fiber and over the ordinary locus.

\begin{remark}
	\label{rem:qisog_points_comparison}
Suppose that $s_0,t_0\in \Ss^{\ord}_{K,k(v)}(\kappa)$ are algebraically closed points, that $L/W(\kappa)[1/p]$ is a finite extension and that $s,t\in \Ss_{K,v}(\Reg{L})$ are lifts of $s_0,t_0$. Write $s_{\eta},t_{\eta}$ for the associated $L$-valued points of $\Sh_K$. Viewing $s,t$ as $\Spf\Reg{L}$-valued points of $\widehat{\Ss}^{\ord}_{K,v}$, we obtain the space $\mathrm{QIsog}_G(s,t)$. On the other hand, we also have the space $\mathrm{QIsog}_G(s_{\eta},t_{\eta})$ from \S~\ref{subsec:isogenies_and_p_hecke_correspondences_in_the_generic_fiber}. 
\end{remark}

\begin{proposition}
\label{prop:qisog_points_comparison}
	There is a canonical isomorphism
\[
	\mathrm{QIsog}_G(s,t)\xrightarrow{\simeq} \mathrm{QIsog}_G(s_{\eta},t_{\eta}).
\]
Furthermore, if $s$ and $t$ are the canonical lifts of $s_0$ and $t_0$, respectively, then the natural map $\mathrm{QIsog}_G(s,t)\to \mathrm{QIsog}_G(s_0,t_0)$ is also an isomorphism.
\end{proposition}
\begin{proof}
	Write $\mathrm{QIsog}(\mathcal{A}_s,\mathcal{A}_t)$ (resp. $\mathrm{QIsog}(\mathcal{A}_{s_\eta},\mathcal{A}_{t_\eta}))$ for the space of $p$-quasi-isogenies from $\mathcal{A}_s$ to $\mathcal{A}_t$ (resp. between their generic fibers). By the N\'eronian property of abelian schemes, the natural map
	\[
		\mathrm{QIsog}(\mathcal{A}_s,\mathcal{A}_t)\to \mathrm{QIsog}(\mathcal{A}_{s_\eta},\mathcal{A}_{t_\eta})
	\]
	is a bijection. Therefore, we only have to check that the two \emph{a priori} different notions of $G$-structure preservation, one using the crystalline realization (on the source) and the other using the \'etale realization (on the target) are compatible. This follows from assertion (2) of Proposition~\ref{lem:GIsoc_struc}. 

	The second assertion follows from Serre-Tate ordinary deformation theory, which shows that the natural map
	\[
		\mathrm{QIsog}(\mathcal{A}_s,\mathcal{A}_t)\to \mathrm{QIsog}(\mathcal{A}_{s_0},\mathcal{A}_{t_0})
	\]
	is a bijection when $s$ and $t$ are canonical lifts.
\end{proof}

\begin{remark}
	\label{rem:qisog_points_comparison_types}
Let $\iota$ be the inverse of the first isomorphism from Proposition~\ref{prop:qisog_points_comparison}. Given $\pow{g}\in K_p\backslash G(\Rat_p)/K_p$, one sees, using Proposition~\ref{prop:igusa_tower_generic_fiber}, that we have
\[
	\iota\left( \mathrm{QIsog}_{G,\pow{g}}(s_{\eta},t_{\eta})\right)\subset \bigsqcup_{\pow{m}\in \mathcal{S}(\pow{g})}\mathrm{QIsog}_{G,\pow{m}}(s,t)
\]
where $\mathcal{S}(\pow{g})\subset M_{\Int_p}(\Int_p)\backslash M(\Rat_p)/M_{\Int_p}(\Int_p)$ consists of those double cosets which admit representatives $m$ in the image of $K_pgK_p\cap P^-(\Rat_p)$.
\end{remark}

\begin{remark}
	\label{rem:qisog_comparison_global}
We can globalize Proposition~\ref{prop:qisog_points_comparison}. Over $\widehat{\Ss}^{\ord,\an}_{K,v}\times \widehat{\Ss}^{\ord,\an}_{K,v}$, we have the analytification $\widehat{\mathrm{QIsog}}^{\ord,\an}_G$, and we also have the restriction of the rigid analytification $\mathrm{QIsog}^{\an}_G$ of the stack of $G$-structure preserving isogenies in the generic fiber. These are both finite \'etale over each factor of the product, and so Proposition~\ref{prop:qisog_points_comparison} shows that we actually have an isomorphism
\[
\mathrm{QIsog}^{\an}_G\vert_{\widehat{\Ss}^{\ord,\an}_{K,v}\times \widehat{\Ss}^{\ord,\an}_{K,v}}\xrightarrow{\simeq}	\widehat{\mathrm{QIsog}}^{\ord,\an}_G.
\]
\end{remark}

\begin{corollary}
	\label{cor:qisogM_ordinary_desc}
The top horizontal arrows in Proposition~\ref{prop:isogM_desc_without_isoms} are isomorphisms.
\end{corollary}
\begin{proof}
	It is enough to check that the horizontal arrow in the second diagram is an isomorphism. Since both source and target are finite \'etale over $\widehat{\Ss}^{\ord}_{K,v}$, it is enough to know that the map is an isomorphism on $\kappa = \overline{\Field}_p$-points. Fix $t_0\in \Ss_{K,k(v)}^{\ord}(\kappa)$, and let $\mathrm{QIsog}^{\ord}_{G,t_0}(\kappa)$ be the fiber of $t:\widehat{\mathrm{QIsog}}^{\ord}_{G}(\kappa)\to \widehat{\Ss}^{\ord}_{K,v}(\kappa)$ over $t_0$. For a fixed choice of isomorphism $\beta:W(\kappa)\otimes_{\Int_p}H_{\Int_p}\xrightarrow{\simeq}\bm{H}_{\cris,t_0}$ as in Definition~\ref{defn:isogG_ordinary_type_pointwise}, we obtain a map 
	\begin{align*}
				\mathrm{QIsog}^{\ord}_{G,t_0}(\kappa)&\to M(\Rat_p)/M_{\Int_p}(\Int_p)\\
				f&\mapsto m(f)M_{\Int_p}(\Int_p),
	\end{align*}
	where $m(f)\in M(\Rat_p)$ is the class of $\beta\circ (f^*)^{-1}\circ \alpha^{-1}$ for any choice of isomorphism $\alpha$ as in the same definition. It is now enough to know that this map is an isomorphism. 

	For this, let $t\in \Ss_K(W(\kappa))$ be the canonical lift of $t_0$ and let $t_\eta\in \Sh_K(L)$ with $L = W(\kappa)[p^{-1}]$ be its generic fiber. We will now show that the map is injective. Indeed, if we have $f\in \mathrm{QIsog}_G(s_0,t_0)$ and $\tilde{f}\in \mathrm{QIsog}_G(\tilde{s}_0,t_0)$ such that $m(f)M_{\Int_p}(\Int_p) = m(\tilde{f})M_{\Int_p}(\Int_p)$, then, since $m(\tilde{f})^{-1}m(f)\in M_{\Int_p}(\Int_p)$, one finds that $\tilde{f}^{-1}\circ f\in \mathrm{QIsog}_G(s_0,\tilde{s}_0)$ is an isomorphism $\mathcal{A}_{s_0}\xrightarrow{\simeq} \mathcal{A}_{\tilde{s}_0}$. If $s,\tilde{s}$ are the canonical lifts, the second assertion in Proposition~\ref{prop:qisog_points_comparison}, this isomorphism lifts to an element in $\mathrm{QIsog}_G(s,\tilde{s})$ that is an isomorphism of abelian schemes. One now finds from Proposition~\ref{prop:isogG_generic_desc} that this is only possible if $s$ and $\tilde{s}$, along with the corresponding lifts of $f$ and $\tilde{f}$, yield the same point of the fiber of $t:\mathrm{QIsog}_G\to \Sh_K$ over $t_\eta$. 

	To show surjectivity, fix an algebraic closure $\overline{L}$ for $L$, and let $\mathrm{QIsog}_{G,t_\eta}(\overline{L})$ be the fiber of $t:\mathrm{QIsog}_{G}(\overline{L})\to \Sh_K(\overline{L})$ over $t_\eta$. We then have a commutative diagram
	\[
		\begin{diagram}
			\mathrm{QIsog}_{G,t_\eta}(\overline{L})&\rTo^{f\mapsto g(f)G_{\Int_p}(\Int_p)}_{\simeq}&G(\Rat_p)/G_{\Int_p}(\Int_p)\\
			\dTo&&\dTo\\
      \mathrm{QIsog}^{\ord}_{G,t_0}(\kappa)&\rTo_{f\mapsto m(f)M(\Rat_p)}&M(\Rat_p)/M_{\Int_p}(\Int_p)
		\end{diagram}
	\] 
	Here, the top arrow is the isomorphism obtained by taking the fibers of the isomorphism from Proposition~\ref{prop:isogG_generic_desc} over $t_\eta\in \Sh_K(\overline{L})$ in the second factor and taking the disjoint union over all $\pow{g}$. The vertical arrow on the left is obtained from Remark~\ref{rem:qisog_comparison_global} and the reduction map. 

	It is enough to show that the right vertical map is surjective. Unwinding definitions, and using Proposition~\ref{prop:igusa_tower_generic_fiber}, one gets the following description of this map: Write $t_{\overline{\eta}}$ for the points $t_\eta$ viewed as a $\bar{L}$-valued point. Using the Iwasawa decomposition 
\[
G(\Rat_p) = U^-_{\mu_v}(\Rat_p)M(\Rat_p)G_{\Int_p}(\Int_p),
\]
we can write $g\in G(\Rat_p)$ in the form $n^-(g)m(g)k(g)$, where the coset $m(g)M_{\Int_p}(\Int_p)$ is canonically determined. The right vertical map sends $gG_{\Int_p}(\Int_p)$ to $m(g)M_{\Int_p}(\Int_p)$. This finishes the proof of surjectivity.
\end{proof}

\section{An argument of Chai and Hida}
\label{sec:monodromy}

The purpose of this section is to abstract some ideas due to Chai and Hida on a `pure thought' study of the monodromy of Igusa towers over Shimura varieties, and apply them to the particular situation of ordinary loci. All the key ideas can already be found in~\cite{Hida2011-fi} and~\cite{Chai2011-ik}. It is also possible that the main statement here can be deduced from the very general results of van Hoften and Xiao~\cite{vanhoften_xiao}.

\subsection{The abstract setup}

\begin{proposition}\label{ord:proposition:weakapprox}
Let $H$ be a connected reductive group over $\Rat$ such that $H_{\Rat_p}$ contains a maximal torus that splits over a cyclic extension of $\Rat_p$ (this hypothesis holds in particular when $H$ is unramified at $p$). Then $H$ satisfies weak approximation with respect to $\{p,\infty\}$; that is, $H(\Rat)$ is dense in $H(\Rat_p)\times H(\Real)$.
\end{proposition}
\begin{proof}
 This is essentially contained in \cite{plat_rap}. If $H$ is semi-simple and simply connected, the result follows directly from Theorem 7.8 of \emph{loc. cit.} In general, let $\widetilde{H}$ be the simply connected cover of the derived group of $H$. Then we find from Proposition 2.11 of \emph{loc. cit.} that there are quasi-trivial\footnote{This means that the Galois representation attached to the character group is a permutation representation.} tori $T_1$ and $T_2$ over $\Rat$, and an integer $m\geq 1$ such that there is a central isogeny: $\widetilde{H}^m\times T_1\rightarrow H^m\times T_2$. In fact, the proof of this result shows that we can choose $T_1$ and $T_2$ to have the same splitting field as the maximal central torus of $H$. It is easy to see $H$ satisfies weak approximation with respect to $\{p,\infty\}$ whenever $H^m\times T_2$ does, so we can replace $H$ by the latter group and assume that it admits a central cover $H_1\to H$ where $H_1$ is a product of a semi-simple, simply connected group with a quasi-trivial torus.

 Let $F$ be the kernel of $H_1\to H$: It is a central sub-group of $H_1$, and so, by our hypothesis, splits over a cyclic extension of $\Rat_p$. The result now follows from Proposition 7.10 and Corollary 2 in Ch. 7 of \emph{loc. cit.}.
\end{proof}

\begin{corollary}\label{ord:cor:weakapprox}
Let the notation be as in the hypotheses of Proposition~\ref{ord:proposition:weakapprox} above, and let $H_{\Int_{(p)}}$ be a smooth group scheme over $\Int_{(p)}$ with generic fiber $H$. For any $\Int_{(p)}$-algebra $R$ set $H(R) = H_{\Int_{(p)}}(R)$. With the hypotheses as in the proposition, for any integer $n\geq 1$, the map 
\[
H(\Int_{(p)})\to H(\Int/p^n\Int) 
\]
is surjective.
\end{corollary}
\begin{proof}
Let $P_n=\ker(H(\Int_p)\to H(\Int/p^n\Int))$. By (\ref{ord:proposition:weakapprox}), $H(\Rat)P_n=H(\Rat_p)$. We now have:
  \[
   H(\Int_p)=H(\Rat_p)\cap H(\Int_p)=H(\Rat)P_n\cap H(\Int_p)=H(\Int_{(p)})P_n.
  \]
The corollary now follows, since the map $H(\Int_p)\to H(\Int/p^n\Int)$ is surjective by the smoothness of $H_{\Int_{(p)}}$.
\end{proof}

\begin{notation}
Suppose that $S$ is an algebraic or formal algebraic stack, and that we have an $H(\Int_p)$-torsor over $S$ given by a sequence $\mathcal{P} = \{ \mathcal{P}_n\}_{n\geq 1}$ of compatible finite \'etale $H(\Int/p^n\Int)$-torsors $\mathcal{P}_n\to S$. If $\mathsf{X}$ is an $H_{\Int_p}$-equivariant scheme over $\Int_p$, for every $n\ge 1$, we set
\[
\mathcal{P}_{\mathsf{X},n} = \mathcal{P}_{\mathsf{X}(\Int/p^n\Int)}.
\] 
\end{notation}

\subsubsection{}\label{subsubsecchai_hida_setup}
Suppose now that $G$ is a reductive group over $\Rat$, $T$ is a finite set of primes containing $p$, and that $S$ is a scheme over $\Int_{(p)}$ equipped with an action of $G(\Adele_f^T)$. Assume that this action lifts to one on the $H(\Int_p)$-torsor $\mathcal{P}$ that commutes with the $H(\Int_p)$-action. Suppose further that there is another reductive group $J$ over $\Rat$ with the following properties:
\begin{itemize}
	\item There exists an embedding 
	\[
     \psi: J_{\Adele_f^T}\hookrightarrow G_{\Adele_f^T}.
	\]
	\item There exists a (necessarily smooth) model $J_{\Int_{(p)}}$ for $J$ over $\Int_{(p)}$, and an isomorphism
\[
\varphi:J_{\Int_{(p)}}\otimes_{\Int_{(p)}}\Int_p \xrightarrow{\simeq}H_{\Int_p}
\]
\end{itemize}

In particular, we have an embedding
\[
\Phi:J(\Int_{(p)})\xrightarrow{m\mapsto (\varphi(m),\psi(m))} H(\Int_p)\times G(\Adele_f^T),
\]
inducing for every $n\geq 1$ a map
\[
\Phi_n: J(\Int_{(p)})\xrightarrow{m\mapsto (\varphi_n(m),\psi(m))} H(\Int/p^n\Int)\times G(\Adele_f^T).
\]

\begin{notation}
	Suppose that $Q\subset H_{\Int_p}$ is a closed $\Int_p$-subgroup scheme such that $\mathsf{X} = H_{\Int_p}/Q$ is represented by a scheme. Let $\widetilde{J}_{\Int_{(p)}}$ (resp. $\widetilde{H}_{\Int_p}$) be the normalization of $J_{\Int_{(p)}}$ (resp. $H_{\Int_p}$) in $\widetilde{J}$ (resp. $\widetilde{H}$), and let $\widetilde{Q}$ be the pre-image of $Q$ in $\widetilde{H}_{\Int_p}$. Let $Z_{H,p}\subset H_{\Int_p}$ be the Zariski closure of the center $Z_H\subset H$.
\end{notation}

\begin{proposition}\label{prop:trivial_set}
Suppose that the following conditions hold:
\begin{enumerate}
	\item\label{hyp:sc}$\rho_G(\widetilde{G}(\Adele_f^T))$ acts trivially on $\pi_0(S)$;
	\item\label{hyp:isotropic}For all $\ell\notin T$, $G_{\Rat_\ell}$ is isotropic;
	\item\label{hyp:hecke}$G(\Adele_f^T)$ acts transitively on $\pi_0(S)$;
	\item\label{hyp:trivial}$\Phi(J(\Int_{(p)}))$ fixes a point $\varpi\in \pi_0(\mathcal{P})$;
	\item\label{hyp:central}$Q$ contains $Z_{H,p}$;
    \item\label{hyp:smooth}The $\Int_{(p)}$-group $\widetilde{J}_{\Int_{(p)}}$ (equivalently, the $\Int_p$-group $\widetilde{H}_{\Int_p}$) and the $\Int_p$-group $\widetilde{Q}$ are smooth with connected special fiber.
\end{enumerate}
Then, for every $n\in \Int_{\geq 1}$, the map $\pi_0(\mathcal{P}_{\mathsf{X},n})\to \pi_0(S)$ is a bijection.
\end{proposition}
\begin{proof}
We will need the following consequence of the Kneser-Tits conjecture (see~\cite[Theorem 7.6]{plat_rap}): For any simply connected isotropic group $D$ over $\Rat_\ell$, $D(\Rat_\ell)$ does not admit any finite index sub-groups. 

This, combined with hypotheses~\eqref{hyp:sc} and~\eqref{hyp:isotropic}, implies that $\rho_G(\widetilde{G}(\Rat_\ell))$ acts trivially on $\pi_0(\mathcal{P}_{\mathsf{X},n})$ as well. Let $\varpi$ be as in hypothesis~\eqref{hyp:trivial}, and let $F_\varpi\subset \pi_0(\mathcal{P}_{\mathsf{X},n})$ be the fiber over the image of $\varpi$ in $\pi_0(S)$. By hypothesis~\eqref{hyp:hecke}, it is enough to show that $F_\varpi$ is a singleton: Any other fiber is a translate of this by an element of $G(\Adele_f^T)$.

Hypothesis~\eqref{hyp:trivial} implies that the subgroup
\[
\tilde{H}_n \coloneqq \{\varphi_n(m):\;m\in J(\Int_{(p)})\text{, $\psi(m)\in \rho_G(\widetilde{G}(\Adele_f^T))$}\}\subset H(\Int/p^n\Int)
\]
fixes the image of $\varpi$ in $\pi_0(\mathcal{P}_{\mathsf{X},n})$ (which we once again denote by $\varpi$). Here, $\varphi_n:J(\Int_{(p)})\to H(\Int/p^n\Int)$ is obtained by reduction-mod-$p^n$ from the map $\varphi$.

It is now enough to show that $\tilde{H}_n$ acts transitively on the fiber $F_\varpi\subset \pi_0(\mathcal{P}_{\mathsf{X},n})$. For this, it is enough to know that it surjects onto $\mathsf{X}(\Int/p^n\Int)$ via the map induced by $\varphi_n$. Note, however that $\tilde{H}_n$ contains $\rho_J(\widetilde{J}(\Int_{(p)}))$. Therefore, it is enough to show that the latter surjects onto $\mathsf{X}(\Int/p^n\Int)$.

First, note that the natural map
\[
\widetilde{H}_{\Int_p}/\widetilde{Q}\to H_{\Int_p}/Q
\]
is an isomorphism of fppf sheaves over $\Int_p$. Indeed, it is a monomorphism by definition, and hypothesis~\eqref{hyp:central} implies that it is also surjective.  

Now, hypothesis~\eqref{hyp:smooth} ensures that $\widetilde{H}_{\Int_p}$ and $\widetilde{Q}$ are smooth over $\Int_p$ with connected fibers. Therefore, by Lang's theorem~\cite{Lang1956-vh}, we have 
\[
\mathsf{X}(\Int/p^n\Int) = \widetilde{H}(\Int/p^n\Int)/\widetilde{Q}(\Int/p^n\Int).
\]

It now follows from~\eqref{ord:cor:weakapprox} that $\rho_J(\widetilde{J}(\Int_{(p)}))$ maps surjectively onto $\mathsf{X}(\Int/p^n\Int)$ via $\varphi_n$.
\end{proof}

\subsection{Hecke action on connected components of Shimura varieties}
Let $(G,X)$ be a Shimura datum. Write $g\mapsto g^{\ad}$ for the natural map $G\to G^{\ad}$, and set
\[
G(\Rat)_+ = \{g\in G(\Rat):\; g^{\ad}\in G(\Real)^0\}
\]
where $G(\Real)^0\subset G(\Real)$ is the topological connected component of the identity.

\subsubsection{}\label{subsubsecshim_var_setup}
Fix a compact open subgroup $K_p \subset G(\Rat_p)$. For any sufficiently small compact open $K\subset G(\Adele_f)$ of the form $K_pK^p$ with $K^p\subset G(\Adele_f^p)$, we obtain a Shimura variety $\Sh_K = \Sh_K(G,X)$ over the reflex field $E$ with
\[
\Sh_K(\Comp) = G(\Rat)\backslash (X\times G(\Adele_f)/K).
\]
We will be interested in the inverse limit
\[
\Sh_{K_p} = \varprojlim_{K^p\subset G(\Adele_f^p)}\Sh_{K_pK^p},
\]
which is a scheme over $\Rat$. There is a natural action of $G(\Adele_f^p)$ on $\Sh_{K_p}$ obtained over $\Comp$ via the right multiplication action on $G(\Adele_f)$. 

\begin{lemma}
\label{lem:hecke_action_conncomp}
Let $\pi_0(\Sh_{K_p,\overline{\Rat}})$ be the set of connected components of $\Sh_{p,\overline{\Rat}}$. Suppose that $G(\Rat))_+$ is dense in $G(\Rat_p)$; for instance, this is the case if $G$ satisfies the hypotheses of Proposition~\ref{ord:proposition:weakapprox}. Then:
\begin{enumerate}
	\item $G(\Adele_f^p)$ acts transitively on $\pi_0(\Sh_{K_p,\overline{\Rat}})$.
	\item $\rho_G(\tilde{G}(\Adele_f^p))$ acts trivially on $\pi_0(\Sh_{K_p,\overline{\Rat}})$.
\end{enumerate}
\end{lemma}	
\begin{proof}
By~\cite[(2.1.3)]{deligne:corvallis}, the set of connected components of $\Sh_K(\Comp)$ is a torsor under the group
\[
\overline{\pi}_0\pi(G)/K \defn G(\Adele_f)/\rho_G(\tilde{G}(\Adele_f))G(\Rat)_+K,
\]
where $G(\Rat)_+\subset G(\Rat)$ is the stabilizer of a connected component of $X$, the action being induced from that of $G(\Adele_f)$ on itself via right multiplication. This implies that $\pi_0(\Sh_{K_p,\overline{\Rat}})$ is a torsor under
\[
G(\Adele_f)/\rho_G(\tilde{G}(\Adele_f))G(\Rat)_+K_p.
\]

Assertion (2) follows immediately from this; and assertion (1) follows from the additional observation that, under our hypotheses, $G(\Rat)_+K_p = G(\Rat_p)$.
\end{proof}

\subsection{Kisin's analogue of Tate's theorem}

Suppose now that $(G,X)$ is of Hodge type and fix a place $v\vert p$ of $E$. 

\begin{construction}
	By Construction~\ref{const:integral_models}, for a given choice of symplectic representation $H$ with lattice $H_{\Int}$, we obtain an integral model $\Ss_K$ over $\Reg{E}$, and hence a model $\Ss_{K,(v)}$ over $\Reg{E,(v)}$. If we vary $K^p$, then we obtain a tower of such models $\Ss_{K_pK^p,(v)}$ where the transition maps $\Ss_{K_pK^p,(v)}\to \Ss_{K_p\tilde{K}^p,(v)}$ for $K^p\subset \tilde{K}^p$ are finite \'etale, and we can take the inverse limit
\[
	\Ss_{K_p,(v)} \defn \varprojlim_{K^p}\Ss_{K_pK^p,(v)}.
\]
The $G(\Adele_f^p)$-action on $\Sh_{K_p}$ extends to one on $\Ss_{K_p,(v)}$: More precisely, over $\Ss_{K_p,(v)}$, we have a canonical isomorphism of $\Adele_f^p$-local systems
\[
\epsilon^p:\underline{\Adele}_f^p\otimes_{\Rat}H\xrightarrow{\simeq}\hat{V}^p(\mathcal{A})^\vee,
\]
where on the right hand side we have the dual of the prime-to-$p$ ad\'elic Tate module of $\mathcal{A}$, and a functorial point $x$ of $\Ss_{K_p,(v)}$ gives rise to a pair $(\mathcal{A}_x,\epsilon_x)$, where $\mathcal{A}_x$ is an abelian scheme and $\epsilon_x$ is a trivialization as above of its dual prime-to-$p$ ad\'elic Tate module. For $g\in G(\Adele_f^p)$, the functorial point $g\cdot x$ will be the unique one giving rise to the pair $(\mathcal{A}_x,\epsilon_x\circ g^{-1})$.
\end{construction}

\begin{construction}
	Let the notation be as in Lemma~\ref{lem:GIsoc_struc}, but assume that $\kappa$ is the algebraic closure of a finite field $k$, and, for any $m\in \Int_{\geq 1}$, let $k_m\subset \kappa$ be the unique degree $m$ extension of $k$. Suppose also that $x_0$ is defined over $k$ and  write $x_{0,m}$ for the corresponding $k_m$-point of $\Ss_{K_pK^p}$. Let $J^\circ_{x_{0,m}} = J^\circ_{\mathcal{G}_{x_{0,m}}}$ be the algebraic group over $\Rat_p$ such that $J^{\circ}_{x_0}(\Rat_p) = \mathrm{QIsog}_G(x_{0,m},x_{0,m})$. In more detail, if $\bm{H}_{\cris,x_{0,m}}$, then the $G$-structure up to isogeny from Lemma~\ref{lem:GIsoc_struc} gives us a subgroup $\GL(\bm{H}_{x_{0,m}})$ that can be identified with $G_{W(k_m)[1/p]}$, and, for any $\Rat_p$-algebra $R$, we have
\[
J^{\circ}_{x_0}(R) = G(W(k_m)\otimes_{\Int_p}R)\cap(\underline{\Aut}^\circ_F(\bm{H}_{\cris,x_{0,m}}))(R) \subset \GL(R\otimes_{W(k_m)\bm{H}_{\cris,x_{0,m}}}),
\]
where $\underline{\Aut}^\circ_F(\bm{H}_{\cris,x_{0,m}})$ is the algebraic group over $\Rat_p$ obtained as the group scheme of invertible elements in the ring $\End_F(\bm{H}_{\cris,x_{0,m}}[1/p])$ of endomorphisms of the $F$-isocrystal $\bm{H}_{\cris,x_{0,m}}[1/p]$.

Let $\underline{\Aut}^\circ(\mathcal{A}_{x_{0,m}})$ be the algebraic group over $\Rat$ obtained as the group scheme of invertible elements in $\End(\mathcal{A}_{x_{0,m}})_{\Rat}$; then we have a natural map of $\Rat_p$-group schemes
\[
\Rat_p\otimes_{\Rat}\underline{\Aut}^\circ(\mathcal{A}_{x_{0,m}})\to \underline{\Aut}^\circ_F(\bm{H}_{\cris,x_{0,m}})
\]

We now define an algebraic group $I^{\circ}_{x_{0,m}}$ over $\Rat$ such that for any $\Rat$-algebra $R$, we have
\[
I^{\circ}_{x_{0,m}}(R) = \underline{\Aut}^\circ(\mathcal{A}_{x_{0,m}})(R)\cap J^{\circ}_{x_{0,m}}(\Rat_p\otimes_{\Rat}R)\subset \underline{\Aut}^\circ_F(\bm{H}_{\cris,x_{0,m}})(\Rat_p\otimes_{\Rat}R).
\]
\end{construction}

\begin{proposition}
\label{proposition:Aut_isom_tate}
With the notation as above:
\begin{enumerate}
	\item $I^{\circ}_{x_{0,m}}$ is a connected reductive group over $\Rat$.
	\item For $m$ sufficiently divisible, the natural map
\[
\Rat_p\otimes_{\Rat}I^{\circ}_{x_{0,m}}\to J^{\circ}_{x_{0,m}}
\]
is an isomorphism of algebraic groups over $\Rat_p$.
	\item The action of $I^{\circ}_{x_{0,m}}$ on the prime-to-$p$ ad\'elic Tate module $\widehat{V}^p(\mathcal{A}_{x_0})$, via the trivialization
	\[
    \epsilon_{x_0}:\Adele_f^p\otimes_{\Int_{(p)}}H_{(p)}\xrightarrow{\simeq}\widehat{V}^p(\mathcal{A}_{x_0})^\vee
	\]
	gives rise to an embedding
	\[
    \Adele_f^p\otimes_{\Rat}I^{\circ}_{x_{0,m}}\hookrightarrow G_{\Adele_f^p}.
	\]
\end{enumerate}
\end{proposition}
\begin{proof}
The reductivity of $I^{\circ}_{x_{0,m}}$ is a consequence of the fact that its real points are a compact Lie group modulo scalars: see (2.1.3) of~\cite{KMPS}.

We can define a subgroup $I^{\circ,p}_{x_{0,m}}\subset I^{\circ}_{x_{0,m}}$ as the largest $\Rat$-subgroup whose action on $\widehat{V}^p(\mathcal{A}_{x_0})^\vee$ gives rise, via $\epsilon_{x_0}$, to an embedding $\Adele_f^p\otimes_{\Rat}I^{\circ,p}_{x_{0,m}}\hookrightarrow G_{\Adele_f^p}$.

Then Corollary 2.2.10 of~\cite{KMPS} shows that, for $m$ sufficiently divisible, the natural map $\Rat_p\otimes_{\Rat}I^{\circ,p}_{x_{0,m}}\to J^{\circ}_{x_{0,m}}$ is an isomorphism, which \emph{a fortiori}, implies that $I^{\circ,p}_{x_{0,m}} = I^{\circ}_{x_{0,m}}$, and so verifies assertions (2) and (3).
\end{proof}

\subsection{Monodromy over the ordinary loci of Shimura varieties}
\label{subsec:monodromy_shimura}

We will now put ourselves in the situation of \S~\ref{subsec:the_ordinary_igusa_tower} so that we have the formal Igusa tower $\widehat{\mathrm{Ig}}^{\ord}_{K_pK^p,v}\to \widehat{\Ss}^{\ord}_{K_pK^p,v}$ with special fiber $\mathrm{Ig}^{\ord}_{K_pK^p,k(v)}\to \Ss^{\ord}_{K_pK^p,k(v)}$. Taking the limit over $K^p$ gives an $M_{\Int_p}(\Int_p)$-torsor
\[
	\mathrm{Ig}^{\ord}_{K_p,k(v)}\to \Ss^{\ord}_{K_p,k(v)}.
\]

\begin{proposition}
\label{proposition:chai_hida_application}
Suppose that the following conditions hold:
\begin{enumerate}
	\item\label{hypo:smooth}$M_{\Int_p}$ and $\tilde{M}_{\Int_p}$ are smooth over $\Int_p$ with connected special fiber.

	\item\label{hypo:weakapprox}$M$ admits a maximal torus splitting over a cyclic extension of $\Rat_p$ (see the hypotheses of Proposition~\ref{ord:proposition:weakapprox}).

	\item\label{hypo:Qsmooth}$Q\subset M_{\Int_p}$ is a closed subgroup scheme containing $Z_{M_{\Int_p}}$ with $\tilde{Q}$ smooth with connected special fiber.

	\item\label{hypo:conncomp}If $\overline{\Int}_p\subset \overline{\Rat}_p$ is the subring of algebraic $p$-adic integers, then the map
	\[
     \pi_0(\Ss^{\ord}_{K_p,\overline{\Field}_p})\to \pi_0(\Ss_{K_p,\overline{\Int}_p}) \simeq \pi_0(\Sh_{K_p,\overline{\Rat}_p})
	\]
	is a bijection.

	\item\label{hypo:isotropic}For every prime $\ell\neq p$, $G_{\Rat_\ell}$ is isotropic.

	\item\label{hypo:hypersymm}There exists a finite field $k\subset \overline{\Field}_p$, a compact open $K^p\subset G(\Adele_f^p)$ and $x_0\in \Ss^{\ord}_{K^pK_p}(k)$ such that the map
	\[
      \Aut^{\circ}_G(\mathcal{G}_{x_0})\to \Aut^{\circ}_G(\overline{\Field}_p\otimes_k \mathcal{G}_{x_0})
	\]
	is a bijection.
\end{enumerate}
Then for $\mathsf{X} = M_{\Int_p}/Q$, and every $n\in \Int_{\ge 1}$, the natural map
\[
\pi_0(\mathrm{Ig}^{\mathrm{ord}}_{K_p,k(v),\mathsf{X},n})\to \pi_0(\Ss_{K_p,(v)}^{\mathrm{ord}})
\]
is a bijection.
\end{proposition}
\begin{proof}
This is an application of Proposition~\ref{prop:trivial_set}. In the notation there, we will take $S = \Ss_{K_p,(v)}^{\mathrm{ord}}$, $H = J_0$, $\mathcal{P} = \widehat{\mathcal{P}}$, $G = G$, and $T = \{p\}$. 

By Lemma~\ref{lem:hecke_action_conncomp} and Assumption~\eqref{hypo:conncomp}, $\tilde{G}(\Adele_f^p)$ acts trivially on $\pi_0(\widehat{\mathcal{C}})$, which verifies Assumption~\ref{hyp:sc} from~\eqref{prop:trivial_set}. 

Assumption~\eqref{hyp:isotropic} in~\eqref{prop:trivial_set} is Assumption~\eqref{hypo:isotropic} here.

Assumption~\eqref{hyp:hecke} in~\eqref{prop:trivial_set} follows from Assumption~\eqref{hypo:conncomp} here and Lemma~\ref{lem:hecke_action_conncomp}.

Assumptions~\eqref{hyp:central} and~\eqref{hyp:smooth} in~\eqref{prop:trivial_set} follow from Assumptions~\eqref{hypo:smooth} and ~\eqref{hypo:Qsmooth} here.

To finish, it remains to find the $\Int_{(p)}$-group scheme $J_{\Int_{(p)}}$ as in~\eqref{subsubsecchai_hida_setup} with generic fiber $J$ satisfying Assumption~\eqref{hyp:trivial}. 

Let $x_0$ be as in Assumption~\eqref{hypo:hypersymm}. By replacing $k$ with a suitable extension, we can assume that $I^{\circ}_{x_0} = I^{\circ}_{x_0,1}$, $J^{\circ}_{x_0} = J^{\circ}_{x_0,1}$ are such that we have the isomorphism
\[
\Rat_p\otimes_{\Rat}I^{\circ}_{x_0}\xrightarrow{\simeq}J^{\circ}_{x_0}
\]
given to us by Proposition~\ref{proposition:Aut_isom_tate}.

We will take $J = I^{\circ}_{x_0}$ and $J_{\Int_{(p)}}$ to be the Zariski closure of $J$ in the $\Int_{(p)}$-group $\underline{\Aut}(\mathcal{A}_{x_0})_{(p)}$ obtained as the group scheme of invertible elements in the $\Int_{(p)}$-algebra $\End(\mathcal{A}_{x_0})\otimes\Int_{(p)}$. We can also interpret $J_{\Int_{(p)}}$ as the largest subgroup of $J$ acting on $\mathcal{G}_{x_0}$ via \emph{automorphisms} instead of self-quasi-isogenies.

Choose a lift $\eta\in \mathrm{Ig}^{\ord}_{K_p,k(v)}(\overline{\Field}_p)$ of $x_0$. This gives a $G$-structure preserving isomorphism $\overline{\Field}_p\otimes_{\Int_p}\mathcal{G}_0\xrightarrow{\simeq}\overline{\Field}_p\otimes_k \mathcal{G}_{x_0}$, which we also denote by $\eta$. Assumption~\ref{hypo:hypersymm} now gives us isomorphisms
\[
\varphi:J_{\Rat_p}\xrightarrow{\simeq}J^{\circ}_{x_0}\xrightarrow{\simeq}\Aut^{\circ}_G(\overline{\Field}_p\otimes_k \mathcal{G}_{x_0})\xrightarrow{\simeq}\Aut^{\circ}_G(\overline{\Field}_p\otimes_{\Int_p} \mathcal{G}_{0})\xrightarrow{\simeq}M_{\Rat_p},
\]
where the penultimate isomorphism is obtained via conjugation by $\eta$, and the last one is from Lemma~\ref{lem:aut_g_loc_const}. This isomorphism maps $\Int_p\otimes_{\Int_{(p)}}J_{\Int_{(p)}}$ onto $M_{\Int_p}$. 

Now, (3) of Proposition~\ref{proposition:Aut_isom_tate} shows that the action of $M(\Adele_f^p)$ on the prime-to-$p$ ad\'elic Tate module of $\mathcal{A}_{\overline{x}_0}$, along with the trivialization $\epsilon_{\overline{x}_0}$, gives an embedding $\psi:J_{\Adele_f^p}\hookrightarrow G_{\Adele_f^p}$.

In this way, we get a map
\[
\Phi:J(\Int_{(p)})\xrightarrow{(\varphi,\psi)} M_{\Int_p}(\Int_p)\times G(\Adele_f^p).
\]

To verify Assumption~\eqref{hyp:trivial}, it is now enough to show that for all $m\in J(\Int_{(p)})$, $\Phi(m)$ fixes $\eta$. For this, note that $\Phi(m)(\eta)$ corresponds to the same underlying abelian variety $\mathcal{A}_{\overline{x}_0}$ but with $\epsilon_{\overline{x}_0}$ replaced by $\epsilon_{\overline{x}_0}\circ \psi(m)^{-1}$ and $\eta$ replaced by the point of $\widehat{\mathrm{Ig}}^{\ord}_{K,v}$ above $\Phi(m)(\overline{x}_0)$ corresponding to the isomorphism $\varphi(m)\circ \eta$.

This means that the isomorphism 
\[
m:\mathcal{A}_{\overline{x}_0}\xrightarrow{\simeq}\mathcal{A}_{\Phi(m)(\overline{x}_0)} = \mathcal{A}_{\overline{x}_0}
\]
is $G$-structure preserving and carries $\eta$ to the isomorphism $\varphi(m)\circ \eta$. Therefore, arguing as in the proof of Corollary~\ref{cor:qisogM_ordinary_desc}, one sees that $\eta = \Phi(m)(\eta)$, thus finishing the proof of the proposition.
\end{proof}

\begin{definition}\label{ord:defn:hypersymmetric}
 We will say that $x_0$ is \defnword{hypersymmetric} if it satisfies Assumption~\eqref{hypo:hypersymm} above. The definition is originally due to Chai~\cite{Chai2011-ik} in the case where $G = \GSp(H)$.
\end{definition}

\begin{remark}
	\label{rem:hypersymmetric_one}
Suppose that $\mathcal{A}_{x_0}$ is hypersymmetric as an abelian variety; that is, suppose that the natural map
	\[
      \Int_p\otimes_{\Int}\End(\mathcal{A}_{x_0})\to \End(\overline{\Field}_p\otimes_k\mathcal{G}_{x_0}).
	\]
	is an isomorphism. Then it is immediate that $x_0$ is hypersymmetric in the sense of the definition above. In fact, it is enough to assume that $\mathcal{A}_{x_0}$ is isogenous to a hypersymmetric abelian variety: any abelian variety isogenous to a hypersymmetric one is itself hypersymmetric.
\end{remark}

\begin{remark}
\label{rem:hypersymmetric_two}
Any ordinary elliptic curve $\mathcal{E}$ is hypersymmetric in the above sense: The right hand side is $\Int_p\times \Int_p$, and so it is enough to know that $\End(\mathcal{E})$ has rank at least $2$ as a $\Int$-module (since the image of the map in question is saturated), which is clear, since Frobenius does not act as a scalar.
\end{remark}

\begin{remark}
\label{rem:hypersymmetric_CM}
	Suppose that we have an imaginary quadratic extension $L/\Rat$ and a map $T_L\defn \Res_{L/\Rat}\Gm\to G$ whose real fiber yields an element of $X$. Suppose also that $L$ is split at $p$ and that $G_{\Int_p}(\Int_p)\subset G(\Rat_p)$ contains the image of $(\Reg{L}\otimes_{\Int}\Int_p)^\times$. Then the mod-$v$ reduction of the CM points arising from $T_L$ are all hypersymmetric points of $\Ss^{\ord}_{K_p,k(v)}$ contains many hypersymmetric points: Indeed, the symplectic representation $H$, viewed as a representation of $T_L$, must be isomorphic to a direct sum of the tautological representation on $L$, for weight reasons. Therefore, if $x_0$ is a point of $\Ss_K$ that is the reduction of a special point arising from $T_L$, then $\mathcal{A}_{x_0}$ is isogenous to a power of the CM elliptic curve associated with that point, which is ordinary, since we have assumed that $p$ is split in $L$. By Remark~\ref{rem:hypersymmetric_two}, $x_0$ is hypersymmetric.
\end{remark}

\begin{remark}
	In~\cite{MR4419972}, one finds a somewhat general criterion for when a Newton stratum in a Shimura variety at a place of hyperspecial level contains a hypersymmetric point. Specialized to the ordinary case, it says that the ordinary locus contains a hypersymmetric point precisely when $G$ admits a subgroup $I\subset G$ that is the centralizer of an elliptic element and whose Dynkin diagram, when viewed as a $\Gal(\overline{\Rat}_p/\Rat_p)$-equivariant graph, is isomorphic to that of $M_{\Int_p}$.
\end{remark}

\section{Group schemes associated with quadratic lattices}
\label{sec:group_schemes}

In this section, we will prove some technical results about group schemes associated with quadratic lattices that will be employed to prove our main irreducibility results. The reader can refer back to the results here as necessary.

\subsection{Applications of Witt's extension theorem}
\label{subsec:applications_of_witt_s_extension_theorem}
Fix a \defnword{self-dual} quadratic space $(N,Q)$ over $\Int_p$: This is a quadratic form 
\[
Q:N\to \Int_p
\]
on a finite free $\Int_p$-module $N$ that is such that the associated bilinear form 
\[
[x,y]_Q = Q(x+y) - Q(x)- Q(y)
\]
on $N$ is non-degenerate. Note that when $p=2$ this forces $n$ to be even.

We have the reductive $\Int_p$-group scheme $\mathrm{GSpin}(N)$, sitting in two short exact sequences of reductive groups
\begin{align*}
1\to \Gm&\to \GSpin(N)\to \SO(N)\to 1;\\
1\to \mathrm{Spin}(N)&\to \GSpin(N)\xrightarrow{\nu}\Gm\to 1.
\end{align*}
Here, $\nu: \mathrm{GSpin}(N)\to \Gm$ is the spinor norm. For more details, see~\cite[\S~1]{mp:reg}.

In the first part of this section, we will see that the various lemmas from \S~2 of \cite{mp:reg}hold in quite some generality, without in particular the hypothesis that $p>2$.

\begin{lemma}
\label{lem:GSpin_torsor_points}
Suppose that $F$ is a field over $\Int_p$. Suppose that we have two proper direct summands $W_1,W_2\subset N_F$ such that there is an isometry $f:W_1\xrightarrow{\simeq}W_2$ of quadratic spaces (with their inherited quadratic forms) over $F$. Suppose that $W_1$ has codimension at least $2$ in $N_F$. Then there exists $h\in \SO(N)(F)$ such that $h(w_1) = f(w_1)$ for all $w_1\in W_1$.
\end{lemma}
\begin{proof}
Let us make some preliminary observations:
\begin{itemize}
	\item If $W_1$ is isotropic, then the lemma reduces to the transitivity of the action of $\SO(N)(F)$ on the Grassmannian parameterizing isotropic subspaces of $N_F$ of fixed rank.
	\item By Witt's extension theorem~\cite[\S 4, Th\'eor\'eme 1]{Bourbaki2006-qu}, there exists an element $g\in \mathrm{O}(N)(F)$ such that $g(w_1) = f(w_1)$ for all $w_1\in W_1$. In particular, if we knew that there is an element of $\mathrm{O}(N)(F)\backslash \mathrm{SO}(N)(F)$ restricting to the identity on $W_1$, then we would be done. 
\end{itemize}

Now, suppose that $W_1$ is itself self-dual with the inherited quadratic form. Then we have $N= W_1 \oplus W_1^{\perp}$  where
\[
W_1^{\perp} = \{n\in N_F:\;[n,w] = 0\text{ for all $w\in W_1$ }\}
\]
is also self-dual. We can now pick any element of $\mathrm{O}(W_1^{\perp})(F)\backslash\SO(W_1^{\perp})(F)$ to extend the identity on $W_1$ to an element $g'\in \mathrm{O}(N)(F)\backslash \mathrm{SO}(N)(F)$.

Next, suppose that $W_1 = F\cdot u_1$ where $Q(u_1) \neq 0$. After scaling if necessary we can assume that $Q(u_1) = 1$. If $2$ is invertible in $F$, then this is a special case of the second paragraph. Assume therefore that $F$ has charactersitic $2$: in this case $\dim N_F>2$ is even, and we can find $e_1\in N_F$ such that $Q(e_1) = 0$ and $[e_1,u_1]_Q = 1$, so that $u_1 = e_1+e'_1$, where $Q(e'_1) = 0$ and $[e_1,e'_1]_Q = 1$. Now, $e_1,e'_1$ span a self-dual proper subspace of $N_F$, and we return to the situation from the previous paragraph. 

Suppose now that $W_1$ is \emph{non-degenerate}: this means that the projective quadric defined by the restriction of $Q$ to $W_1$ is a smooth $F$-scheme. If $2$ is invertible in $F$, then this is equivalent to saying that $W_1$ is self-dual, and we return to a previously considered case. Otherwise, we have the possibility that there exists $u_1\in W_1$ such that $Q(u_1) = 1$, and such that $(F\cdot u_1)^\perp = W_1$. Choose a direct sum decomposition $W_1 = F\cdot u_1 \oplus V_1$, and set $V_2 = f(V_1)\subset W_2$. Then $V_1\subset N_F$ is self-dual, and so by the second paragraph of the proof, there is an $h'\in \SO(N)(F)$ with $h'(v_1) = f(v_1)$ for all $v_1\in V_1$. If $h'(u_1) = f(u_1)$, then we are done. Otherwise, note that $V_2^{\perp}$ has dimension at least $4$. Therefore, by the previous paragraph, we can find $h''\in \SO(V_2^{\perp})(F)$ such that $h''(h'(u_1)) = f(u_1)$. We now take $h$ to be the composition of $h'$ with the element of $\SO(N)(F)$ that restricts to the identity on $V_1$ and to $h''$ on $V_1^{\perp}$.

Finally, let $W_1$ be arbitrary, and, for $i=1,2$, let $V_i\subset W_i$ be the subspace consisting of the isotropic vectors in the radical $W_i\cap W_i^{\perp}$. Then we have $f(V_1) = V_2$. Choose Witt decompositions
\[
N_F = V_1 \oplus U_1 \oplus V'_1 = V_2 \oplus U_2 \oplus V'_2.
\]
Note that we have $W_i = V_i \oplus (U_i\cap W_i)$ for $i=1,2$ and that $U_i\cap W_i\subset U_i$ is a non-degenerate subspace of a self-dual quadratic space. With this, we can reduce to the situation in the previous paragraph.
\end{proof}

\begin{lemma}
\label{lem:GSpin_surj_W}
Suppose that $R$ is a $\Int_p$-algebra, and that $W\subset N_R$ is a direct summand. For any $R$-algebra $S$ set
\[
A_W(S) = \{\varphi\in \Hom_S(W_S,N_S):\;[\varphi(w),w]_Q =0\text{, for all $w\in W_S$}\}
\]
Then: 
\begin{enumerate}
	\item $A_W(R)$ is locally free over $R$;
	\item For any $R$-algebra $S$, we have 
	\[
   S\otimes_RA_W(R) =A_W(S)\subset \Hom_S(W_S,N_S).
	\]
\end{enumerate}
\item The map
\[
S\otimes_{\Int_p}\Lie \SO(N) \xrightarrow{X\mapsto X\vert_{W_S}}A_W(S)
\]
is surjective. 
\end{lemma}
\begin{proof}
Consider the surjective map of finite locally free $S$-modules:
\[
\pi_Q:\Hom_S(W_S,N_S)\to \Hom_S(W_S,W_S^\vee)
\]
induced by the dual surjection $N_S\to W_S^\vee$. Then we see that $A_W(S)$ is the pre-image under $\pi_Q$ of the locally free sub-module consisting of maps $\varphi$ such that $\varphi(w)(w) = 0$ for all $w\in W_S$.\footnote{Concretely, if $W_S$ is free of rank $n$, then we can identify this sub-module with the space of $n\times n$ anti-symmetric matrices with zeros along the diagonal.} From this, the first two assertions of the lemma are immediate.

For the third and final assertion, it now suffices to prove it under the hypothesis that $R = F$ is an algebraically closed field, where we are essentially in the situation of Lemma 2.2 of~\cite{mp:reg}. The only modification one needs in that proof is to consider the case where $W\subset N_F$ is non-degenerate but not self-dual---once again, a characteristic $2$ phenomenon---which, as in the proof of Lemma~\ref{lem:GSpin_torsor_points} above, reduces to the case where $W = F\cdot u$ with $Q(u) = 1$. That is, given a vector $v\in W^\perp\subset N_F$, we must find $X\in F\otimes_{\Int_p}\Lie\SO(N)$, such that $X(u) = v$. For this, we can assume that $v$ as in \emph{loc. cit.}, we can assume that $u = e+f$, where $e,f$ are isotropic vectors spanning a hyperbolic plane $U\subset N_F$; then 
\[
\varphi:U\xrightarrow{e\mapsto e+v-u\;;\;f\mapsto f}N_F
\]
is an  element of $A_U(F)$ satisfying $\varphi(u) = v$, and so we reduce the requisite surjectivity statement to the case where $W$ is itself self-dual, which is covered by the argument in~\cite{mp:reg}.
\end{proof}

\begin{lemma}
\label{lem:GSpin_torsor_lifts}
Suppose that we have a surjection $R\to \overline{R}$ of $\Int_p$-algebras with square-zero kernel $I$. Suppose that we have direct summands $W_1,W_2\subset N_R$ such that there is an isomorphism $f:W_1\xrightarrow{\simeq}W_2$ of quadratic spaces (with their inherited quadratic forms)  over $R$. Suppose also that there exists $g'\in \SO(N)(R)$ such that
\[
g'(w_1)-f(w_1)\in I\otimes_{\Int_p}N\subset N_R\text{, for all $w_1\in  W$}.
\]
Then there exists $g\in \SO(N)(R)$ such that $g(w_1) = f(w_1)$ for all $w_1\in W_1$.
\end{lemma}
\begin{proof}
This is shown exactly as in~\cite[Lemma 2.8]{mp:reg}: For $G=\SO(N)$, the point is to show that there exists $X\in I\otimes_{\Int_p}\Lie G$ such that $X(w_1) = w_1 - g'^{-1}f(w_1)$ for all $w_1\in W_1$ using Lemma~\ref{lem:GSpin_surj_W} above, and to then replace $g'$ with $g'\circ (1-\tilde{X})$, where $\tilde{X}\in I\otimes_{\Int_p}\Lie G$ is a lift of $X$. 
\end{proof}

\begin{lemma}\label{lem:smooth_fiber}
Let $N_0$ be a non-degenerate quadratic space over a field $k$, and let $W_0\subset N_0$ be a $k$-subspace of codimension at least $2$. Define subgroups
\[
Q_{W_0}\subset \mathrm{GSpin}(N_0)\;;\; \tilde{Q}_{W_0}\subset \mathrm{Spin}(N_0)\;;\; \bar{Q}_{W_0}\subset \SO(N_0)
\]
as above. Then $Q_{W_0},\tilde{Q}_{W_0},\bar{Q}_{W_0}$ are smooth connected $k$-algebraic groups.
\end{lemma}
\begin{proof}
It is enough to show this for $\tilde{Q}_{W_0}$: Indeed, we have short exact sequences of group schemes:
\begin{align*}
1\to \Gm&\to Q_{W_0}\to \overline{Q}_{W_0}\to 1;\\
1\to \tilde{Q}_{W_0}&\to Q_{W_0}\xrightarrow{\nu}\Gm\to 1.
\end{align*} 
The surjectivity of $\nu\vert_{Q_{W_0}}$ follows from the surjectivity of its restriction to the central $\Gm$ in $Q_{W_0}$.

Let $W'_0\subset W_0$ be the radical for the restriction of the symmetric bilinear form to $W_0$, let $V_0\subset W'_0$ be the subspace consisting of all the isotropic vectors (we have $V_0 = W'_0$  unless $k$ has characteristic $2$), and let $U_0 = (V_0)^\perp$ be its orthogonal complement. Let $P_0\subset \mathrm{Spin}(N_0)$ be the parabolic subgroup stabilizing $V_0$; its Levi quotient can be identified with
\[
\GL(V_0)\times \Spin(U_0/V_0).
\]
 Let $M_0$ of the Levi quotient consisting of elements of the form $(1,g)$, where $g$ restricts to the identity on $W_0/V_0$, and let $P'_0\subset P_0$ be its pre-image. Note that $M_0 = \mathrm{Spin}(W_0/V_0)$, where $W_0/V_0$ is a non-degenerate quadratic space. In particular, it is connected, and so is its pre-image $P'_0$.

We have $P'_0 \simeq U(P_0)\rtimes M_0$, where $U(P_0)\subset P_0$ is the unipotent radical. It can now be checked that, under any such isomorphism, $\tilde{Q}_{W_0}$ is mapped isomorphically onto the (smooth, connected) subgroup 
\[
V_0\rtimes M_0\subset U(P_0)\rtimes M_0
\]
where $V_0\subset U(P_0)$ consists of those elements $g$ whose restriction to $W_0$ is the identity.
\end{proof}

\subsubsection{}\label{subsubsecqW}
Suppose that we have a direct summand $W\subset N$. Let $Q_W\subset \mathrm{GSpin}(N)$ be the $\Int_p$-group scheme such that for every $\Int_p$-algebra $R$, we have
\[
Q_W(R) = \{g\in \GSpin(N)(R):\;g\cdot w = w\text{, for all $w\in W_R$}\}.
\]
Set $\tilde{Q}_W = \mathrm{Spin}(N)\cap Q_W = \ker \nu\vert_{Q_W}$, and let $\overline{Q}_W \subset \mathrm{SO}(N)$ be the image of $Q_W$. As an immediate consequence of Lemma~\ref{lem:smooth_fiber} we obtain:

\begin{lemma}
\label{lem:qT_smooth}
Suppose that $\mathrm{rank}(W)\leq \mathrm{rank}(N)-2$. Then:
\begin{enumerate}
	\item  We have two short exact sequences of group schemes:
\begin{align*}
1\to \Gm&\to Q_W\to \overline{Q}_W\to 1;\\
1\to \tilde{Q}_W&\to Q_W\xrightarrow{\nu}\Gm\to 1.
\end{align*} 
	\item The group schemes $Q_W$ and $\overline{Q}_W$ are smooth over $\Int_p$ with geometrically connected fibers.
	\item The map $\nu\vert_{Q_W}$ is smooth and so $\tilde{Q}_W$ is also smooth over $\Int_p$.
\end{enumerate}
\end{lemma}

\begin{lemma}\label{lem:GSpin_torsor}
Suppose that $R$ is a $\Int_p$-algebra, that $W$ has codimension at least $2$ in $N$, and that we have another embedding of quadratic spaces over $R$, $j:W_R\hookrightarrow N_R$ mapping onto a direct summand. Then the functor on $R$-algebras
\[
S\mapsto \{g\in \GSpin(N)(S):\;g^{-1}(w) = j(w)\text{ for all $w\in W_S$}\}
\]
is represented by a $Q_W$-torsor over $\Spec R$. In particular, if $R = \Int_p$, then we always have $g\in \GSpin(N)(\Int_p)$ such that $g^{-1}(w) = j(w)$ for all $w\in W$.
\end{lemma}
\begin{proof}
Clearly, the functor is represented by a closed subscheme of $\GSpin(N)_R$. Moreover, $Q_W$ acts on this subscheme by left multiplication, and this action is simply transitive whenever the $S$-points of the subscheme are non-empty. 

It is now enough to show that the functor admits sections over any strictly henselian local ring of $R$. But this is clear from Lemmas~\ref{lem:GSpin_torsor_points} and~\ref{lem:GSpin_torsor_lifts}, which together show that the subscheme in question is faithfully flat and smooth over $R$.

The second assertion now from Lang's theorem~\cite{Lang1956-vh}, combined with the fact that $Q_W$ is smooth with connected special fiber: this implies that every $Q_W$-torsor over $\Spec \Int_p$ admits a section.
\end{proof}

\subsection{Certain subspaces of quadric Grassmannians}
\label{subsec:certain_subspaces_of_quadric_grassmannians}

\begin{notation}
	Let $\{\lambda_0\}$ be the conjugacy class of minuscule cocharacters of $\GSpin(N)$ characterized by the following properties:
\begin{itemize}
	\item For the left multiplication action on the Clifford algebra $C(N)$, $\lambda_0$ has weights $0,-1$.
	\item Via the action of $\lambda_0$, $N$ acquires a weight space decomposition of the form:
\[
N = N^1 \oplus N^0 \oplus N^{-1},
\]
where $N^{\pm 1}$ are complementary isotropic lines and $N^0\subset N$ is the subspace orthogonal to both. 
\end{itemize}
\end{notation}

\begin{definition}
Let $\mathrm{Par}_{\lambda_0}$ be the Grassmannian parameterizing isotropic lines in $N$. Given a $\Int_p$-algebra $R$, we will call an isotropic line $J\subset N_R$ $W$-\defnword{generic} if the following conditions hold:
\begin{itemize}
	\item $J+W_R\subset N_R$ is a local direct summand of rank $\mathrm{rank}(W)+1$; equivalently, $J$ maps isomorphically onto a local direct summand of $N_R/W_R$.
	\item The map
	\[
     W_R \xrightarrow{t\mapsto [t,\cdot]_Q}\Hom_R(J,R)
	\]
	is surjective.
\end{itemize}
$W$-generic isotropic lines are parameterized by an open subscheme $\mathrm{Par}^{\circ}_{\lambda_0}(W)$ of $\mathrm{Par}_{\lambda_0}$
\end{definition}

\begin{definition}
\label{defn:W_generic_type_U}
	Let $\mathbb{P}(W)$ be the projective space over $\Int_p$ parameterizing hyperplanes in $T$. Then there is a natural map $\mathrm{Par}^{\circ}_{\lambda_0}(W)\to \mathbb{P}(W)$ sending an isotropic line $J\in \mathrm{Par}^{\circ}_{\lambda_0}(W)(R)$ to the kernel of the associated surjection from $W_R$ to $\Hom_R(J,R)$. Fix $U\in \mathbb{P}(W)(\Int_p)$, and let
\[
\mathrm{Par}^{\circ}_{\lambda_0}(W,U)\subset \mathrm{Par}^{\circ}_{\lambda_0}(W)
\]
be the fiber above it. 
\end{definition}

\begin{lemma}\label{lem:parT_smooth}
$\mathrm{Par}^{\circ}_{\lambda_0}(W,U)$ is a smooth scheme over $\Int_p$.
\end{lemma}
\begin{proof}
Note that we can identify this scheme with the open subscheme of the projective $\Int_p$-scheme $\mathrm{M}^{\mathrm{loc}}$ parameterizing isotropic lines in $U^{\perp}\subset N$; see~\cite[(2.10)]{mp:reg}. The singular locus of any geometric fiber over a field $k$ of this scheme can be identified with the projective space of isotropic lines contained in the radical of $U_k$; but no line arising from a point of $\mathrm{Par}^{\circ}_{\lambda_0}(W,U)(k)$ can be contained in $U_k\subset W_k$, given that it has to be complementary to $W_k$.
\end{proof}

\subsubsection{}

We will assume now that the quadratic form induced on $W_{\Rat_p}$ is non-degenerate, and that $\mathrm{Par}^{\circ}_{\lambda_0}(W,U)(\Int_p)$ is non-empty, and we will choose a point $J_0$ in it. Let $H_{\Int_p}\subset \mathrm{GSpin}(N)$ be the closed $\Int_p$-subgroup scheme with
\[
H_{\Int_p}(R) = \{g\in \mathrm{GSpin}(N)(R):\;g\cdot w = w\text{, for all $w\in W_R$}\}
\]
for any $\Int_p$-algebra $R$. Let $Q\subset H_{\Int_p}$ be the stabilizer of $J_0$: Since $Q$ has to preserve the non-degenerate pairing of $J_0$ against $W$, it follows that $Q$ actually fixes $J_0$ pointwise, and is thus the pointwise stabilizer of $J_0\oplus W\subset N$. In other words, in the notation of~\eqref{subsubsecqW}, we have $H_{\Int_p} = Q_W$ and $Q = Q_{J_0\oplus W}$.

\begin{lemma}
\label{lem:smooth_grp_schemes}
Let $\widetilde{Q}$ and $\widetilde{H}_{\Int_p}$ be defined as above Proposition~\ref{prop:trivial_set}. Suppose that $\mathrm{rank}(W)\leq \mathrm{rank}(N)-3$.
\begin{enumerate}
	\item $H_{\Int_p}$ is a smooth $\Int_p$-group scheme with reductive generic fiber.
	\item $H_{\Rat_p}$ admits a maximal torus that splits over a quadratic extension of $\Rat_p$.
	\item $Q$ is also a smooth $\Int_p$-group scheme.
	\item The quotient $H_{\Int_p}/Q$ is represented by a smooth scheme over $\Int_p$, and there is an isomorphism
	\[
    H_{\Int_p}/Q\xrightarrow{\simeq}\mathrm{Par}^{\circ}_{\lambda_0}(W,U)
	\]
	of smooth $\Int_p$-schemes.
   \item The $\Int_p$-group schemes $\widetilde{H}_{\Int_p}$ and $\widetilde{Q}$ are smooth with connected special fibers.
   \item Let $\nu:\mathrm{GSpin}(N)\to \mathbb{G}_m$ be the spinor norm; then its restriction to $Q$ is surjective. In particular, the map on $\Field_p$-points
   \[
   \nu:Q(\Field_p)\to \Field_p^\times
   \]
   is surjective.
\end{enumerate}
\end{lemma}
\begin{proof}
Assertions (1), (3), (5), and the first part of assertion (6) are immediate from Lemma~\ref{lem:qT_smooth}. Note that the reductivity of the generic fiber $H_{\Rat_p}$ is a consequence of the hypothesis that $W_{\Rat_p}$ is a non-degenerate quadratic space, so that we can identify
\[
H_{\Rat_p} = \GSpin(W_{\Rat_p}^\perp)\subset \GSpin(N_{\Rat_p}).
\]

As for (2), note that $W_{\Rat_p}^\perp$ is isometric to the orthogonal sum of an anisotropic quadratic space of dimension $\le 4$ and copies of the hyperbolic plane. In particular, it suffices to show (2) under the additional assumption that $V \defn W_{\Rat_p}^\perp$ is anisotropic. Here, an easy case-by-case analysis using the classification from~\cite[\S 25]{shimura:quadratic} does the job.

The second part of (6) is immediate from (5) and Lang's theorem on connected groups over finite fields~\cite{Lang1956-vh}.

It only remains to prove (4), which, by Lemma~\ref{lem:parT_smooth}, comes down to showing that the map
\[
H_{\Int_p}\xrightarrow{g\mapsto g\cdot J_0} \mathrm{Par}^{\circ}_{\lambda_0}(W,U)
\]
induces an isomorphism of fppf sheaves $H_{\Int_p}/Q\xrightarrow{\simeq} \mathrm{Par}^{\circ}_{\lambda_0}(W,U)$. This follows from Lemma~\ref{lem:GSpin_torsor}, which shows that, for any $\Int_p$-algebra $R$, and any $J\in \mathrm{Par}^{\circ}_{\lambda_0}(W,U)(R)$, the scheme of sections $g$ of $H_{\Int_p}(S)$ for $R$-algebras $S$ such that
\[
g(S\otimes_RJ_0) = J
\]
is a torsor over $\Spec R$ under the group scheme $Q$.
\end{proof}

	

\begin{remark}
	\label{rem:quadric_desc_par_lambda_0}
The map
\[
	\GSpin(N)(\Int_p)\xrightarrow{h\mapsto h\lambda_0(p)\GSpin(N)(\Int_p)}\GSpin(N)(\Int_p)\lambda_0(p)\GSpin(N)(\Int_p)/\GSpin(N)(\Int_p)
\]
yields a bijection
\[
	\GSpin(N)(\Int_p)/(\GSpin(N)(\Int_p)\cap \lambda_0(p)\GSpin(N)(\Int_p)\lambda_0(p)^{-1})\xrightarrow{\simeq}\GSpin(N)(\Int_p)\lambda_0(p)\GSpin(N)(\Int_p)/\GSpin(N)(\Int_p).
\]
On the other hand, using the minusculeness of $\lambda_0$, one finds that 
\[
\GSpin(N)(\Int_p)\cap \lambda_0(p)\GSpin(N)(\Int_p)\lambda_0(p)^{-1}\subset \GSpin(N)(\Int_p)
\]
is simply the pre-image of the $\Field_p$-points of the parabolic subgroup of $\GSpin(N)$ fixing the isotropic line $N^1$. In particular, we have a bijection
\[
\GSpin(N)(\Int_p)/(\GSpin(N)(\Int_p)\cap \lambda_0(p)\GSpin(N)(\Int_p)\lambda_0(p)^{-1})\xrightarrow[\simeq]{h\mapsto hN^1_{\Field_p}}\mathrm{Par}_{\lambda_0}(\Field_p),
\]
where $\mathrm{Par}_{\lambda_0}$ is the quadric Grassmannian parameterizing isotropic lines in $N$.
\end{remark}

\begin{remark}
	\label{rem:lat_descp_lambda_0}
Let $\mathrm{Lat}$ be the set of self-dual $\Int_p$-lattices in $N[p^{-1}]$. Consider the map
\[
	\GSpin(N)(\Int_p)\xrightarrow{g\mapsto g\lambda_0(p)N}\mathrm{Lat}.
\]
Its image is exactly the set $\mathrm{Lat}_{\lambda_0}$ of lattices of the form $hN$ for $h\in \GSpin(N)(\Int_p)\lambda_0(p)\GSpin(N)(\Int_p)$ and is in bijection with $\GSpin(N)(\Int_p)\lambda_0(p)\GSpin(N)(\Int_p)/\GSpin(N)(\Int_p)$. By the discussion in Remark~\ref{rem:quadric_desc_par_lambda_0}, we find that there is a bijection 
\begin{align}\label{eqn:lambda_0_lattice_par_bij}
	\mathrm{Par}_{\lambda_0}(\Field_p)\xrightarrow{\simeq}\mathrm{Lat}_{\lambda_0}
\end{align}
This bijection is in fact canonical and can be described as follows:  Suppose that we have an isotropic line $\bar{N}(-1)\in \mathrm{Par}_{\lambda_0}(\Field_p)$. Pick any isotropic lift $N(-1)\subset N$ of $\bar{N}(-1)$ and choose a complementary line $N(1)\subset N$. The lattice
\[
\tilde{N}  = p^{-1}N(-1)\oplus N(0) \oplus pN(1),
\]
where $N(0) = N(-1)^\perp \cap  N(1)^\perp$ is now the associated point in $\mathrm{Lat}_{\lambda_0}$.
\end{remark}

\begin{lemma}
\label{lem:nice_cochar}
Suppose that we have another $\Int_p$-lattice $\tilde{W}\subset W$ with $[W:\tilde{W}] = p$. Choose any hyperplane $U\in \mathbb{P}(W)(\Int_p)$ such that $\tilde{W} = U + pW$. Suppose that we have
\[
\mathrm{rank}(W)\leq \frac{\mathrm{rank}(N)-3}{2}.
\]
Then: 
\begin{enumerate}
	\item $\mathrm{Par}^{\circ}_{\lambda_0}(W,U)(\Field_p)$ is non-empty\footnote{Note that this set only depends on the image of $U$ in $N_{\Field_p}$, which is simply the image of $\tilde{W}$ and so is independent of the choice of $U$.} and, via~\eqref{eqn:lambda_0_lattice_par_bij}, maps onto the set of self-dual $\Int_p$ lattices $\tilde{N}\subset N[p^{-1}]$ such that $\tilde{N}\cap W[p^{-1}] = \tilde{W}\subset W[p^{-1}]$.

  \item Given $\tilde{N}$ as in (1), $N\subset N[p^{-1}]$ is the \emph{unique} self-dual lattice such that
  \begin{itemize}
  	\item $N = h\tilde{N}$ for some $h\in \GSpin(N)(\Int_p)\lambda_0(p)^{-1}\GSpin(N)(\Int_p)$.
  	\item $N\cap W[p^{-1}] = W\subset W[p^{-1}]$
  \end{itemize}
\end{enumerate}
\end{lemma}

\begin{proof}
In this proof, if $Y$ is a $\Int_p$-module, we will write $\overline{Y}$ for the $\Field_p$-vector space $Y\otimes_{\Int_p}\Field_p$.

We will need the following observation: If $Y$ is a self-dual quadratic space over $\Int_p$ of rank $r\geq 3$, then $\overline{Y}$ is isotropic. This means that the space of isotropic lines in $\overline{Y}$ is birationally isomorphic to $\mathbb{P}^{r-2}_{\Field_p}$, and so admits many $\Field_p$-rational points. These in turn can be lifted to isotropic lines in $Y$ via Hensel's lemma. Moreover, these lines can be chosen to lie outside the quadric associated with any proper direct summand of $Y$.

Fix a direct sum decomposition
\[
W = U \oplus \Int_p w_0,
\]
so that
\[
\tilde{W} = U\oplus \Int_p pw_0.
\]
Let $V_1 = U^\perp\subset N$ be the orthogonal complement to $U$. The radical of $\overline{V}_1$ is isomorphic to that of $\overline{U}$, which, by our hypotheses, shows that $\overline{V}_1$ admits a non-degenerate quadratic subspace of rank at least $3$. In particular, by the second paragraph, we can find an isotropic line $N(-1)\subset N$ that is orthogonal to $U$, has trivial intersection with $W$, and such that there exists $y\in N(-1)$ with
\[
[y,w_0]_Q\in \Int_p^\times.
\]
Any such $N(-1)$ corresponds to a point in $\mathrm{Par}^{\circ}_{\lambda_0}(W,U)(\Int_p)$, which is therefore non-empty.

Let $\tilde{N}\in \mathrm{Lat}_{\lambda_0}$ be associated with an isotropic line $\bar{N}(-1)\in \mathrm{Par}_{\lambda_0}(\Field_p)$. We need to show that we have $\tilde{N}\cap W[p^{-1}] = \tilde{W}$ precisely when $\bar{N}(-1)$ belongs to $\mathrm{Par}^{\circ}_{\lambda_0}(W,U)(\Field_p)$. 

Suppose first that $\bar{N}(-1)$ is in $\mathrm{Par}^{\circ}_{\lambda_0}(W,U)(\Field_p)$; then we can lift it to a line $N(-1)$ in $\mathrm{Par}^{\circ}_{\lambda_0}(W,U)(\Int_p)$, and, using the observation from the second paragraph once again, we can find an isotropic line $N(1)\subset N$ that is orthogonal to $W$, and which is complementary to $N(-1)$. 

Let $w_0\in W$ be as above. Then one sees that $N(-1)$ must pair non-degenerately with $w_0$, and so we have $w_0(1)\notin pN(1)$, where $w_0(i)$ is the component of $w_0$ in $N(i)$. Moreover, by construction, we have
\[
W\subset N(1)^\perp = N(0)\oplus N(1)\;; U\subset N(-1)^\perp\cap N(1)^\perp = N(0).
\]
This shows that we must have
\[
\tilde{N}\cap W[p^{-1}] = (N(0)\oplus pN(1))\cap W = U\oplus \Int_p \cdot pw_0 = \tilde{W}\subset W.
\]

For the converse, suppose that $\tilde{N}\cap W[p^{-1}] = \tilde{W}$. Then in particular, we find that, given a decomposition $N = N(-1) \oplus N(0)\oplus N(1)$ as above with
\[
\tilde{N} = p^{-1}N(-1)\oplus N(0) \oplus pN(1),
\]
we must have $U = U[p^{-1}]\cap \tilde{N}$, and $w_0(1)\in N(1)\backslash pN(1)$ (where, once again, this is the projection of $w_0$ onto $N(1)$). The first condition implies that $\overline{U}$ is contained in $\bar{N}(-1)\oplus \bar{N}(0)$, but does not contain $\bar{N}(-1)$ (the latter because $p^{-1}u\notin \tilde{N}$ for $u\in U\backslash pU$). This, together with the second condition, shows that $\bar{N}(-1)$ belongs to $\mathrm{Par}^\circ_{\lambda_0}(W,U)(\Field_p)$.

For assertion (2), suppose that we have a decomposition
\[
\tilde{N} = \tilde{N}(-1)\oplus \tilde{N}(0)\oplus \tilde{N}(1)
\]
arising from a cocharacter $\lambda:\Gm\to \GSpin(\tilde{N})$ in the conjugacy class ${\lambda_0}$ such that $\lambda(p)^{-1}\tilde{N}\cap W[p^{-1}] = W$. Let $\tilde{w}_0(i)$ be the projection of $pw_0$  onto $\tilde{N}(i)$. Then since $p^{-1}\tilde{w}_0(i)\in p^{-i}\tilde{N}(i)$ by hypothesis, we see that
\[
\tilde{w}_0(-1)\in p^2\tilde{N}(-1)\;;\;\tilde{w}_0(0)\in p\tilde{N}(0)\;;\; \tilde{w}_0(1)\in \tilde{N}(1).
\] 
This shows that the image of $\tilde{N}(1)$ in $\tilde{N}_{\Field_p}$ can only be the isotropic line spanned by the image of $pw_0$. This means that $\tilde{\lambda}(p)^{-1}\tilde{N} = N$.
\end{proof}

\section{Cycles on GSpin Shimura varieties}\label{sec:main}

In this section, we apply the above considerations to the special case of GSpin Shimura varieties associated with quadratic spaces over $\Rat$, and show that the certain irreducible special cycles in their generic fibers continue to have irreducible reduction over $\overline{\Field}_p$. Combined with the methods of~\cite{mp:tatek3}, this yields a proof of the irreducibility of the moduli of primitively polarized K3 surfaces of fixed degree---and that of lattice polarized K3 surfaces---in any characteristic.

\subsection{Special cycles on orthogonal and GSpin Shimura varieties}

We review the story of special cycles on GSpin (and orthogonal) Shimura varieties associated with quadratic spaces over $\Rat$ of signature $(n,2)$. Details can be found in~\cite{mp:reg},~\cite{Howard2020-fa}, though the presentation here most closely hews to that found in~\cite[\S 2.1,2.2,3.2]{HMP:mod_codim}.

\subsubsection{}
 The starting point is a quadratic space $(V,Q)$ over $\Rat$ with signature $(n,2)$ for some $n\geq 4$. The quadratic form $Q$ gives rise to a symmetric pairing
\[
[x,y]_Q = Q(x+y)  - Q(x) - Q(y)
\]
on $V$.

Associated with this is the reductive group $G = \GSpin(V)$ over $\Rat$, as well as a Hermitian symmetric domain $X$ that parameterizes the space of oriented negative definite planes in $V_{\Real}$. The pair $(G,X)$ is a Shimura datum of Hodge type with reflex field $\Rat$; a choice of symplectic representation is given by the Clifford algebra $H \coloneqq C(V)$, on which $G$ acts via left multiplication.

We will assume that the quadratic space has been chosen so that it admits a lattice $V_{\Int}\subset V$ on which the quadratic form $Q$ is $\Int$-valued and is such that the completion at $p$, $V_{\Int_p}\subset V_{\Rat_p}$ is a \emph{self-dual} lattice.  In this situation, $G_{\Rat_p}$ admits a reductive model
\[
G_{\Int_{(p)}} = \GSpin(V_{\Int_{(p)}}).
\]

Associated with $V_{\Int}$, we have a compact open subgroup $K_{V_{\Int}} = \prod_\ell K_{V_{\Int_\ell}}\subset G(\Adele_f)$, where, for each prime $\ell$, $K_{V_{\Int_\ell}}\subset G(\Rat_\ell)$ is the largest compact open subgroup contained in $C(V_{\Int_\ell})^\times$ and acting trivialy on $V_{\Int_\ell}^\vee/V_{\Int_\ell}$.

\begin{construction}
[Integral models]
\label{const:integral_models}
For a compact open subgroup $K\subset G(\Adele_f)$, we have the associated Shimura variety (or more precisely stack) $\Sh_K \coloneqq \Sh_K(G,X)$ over $\Rat$. Suppose that $K$ is of the form $K^pK_{V_{\Int_p}}$, and let $S_K$ be the set of primes such that, for $\ell\notin S_K$, we have $K_\ell = K_{V_{\Int_\ell}}$. We then have a normal integral model $\Ss_K$ over $\Int[S_K^{-1}]$ characterized as follows: At any prime $\ell\notin S_K$ such that $V_{\Int_\ell}$ is self-dual, $\Ss_{K,\Int_{(\ell)}}$ is the smooth integral canonical model for $\Sh_K$ constructed in~\cite{kis:abelian} and~\cite{Kim2016-fb}. For arbitrary $\ell\notin S_K$, we have the following characterization: Choose an isometric embedding $V_{\Int}\hookrightarrow V^\sharp_{\Int}$ of quadratic lattices with the following properties:
\begin{enumerate}
	\item The embedding maps onto a direct summand of $V^\sharp_{\Int}$.
	\item $V^\sharp_{\Int}$ has signature $(n^\sharp,2)$.
	\item $V^\sharp_{\Int_\ell}$ is self-dual.
\end{enumerate}
Such an embedding always exists; see~\cite[Lemma B.1.1]{HMP:mod_codim}. Let $\Ss_{K^\sharp,\Int_{(\ell)}}$ be the smooth integral canonical model for the Shimura variety $\Sh_{K^\sharp}$ associated with $V^\sharp$. Then $\Ss_{K,\Int_{(\ell)}}$ is the normalization of $\Ss_{K^\sharp,\Int_{(\ell)}}$ in $\Sh_K$.

Note that this is a localization of Construction~\ref{const:integral_models}: One can choose $V^\sharp_{\Int}$ such that it is self-dual at all finite places, in which case $\Sh_{K^\sharp}$ admits an integral canonical model $\Ss_{K^\sharp}$ over $\Int$, and $\Ss_K$ is obtained by taking the normalization of this model in $\Sh_K$ and inverting the primes in $S_K$.
\end{construction}

\subsubsection{}
Let $\Sigma_K$ be the set of primes $\ell$ such that either $\ell\in S_K$ or $V_{\Int_\ell}$ is \emph{not} self-dual: this excludes $p$ by our hypotheses. The lattice 
\[
H_{\Int} = C(V_{\Int})\subset H = C(V) 
\]
gives us an abelian scheme $\mathcal{A}\to \Ss_K[\Sigma_K^{-1}]$. The lattice $V_{\Int}\subset V$ gives rise to canonical sub-sheaves
\[
\bm{V}_? \subset \underline{\End}\bigl(\bm{H}_?\bigr)
\]
for $? = B,\ell,\dR,\cris$. For every morphism $x:T\to \Ss_K[\Sigma_K^{-1}]$, we have a canonical $\Int$-submodule
\[
V(x) \subset \End(\mathcal{A}_x)
\]
whose cohomological realizations are sections of $\bm{V}_?$ for appropriate values of $?$. The space $V(x)$ has a canonical positive definite quadratic form 
\[
Q:V(x)\to \Int
\]
characterized by the identity $Q(f)\mathrm{id}_{\mathcal{A}_x} = f\circ f\in \End(\mathcal{A}_x)$. This is the module of \defnword{special endomorphisms} of $x$.  All of this is explained in quite a bit more detail in~\cite{Howard2020-fa}.

\subsubsection{}
Let $\Lambda$ be a positive definite lattice over $\Int$. Set
\[
\mathsf{L}(\Lambda) = \{ \Int\mbox{-lattices }\Lambda'\subset \Lambda_\Rat : \;\Lambda\subset \Lambda'\subset(\Lambda')^\vee\subset\Lambda^\vee\},
\]

We will consider the stack $\mathcal{Z}_K(\Lambda)\to \Ss_K[\Sigma_K^{-1}]$ associating with $x:T\to \Ss_K[\Sigma_K^{-1}]$ the set
\[
\mathcal{Z}_K(\Lambda)(x) = \{ \text{isometric embeddings }\iota:\Lambda\to V(x)\}.
\]

The next result can be found in~\cite[\S 3.2]{HMP:mod_codim}.
\begin{lemma}
\label{lem:ZLambda_basic}
The stack $\mathcal{Z}_K(\Lambda)$ is finite and unramified over $\Ss_K[\Sigma_K^{-1}]$. Moreover, if $\Lambda'\in \mathsf{L}(\Lambda)$, then there is a natural closed immersion $\mathcal{Z}_K(\Lambda')\hookrightarrow \mathcal{Z}_K(\Lambda)$ obtained by restricting an isometric embedding to $\Lambda\subset \Lambda'$. 
\end{lemma}

\begin{remark}
If we fix a basis $\Lambda\simeq\Int^m$, we obtain a map of $\Ss_K[\Sigma_K^{-1}]$-stacks $\mathcal{Z}_K(\Lambda)\to \underline{\End}(\mathcal{A})^m$, where $\underline{\End}(\mathcal{A})$ is the endomorphism stack of $\mathcal{A}$ over $\Ss_K[\Sigma_K^{-1}]$. This latter stack is locally finite and unramified over $\Ss_K[\Sigma_K^{-1}]$ by standard facts about endomorphisms of abelian schemes, and $\mathcal{Z}_K(\Lambda)$ is realized as an open and closed stack of $\underline{\End}(\mathcal{A})^m$.
\end{remark}

\begin{remark}
\label{rem:ZLambda_generic_fiber_shimura_varieties}
Fix an embedding $\Lambda_{\Rat}\hookrightarrow V$, and let $V^\flat = \Lambda_{\Rat}^\perp\subset V$. Let $G^\flat\subset G$ be the subgroup acting trivially on $\Lambda_{\Rat}$: this is isomorphic to $\GSpin(V^\flat)$. Set 
\[
\Upsilon(\Lambda) = \{g\in G(\Adele_f):\;\Lambda\subset V\cap gV_{\widehat{\Int}}\}.
\]
For each $g\in \Upsilon(\Lambda)$, we obtain a quadratic lattice $V^\flat_{g,\Int} = V^\flat \cap gV_{\widehat{\Int}}\subset V^\flat$ of signature $(n- \mathrm{rank}(\Lambda),2)$. We can associate with this a GSpin Shimura variety $\Sh_{K^\flat_g}$ mapping to $\Sh_K$, where $K^\flat_g = gKg^{-1}\cap G^\flat(\Adele_f)$. Then there is a canonical isomorphism of $\Sh_K$-stacks
\begin{align}\label{eqn:ZLambda_generic_fiber}
\bigsqcup_{g\in G^\flat(\Rat)\backslash \Upsilon(\Lambda)/K} \Sh_{K^\flat_g} \xrightarrow{\simeq}Z_K(\Lambda).
\end{align}
See the proof of~\cite[Lemma 3.2.3]{HMP:mod_codim}. 
\end{remark}

\begin{remark}
\label{rem:leftcirc_ZLambda}
Consider the open substack
\[
\leftcirc\mathcal{Z}_K(\Lambda) = \mathcal{Z}_K(\Lambda)\smallsetminus \bigcup_{\Lambda'\in \mathsf{L}(\Lambda),\Lambda\subsetneq \Lambda'}\mathcal{Z}_K(\Lambda')\subset \mathcal{Z}_K(\Lambda)
\]

Set $Z_K(\Lambda) = \mathcal{Z}_K(\Lambda)_\Rat$ and $\leftcirc Z_K(\Lambda) = \leftcirc\mathcal{Z}_K(\Lambda)_\Rat$. Let $\leftcirc\Upsilon(\Lambda)\subset \Upsilon(\Lambda)$ be the subset of element $g$ such that $\Lambda$ is a direct summand of $V \cap gV_{\widehat{\Int}}$. Then the decomposition~\eqref{eqn:ZLambda_generic_fiber} restricts to an isomorphism
\begin{align}\label{eqn:leftcirc_ZLambda_generic_fiber}
\bigsqcup_{g\in G^\flat(\Rat)\backslash \leftcirc\Upsilon(\Lambda)/K} \Sh_{K^\flat_g} \xrightarrow{\simeq}\leftcirc Z_K(\Lambda).
\end{align}

Using this, we see that, for every $\Lambda' \in \mathsf{L}(\Lambda)$, the  morphism
\[
Z_K(\Lambda') \to Z_K(\Lambda)
\]
is an open and closed immersion, and there is an isomorphism  of $\Rat$-stacks
\[
\bigsqcup_{\Lambda'\in \mathsf{L}(\Lambda)}\leftcirc Z_K(\Lambda')\xrightarrow{\simeq}Z_K(\Lambda) 
.
\]

\end{remark}

\begin{remark}
\label{rem:open_closed_prime_to_p}
Suppose that $\Lambda'\in \mathsf{L}(\Lambda)$ is such that $p\nmid [\Lambda':\Lambda]$. Then the map $\mathcal{Z}_K(\Lambda')_{\Int_{(p)}}\to \mathcal{Z}_K(\Lambda)_{\Int_{(p)}}$ is an open and closed immersion. It is enough to know that it is \'etale after base-change to $\Field_p$, and this follows from the fact that the deformation theory of $\mathcal{Z}_K(\Lambda)_{\Field_p}$ as a stack over $\Ss_{K,\Field_p}$ depends only on $\Lambda_{\Int_p}$. This follows for instance from the Serre-Tate theorem telling us that the deformation theory of abelian schemes in characteristic $p$ is equivalent to that of the corresponding $p$-divisible groups.
\end{remark}

\begin{proposition}
\label{prop:ZLambda_props}
Suppose that $\mathrm{rank}(\Lambda)\leq (n-4)/2$. 
\begin{enumerate}
	\item $\leftcirc\mathcal{Z}_K(\Lambda)_{\Int_{(p)}}$ is normal, flat, and equidimensional of dimension $n- \mathrm{rank}(\Lambda)+1 = \dim \Ss_K - \mathrm{rank}(\Lambda)$.
	\item The special fiber $\leftcirc\mathcal{Z}_K(\Lambda)_{\Field_p}$ is geometrically normal and equidimensional of dimension $n- \mathrm{rank}(\Lambda)$.
	\item Let $\leftcirc\mathcal{Z}^{\mathrm{ord}}_K(\Lambda)_{\Field_p}\subset \leftcirc\mathcal{Z}_K(\Lambda)_{\Field_p}$ be the pre-image of $\Ss_{K,\Field_p}^{\mathrm{ord}}$. Then $\mathcal{Z}^{\mathrm{ord}}_K(\Lambda)_{\Field_p}$ is open and dense in $\mathcal{Z}_K(\Lambda)_{\Field_p}$
\end{enumerate}
\end{proposition}
\begin{proof}
See~\cite[Proposition 3.2.4]{HMP:mod_codim} for the first two assertions. The last one follows from~\cite[Proposition 7.1.2{1}]{Howard2020-fa}: This shows that the ordinary locus of a certain open substack $\mathcal{Z}^{\mathrm{pr}}_K(\Lambda)_{\Field_p}\subset \leftcirc \mathcal{Z}_K(\Lambda)_{\Field_p}$ is dense within it; but, as argued in the proof of ~\cite[Proposition 3.2.4]{HMP:mod_codim}, $\mathcal{Z}^{\mathrm{pr}}_K(\Lambda)_{\Field_p}$ is in turn dense in $\leftcirc\mathcal{Z}_K(\Lambda)_{\Field_p}$.
\end{proof}

\begin{remark}
\label{rem:leftcirc_ZLambda_integral_models}
Suppose that we are in the situation of Proposition~\ref{prop:ZLambda_props}. For each $g\in \leftcirc\Upsilon(\Lambda)$, we obtain from~\eqref{eqn:leftcirc_ZLambda_generic_fiber}, an open and closed immersion $\Sh_{K^\flat_g}\hookrightarrow \leftcirc Z_K(\Lambda)$. Set $\leftcirc \Ss_{K^\flat_g,(p)}$ to be the Zariski closure of $\Sh_{K^\flat_g}$ in $\leftcirc \mathcal{Z}_K(\Lambda)_{\Int_{(p)}}$. By the normality of the latter stack, we see that we have a decomposition
\begin{align}
\label{eqn:leftcirc_ZLambda_decomp_integral_model}
\bigsqcup_{g\in G^\flat(\Rat)\backslash \leftcirc\Upsilon(\Lambda)/K} \leftcirc\Ss_{K^\flat_g,(p)} \xrightarrow{\simeq}\leftcirc \mathcal{Z}_K(\Lambda)_{\Int_{(p)}}
\end{align}
into open and closed (normal) substacks. 

Note that, if $\Ss_{K^\flat_g}$ is the integral model for $\Sh_{K^\flat_g}$ from Construction~\ref{const:integral_models}, then the map $\leftcirc\Ss_{K^\flat_g,(p)}\to \Ss_{K,(p)}$ lifts to an open immersion $\leftcirc\Ss_{K^\flat_g,(p)}\hookrightarrow \Ss_{K^\flat_g,(p)}$. Indeed, this follows from the construction of $\Ss_{K^\flat_g,(p)}$ as the normalization of $\Ss_{K,(p)}$ in $\Sh_{K^\flat_g}$, and Lemma~\ref{lem:zmt_open_immersion} below.
\end{remark} 

\begin{lemma}
\label{lem:zmt_open_immersion}
Let $f:X\to Y$ be a quasi-finite separated map of normal Deligne-Mumford stacks flat over $\Int_{(p)}$. If the generic fiber $f_{\Rat}$ is an open immersion, then so is $f$.
\end{lemma}
\begin{proof}
By~\cite[Th\'eor\`eme 16.5]{MR1771927}, we can write $f$ as the composition of a finite map $\overline{f}:\overline{X}\to Y$ of Deligne-Mumford stacks with an open immersion $j:X\hookrightarrow \overline{X}$. By replacing $\overline{X}$ with the flat closure of its generic fiber, we can also assume that $\overline{X}$ is flat over $\Int_{(p)}$. By the normality of $Y$, the finite map $\overline{X}\to Y$ must be an isomorphism onto a union of connected components, and hence $f$ is indeed an open immersion.
\end{proof}

\subsection{The ordinary loci of special cycles}\label{subsec:the_ordinary_loci_of_special_cycles}

We will now look at the ordinary locus $\leftcirc \mathcal{Z}^{\ord}_K(\Lambda)_{\Field_p}$ in more detail. For the rest of this section, we will maintain the assumption $\mathrm{rank}(\Lambda)\leq \frac{n-4}{2}$ from Proposition~\ref{prop:ZLambda_props}. We will also assume $\Lambda\neq 0$.

\begin{remark}
	\label{rem:Vflat_isotropic}
	The representatives $\mu:\Gm\to G_{\Int_p}$ for the geometric conjugacy class associated with the Shimura cocharacters of the GSpin Shimura datum are exactly those that yield weight decompositions
\begin{align}\label{eqn:mu_gspin_zp_splitting}
	H_{\Int_p} = H^1_{\Int_p}\oplus H^0_{\Int_p}\;;\; V_{\Int_p} = V^{-1}_{\Int_p}\oplus V^0_{\Int_p}\oplus V^1_{\Int_p}
\end{align}
where $\mu(z)$ acts on $H^i_{\Int_p}$ via $z^{-i}$ and where $V^{\pm 1}_{\Int_p}\subset V_{\Int_p}$ are complementary isotropic lines on which $\mu(z)$ acts via $z^{\mp 1}$ and $V^0_{\Int_p}$ is their mutual orthogonal complement. 
\end{remark}

\begin{remark}
	\label{rem:p-adic_special_endomorphisms}
For every point $x_0\in \Ss_{K,\Field_p}(\kappa)$, we have a $p$-adic counterpart of the space of special endomorphisms: Namely the $F$-invariant vectors $\bm{V}_{\cris,x_0}^{F=1} \subset \bm{V}_{\cris,x_0}$. These are precisely the elements of $\bm{V}_{\cris,x_0}$ that are also crystalline realizations of endomorphisms of $\mathcal{G}_{x_0}$. When $x_0$ is ordinary, there exist isomorphisms $V^0_{\Int_p}\xrightarrow{\simeq}\bm{V}_{\cris,x_0}^{F=1}$ well-defined up to an element of $M_{\Int_p}(\Int_p)$. In particular, for any point $(x_0,\iota)\in \leftcirc \mathcal{Z}^{\ord}_K(\Lambda)_{\Field_p}(\kappa)$, $\iota$ the crystalline realization of $\iota$ gives an isometric embedding $\Lambda_{\Int_p}\hookrightarrow \bm{V}_{\cris,x_0}^{F=1}$ onto a direct summand.
\end{remark}

\begin{remark}
	For $g\in\leftcirc\Upsilon(\Lambda)$, we can apply the constructions of~\eqref{subsubsecqW} with $N = V_{\Int_p}$ and $W = W_g\defn g_p^{-1}\Lambda_{\Int_p}\subset N$. We set $G^\flat_{g,\Int_p} = Q_{W_g}$ in the notation there: This is a smooth group scheme over $\Int_p$, clearly a subgroup of $G_{\Int_p}$, and conjugation by $g_p$ identifies its generic fiber with the subgroup $G^\flat_{\Rat_p}\subset G_{\Rat_p}$. Moreover, via this identification, we have $G^\flat_{g,\Int_p}(\Int_p) = K^\flat_{g,p}\subset G^\flat(\Rat_p)$. 
\end{remark}

\begin{lemma}
	\label{lem:mu_choice_flat}
Fix $g\in \leftcirc \Upsilon(\Lambda)$. Then:
\begin{enumerate}
	\item We can choose the decomposition~\eqref{eqn:mu_gspin_zp_splitting} so that the corresponding cocharacter factors through $G^\flat_{g,\Int_p}$. 
	\item All such decompositions---and hence all such cocharacters---are conjugate under $G^\flat_{g,\Int_p}(\Int_p)$.
	\item The centralizer of any such cocharacter is a smooth subgroup scheme $M^\flat_{g,\Int_p}\subset G^\flat_{g,\Int_p}$ such that $\tilde{M}^\flat_{g,\Int_p}$\footnote{This is the Zariski closure of the derived subgroup of the generic fiber.} is also smooth.
\end{enumerate}
\end{lemma}
\begin{proof}
This amounts to knowing that, if $U_{\Int_p}$ is the hyperbolic plane over $\Int_p$, then there is an isometric embedding $U_{\Int_p}\hookrightarrow N$ that is orthogonal to $W_g$.

The numerical condition on $\mathrm{rank}(\Lambda)$ ensures that we have $n\ge 6$, $\dim V^\flat\ge \frac{n}{2}+4$ and $\dim V^\flat - \mathrm{rank}(\Lambda)\ge 6$. In particular, for any $g\in \leftcirc\Upsilon(\Lambda)$, the quadric of isotropic lines in $W_g^\perp$ is a flat local complete intersection over $\Int_p$ of relative dimension $\ge 5$, whose special fiber is a rational variety. Moreover, the singular locus of the special fiber consists of isotropic lines that are contained in the radical of $W_{g,\Field_p}^{\perp}$ and has codimension $\ge 5$; see~\cite[Lemma 2.11]{mp:reg}. 

Using this geometric information, we see that we can first pick any isotropic line $V^{-1}_{\Int_p}\subset W_g^\perp$ in the smooth locus, and then choose a complementary isotropic line $V^1_{\Int_p}$ also in the smooth locus of the quadric associated with $W_g^{\perp}$. 

That all such decompositions are conjugate under $G^\flat_{g,\Int_p}(\Int_p)$ follows from Lemma~\ref{lem:GSpin_torsor}, applied with $W$ the orthogonal direct sum $U_{\Int_p}\oplus W_g^{\perp}$.

The group $M^\flat_{g,\Int_p}$ is an extension of $\Gmh{\Int_p}$ by the pointwise stabilizer of the direct summand $V^{-1}_{\Int_p}\oplus W_g \oplus V^1_{\Int_p}$. Therefore, assertion (3) follows from Lemma~\ref{lem:qT_smooth}. 
\end{proof}

\begin{remark}
	\label{rem:ordinary_SKg}
Choose any cocharacter $\mu_g:\Gmh{\Int_p}\to G^\flat_{g,\Int_p}$ as in Lemma~\ref{lem:mu_choice_flat}. Then Lemma~\ref{lem:mu_choice_flat} tells us that Assumption~\ref{assump:ordinary_cochar} is valid, and so we can now consider the ordinary locus $\Ss^{\ord}_{K_g^\flat,\Field_p}\subset \Ss_{K_g^\flat}$. By (2) of the lemma, this is independent of the choice of cocharacter.
\end{remark}

\begin{lemma}
	\label{lem:ordinary_SKg_agrees}
We have $\Ss^{\ord}_{K_g^\flat,\Field_p} = \leftcirc\Ss_{K_g^\flat,\Field_p}\times_{\leftcirc \mathcal{Z}_K(\Lambda)_{\Field_p}}\leftcirc\mathcal{Z}^{\ord}_K(\Lambda)_{\Field_p}$. Moreover, Assumption~\ref{assump:canonical_lift} holds: the canonical lift of any point of $\Ss^{\ord}_{K_g^\flat,\Field_p}$ lifts to $\Ss_{K_g^\flat}$.
\end{lemma}
\begin{proof}
	We want to show that, for any algebraically closed point $x_0\in \leftcirc\Ss_{K_g^\flat}(\kappa)$ with $\mathcal{A}_{x_0}$ ordinary, we can find a $G^\flat_{g,\Int_p}$-structure preserving trivialization $W(\kappa)\otimes_{\Int_p}H_{\Int_p}\xrightarrow{\simeq}\bm{H}_{\cris,x_0}$ such that the Frobenius endomorphism on the right hand side conjugates to the endomorphism $1\otimes \mu_g(p)^{-1}$ of $W(\kappa)\otimes_{\Int_p}H_p$. That we can find a $G_{\Int_p}$-structure preserving isomorphism follows from Remark~\ref{rem:unramified_ordinary}. We can then use Lemma~\ref{lem:GSpin_torsor} to modify this to a $G^\flat_{g,\Int_p}$-structure preserving one.

	The validity of Assumption~\ref{assump:canonical_lift} is a consequence of the fact that all endomorphisms of $\mathcal{A}_{x_0}$ deform to endomorphisms of its canonical lift.
\end{proof}

\begin{remark}
	\label{rem:igusa_reduction_g_flat}
	The Igusa tower $\widehat{\mathrm{Ig}}^{\ord}_{K,p}$ over $\widehat{\Ss}^{\ord}_{K,p}$, when restricted to $\widehat{\Ss}^{\ord}_{K^\flat_g,p}$ acquires a reduction of structure group to the subgroup $M^\flat_{g,\Int_p}(\Int_p)\subset M_{\Int_p}(\Int_p)$, given by its own Igusa tower $\widehat{\mathrm{Ig}}^{\ord}_{K^\flat_g,p}$. This can be understood explicitly. First, Remark~\ref{rem:p-adic_special_endomorphisms} tells us that, for every algebraically closed point  $x_0\in \widehat{\Ss}^{\ord}_{K^\flat_g,p}(\kappa)$, the crystalline realization of $\iota:\Lambda\to V(x_0)$ yields an embedding as a direct summand
\[
	W_g \xrightarrow[\simeq]{g_p}\Lambda_{\Int_p}\hookrightarrow \bm{V}^{F=1}_{\cris,x_0}.
\]
	The reduction of structure is now given by the subsheaf parameterizing for $x:\Spf R\to \widehat{\Ss}^{\ord}_{K_g^\flat,p}$, trivializations 
  \[
	R\otimes_{\Int_p}\mathcal{G}_0\xrightarrow{\simeq}\mathcal{G}^{\et}_x\times \mathcal{G}^{\mult}_x
	\]
	such that, at every geometric point $s:\Spec \kappa\to \Spf R$, we obtain a commuting diagram
  \begin{align*}
  \begin{diagram}
  		 W_g&\rEquals& W_g\\
  	 \dInto&&\dInto\\
  	 V^0_{\Int_p}&\rTo_{\simeq}&\bm{V}^{F=1}_{\cris,x\circ s}.
  \end{diagram}
  \end{align*}
\end{remark}

\begin{lemma}
	\label{lem:ordinary_SKg_hypersymmetric}
$\Ss^{\ord}_{K_g^\flat,\Field_p}$ contains a hypersymmetric point (see Definition~\ref{ord:defn:hypersymmetric}).
\end{lemma}
\begin{proof}
This just amounts to the observation that there is a rank $2$ negative definite lattice $L_{\Int}\subset V^\flat_{g,\Int}$ such that $L_{\Int_p}$ is isomorphic to the hyperbolic plane. Indeed, $\GSpin(L)$ will then be of the form $\Res_{F/\Rat}\Gm$ for an imaginary quadratic extension $F/\Rat$ that is split at $p$, and we can then conclude using Remark~\ref{rem:hypersymmetric_CM}.	
\end{proof}

\begin{remark}
	\label{rem:ordinary_orthogonal_unipotent_unramified}
Let $U^-_{\mu_g}\subset G_{\Int_p}$ be the opposite unipotent associated with $\mu_g$, so that we have a canonical inclusion
\begin{align}\label{eqn:unip_embedding}
\Lie U^-_{\mu_g}\subset \Hom(H_{\Int_p}^1,H_{\Int_p}^0).
\end{align}
This can be made explicit. Namely, let $G_0 = \SO(V_{\Int_p})(\Int_p)$ be the special orthogonal quotient of $G_{\Int_p}$. Then we can also identify $U^-_{\mu_g}$ with the unipotent subgroup of $G_0$ associated with the isotropic line $V_{\Int_p}^{-1}$. That is, we have
\[
\Lie U^-_{\mu_g} = \{(\varphi,\psi)\in \Hom(V_{\Int_p}^0,V_{\Int_p}^{-1})\times \Hom(V_{\Int_p}^1,V_{\Int_p}^0):\; \varphi^\vee + \psi = 0\}\subset \End(V_{\Int_p}).
\]
Here, we have used the non-degenerate bilinear form on $V_{\Int_p}$ to identify $V_{\Int_p}^0$ with its own dual, and $V_{\Int_p}^1$ with the dual of $V_{\Int_p}^{-1}$, and hence the dual $\varphi^\vee$ of $\varphi$ with a map $\varphi^\vee:V_{\Int_p}^1\to V_{\Int_p}^0$. In what follows, we can and will identify $\Lie U^-_{\mu_g}$ with its image in $\Hom(V_{\Int_p}^0,V_{\Int_p}^{-1})$. Fix generators $v^{\pm}$ of $V_{\Int_p}^{\pm 1}$. Then, as explained in \cite[\S 1]{mp:reg}, under the left multiplication action of $V_{\Int_p}$ on $H_{\Int_p}$, we have
\[
H_{\Int_p}^1 = \ker(v^+) = \im(v^+)\;;\; H_{\Int_p}^0 = \ker(v^{-})= \im (v^{-}).
\]
The embedding~\eqref{eqn:unip_embedding} can now be described as follows: Suppose that we have a map $\varphi:V_{\Int_p}^0\to V_{\Int_p}^{-1}$ in $\Lie U^-_{\mu_g}$. There exists a unique $v_\varphi^0\in V_{\Int_p}^0$ such that, for all $v\in V_{\Int_p}^0$, we have $[v_\varphi^0,v]_Q\cdot v^{-} = \psi(v).$ Now, one can check that, up to sign, under~\eqref{eqn:unip_embedding}, $(\varphi,\psi)$ maps to left multiplication by the element $v^{-}v_\varphi^0$ in the Clifford algebra.
\end{remark}

\begin{remark}
	\label{rem:ordinary_SKg_def_rings}
Now, let $\Lie U^{\flat,-}_{\mu_g} = \Lie U^-_{\mu_g}\cap \Lie G_{g,\Int_p}^\flat$. One checks that this is identified with the subspace 
\[
	\Hom(V_{\Int_p}^0/W_g,V_{\Int_p}^1)\subset \Hom(V_{\Int_p}^0,V_{\Int_p}^1).
\]
Fix an algebraically closed point $x_0\in \Ss^{\ord}_{K^\flat_g,\Field_p}(\kappa)$, and let $\widehat{U}^\flat_{x_0}$ (resp. $\widehat{U}_{x_0}$) be the deformation spaces of $\widehat{\Ss}^{\ord}_{K^\flat_g,p}$ (resp. $\widehat{\Ss}^{\ord}_{K,p}$).
Using Lemma~\ref{lem:ordinary_SKg_agrees} and Proposition~\ref{prop:ordinary_deformations}, we see that we have a diagram 
\[
	\begin{diagram}
		\widehat{U}^\flat_{x_0}&\rTo&\widehat{U}_{x_0}\\
		\dTo^\simeq&&\dTo_\simeq\\
		\widehat{T}_{G^\flat_g}&\rTo&\widehat{T}_G
	\end{diagram}
\]
where $\widehat{T}_{G^\flat_g}$ (resp. $\widehat{T}_G$) is the formal torus over $W(\kappa)$ with cocharacter group $\Lie U^{\flat,-}_{\mu_g}$ (resp. $\Lie U^-_{\mu_g}$). This diagram is uniquely determined up to the action of $M^\flat_{\mu_g,\Int_p}(\Int_p)$. 
\end{remark}

\begin{remark}
	\label{rem:ordinary_SKg_character_groups}
If we choose a generator $v^1$ for $V^1_{\Int_p}$, then we can compatibly identify the character groups of $\widehat{T}_G$ and $\widehat{T}_{G^\flat_g}$ with $V^0_{\Int_p}$ and its quotient $V^0_{\Int_p}/W_g$, respectively.
\end{remark}

\subsection{Correspondences between special cycles in characteristic $0$}\label{subsec:corr_char_0}

We will now apply the theory of Hecke correspondences from \S~\ref{subsec:isogenies_and_p_hecke_correspondences_in_the_generic_fiber} in the particular context of GSpin Shimura varieties.

\begin{remark}
\label{rem:conjugation_isogeny_generic_fiber}
	Suppose that we have $s,t:\Spec R\to \Sh_K$ and $f\in \mathrm{QIsog}_G(s,t)$. Then conjugation by $f^{-1}$ induces an isomorphism 
\[
	c(f) \defn f^{-1}\circ (\cdot)\circ f:\End(\mathcal{A}_s)[p^{-1}]\xrightarrow{\simeq}\End(\mathcal{A}_t)[p^{-1}] 
\]
	carrying $V(s)[p^{-1}]$ onto $V(t)[p^{-1}]$. 
\end{remark}

\begin{definition}
A \defnword{$p$-minimal pair} or simply \defnword{minimal pair} $\Xi = (\tilde{\Lambda}\subset \Lambda)$ is a pair of positive definite lattices $(\Lambda,\tilde{\Lambda})$ equipped with an isometric embedding $\tilde{\Lambda}\subset \Lambda$ such that $[\Lambda:\tilde{\Lambda}] = p$.
\end{definition}

\begin{definition}
	Suppose that $\Xi$ is a minimal pair. Define a stack
\[
	\mathrm{QIsog}_{G,\mu_p,\Xi}\to \leftcirc Z_K(\Lambda)\times \leftcirc Z_K(\tilde{\Lambda})
\]
whose fiber over $((s,\iota),(t,\tilde{\iota}))\in \leftcirc Z_K(\Lambda)(R)\times \leftcirc Z_K(\tilde{\Lambda})(R)$ is given by
\[
	\mathrm{QIsog}_{G,\mu_p,\Xi}((s,\iota),(t,\tilde{\iota})) = \{f\in \mathrm{QIsog}_{G,\mu_p}(s,t):\;c(f)\circ \iota\vert_{\tilde{\Lambda}} = \tilde{\iota}\}
\]
In other words, we are looking for $p$-quasi-isogenies of type $\mu_p$ that `shrink' the isometric embedding of $\Lambda$ onto a direct summand of the space of special endomorphisms to that of $\tilde{\Lambda}$ onto a direct summand of the space of special endomorphisms of an isogenous point. Write $s_{\mu_p,\Xi,K}:\mathrm{QIsog}_{G,\mu_p,\Xi}\to \leftcirc Z_K(\Lambda)$ and $t_{\mu_p,\Xi,K}:\mathrm{QIsog}_{G,\mu_p,\Xi}\to \leftcirc Z_K(\tilde{\Lambda})$ for the natural maps.
\end{definition}

\begin{remark}
	\label{rem:reduction_of_structure_group}
Suppose that we have $g\in \leftcirc\Upsilon(\Lambda)$. Over $\Sh_{K^\flat_g}$, the $p$-adic realization of the tautological embedding of $\Lambda$ into the space of special endomorphisms gives an embedding
\[
	\bm{\iota}_p:\underline{\Lambda}_{\Int_p}\hookrightarrow \bm{V}_p\vert_{\Sh_{K^\flat_g}}
\]
as a local direct summand. Using Lemma~\ref{lem:GSpin_torsor}, we now see that the canonical $K^\flat_{g,p}$-torsor $\Sh_{K^{\flat,p}_g}\to \Sh_{K^\flat_g}$ parameterizes $G_{\Int_p}$-structure preserving trivializations $\underline{H}_{\Int_p}\xrightarrow{\simeq}\bm{H}_p$ that carry the embedding $W_g\hookrightarrow V_{\Int_p}$ onto $\bm{\iota}_p$.
\end{remark}

\begin{proposition}
	\label{prop:char0_corr_cycles}
\begin{enumerate}
\item The map $s_{\mu_p,\Xi,K}$ is finite \'etale and induces a bijection on geometric connected components.
	\item The map $t_{\mu_p,\Xi,K}$ is an isomorphism.
\end{enumerate}
\end{proposition}
\begin{proof}
Let $\mathrm{Par}_{\mu_p}$ be the quadric Grassmannian over $\Int_p$ parameterizing isotropic lines in $V_{\Int_p}$. Fix $g\in \leftcirc\Upsilon(\Lambda)$, consider the map
\[
\mathrm{QIsog}_{G,\mu_p,\Xi}\times_{s_{\mu_p,\Xi,K},\leftcirc Z_K(\Lambda)}\Sh_{K^\flat_g}\to \Sh_{K^\flat_g}
\]
By Proposition~\ref{prop:isogG_generic_desc} and its proof, combined with Remark~\ref{rem:quadric_desc_par_lambda_0}, we obtain a commutative diagram
\[
\begin{diagram}
	\mathrm{QIsog}_{G,\mu_p,\Xi}\times_{s_{\mu_p,\Xi,K},\leftcirc Z_K(\Lambda)}\Sh_{K^\flat_g}&\rTo&\Sh_{K^{\flat,p}_g}\times^{G^\flat_{g,\Int_p}(\Int_p)}\mathrm{Par}_{\mu_p}(\Field_p)\\
	&\rdTo(1,2)\ldTo(1,2)\\
		&\Sh_{K^\flat_g}
\end{diagram}
\]
Unwinding the definitions and using (1) of Lemma~\ref{lem:nice_cochar}, one finds that the top arrow maps its source isomorphically onto the substack
\[
	\Sh_{K^{\flat,p}_g}\times^{G^\flat_{g,\Int_p}(\Int_p)}\mathrm{Par}_{\mu_p}(W_g,U_g)(\Field_p)\subset \Sh_{K^{\flat,p}_g}\times^{G^\flat_{g,\Int_p}(\Int_p)}\mathrm{Par}_{\mu_p}(\Field_p),
\]
where $U\in \mathbb{P}(W_g)(\Int_p)$ is any hyperplane such that $g_p^{-1}\tilde{\Lambda}_{\Int_p} = U + pW_g$. Moreover, by Lemma~\ref{lem:smooth_grp_schemes}, $G^\flat_{g,\Int_p}(\Field_p)$ acts transitively on $\mathrm{Par}_{\mu_p}(W_g,U_g)(\Field_p)$, and the stabilizer $Q(\Field_p)$ of any point in this set maps surjectively on $\Field_p^\times$ via the spinor norm.

To finish, we need to show that the right diagonal arrow induces a bijection on geometric connected components. Since $G^\flat$ has simply connected derived subgroup, we find from~\cite[Th\'eor\`eme 2.4]{deligne:travaux} that for any level $K'\subset G^\flat(\Adele_f)$ the spinor norm induces a bijection
\[
	\pi_0(\Sh_{K',\overline{\Rat}})\xrightarrow{\simeq}\Adele_f^\times/\Rat_{>0}\nu(K').
\]
Now, the discussion in the previous paragraph shows that we have
\[
	\Sh_{K^{\flat,p}_g}\times^{G^\flat_{g,\Int_p}(\Int_p)}\mathrm{Par}_{\mu_p}(\Field_p) = \Sh_{K'}
\]
where $K'\subset K^\flat_g$ is the pre-image of $Q(\Field_p)\subset G^\flat_{g,\Int_p}(\Field_p)$. The surjectivity of the spinor norm $Q(\Field_p)\to \Field_p^\times$ shows that we have $\nu(K') = \nu(K^\flat_g)$, and hence that the map on connected components is a bijection as desired.

The proof of assertion (2) is along analogous lines and uses (2) of Lemma~\ref{lem:nice_cochar}.
\end{proof}

\subsection{Ordinary correspondences of minuscule type}\label{sec:ordinary_correspondences_of_minuscule_type}

Fix a cocharacter $\mu_p:\Gmh{\Int_p}\to G_{\Int_p}$ as in Remark~\ref{rem:Vflat_isotropic} with associated splittings~\eqref{eqn:mu_gspin_zp_splitting}. Let $M_{\Int_p}\subset G_{\Int_p}$ be the centralizer of $\mu_p$: This is an extension of $\Gm$ by $\GSpin(V^0_{\Int_p})$. In this subsection, we will study in detail the structure of $\widehat{\mathrm{QIsog}}^{\ord}_{G,\lambda_p}$ for the choice of a non-trivial minuscule cocharacter $\lambda_p$ of $M_{\Int_p}$. Up to isomorphism as a formal stack over $\widehat{\Ss}^{\ord}_{K,p}\times \widehat{\Ss}^{\ord}_{K,p}$, this will not depend on the choice of the pair $(\mu_p,\lambda_p)$ in their $G_{\Int_p}(\Int_p)$-conjugacy class.

\begin{definition}
		Consider the conjugacy class of cocharacters $\lambda_p:\Gmh{\Int_p}\to M_{\Int_p}$ factoring through $\GSpin(V^0_{\Int_p})$, acting via weights $0,1$ on the Clifford algebra $C(V^0_{\Int_p})$ and giving a splitting
	\begin{align}\label{eqn:MZp_orthogonal_splitting}
		V^0_{\Int_p} = V^0_{\Int_p}(-1)\oplus V^0_{\Int_p}(0)\oplus V^0_{\Int_p}(1)
	\end{align}
	where $V^0_{\Int_p}(\pm 1)\subset V^0_{\Int_p}$ are complementary isotropic lines and $V^0_{\Int_p}(0)$ is their mutual orthogonal complement. 
\end{definition}

\begin{remark}
	\label{rem:isogG_lambdap_etale_struct}
For any cocharacter $\varpi:\Gmh{\Int_p}\to M_{\Int_p}$, set
	\[
		\widehat{\mathrm{Ig}}^{\ord}_{\varpi} \defn \widehat{\mathrm{Ig}}^{\ord}_{K,v}\times^{M_{\Int_p}(\Int_p)}\underline{M_{\Int_p}(\Int_p)\varpi(p)M_{\Int_p}(\Int_p)/M_{\Int_p}(\Int_p)}.
	\]
Then it follows from Proposition~\ref{prop:isogM_desc_without_isoms} and Corollaries~\ref{cor:qisogM_ordinary_desc} and~\ref{cor:isogG_ordinary_deformation_desc} that we have diagrams
\begin{align*}
		\begin{diagram}
		\widehat{\mathrm{QIsog}}^{\ord}_{G,\lambda_p}&\rTo&\widehat{\mathrm{Ig}}^{\ord}_{-\lambda_p}\\
		&\rdTo(1,2)_s\ldTo(1,2)\\
		&\widehat{\Ss}^{\ord}_{K,p}
	\end{diagram}\qquad;\qquad
	\begin{diagram}
		\widehat{\mathrm{QIsog}}^{\ord}_{G,\lambda_p}&\rTo&\widehat{\mathrm{Ig}}^{\ord}_{\lambda_p}\\
		&\rdTo(1,2)_t\ldTo(1,2)\\
		&\widehat{\Ss}^{\ord}_{K,p}
	\end{diagram}
\end{align*}
where both top horizontal arrows are finite flat homeomorphisms. Moreover, one finds from Remark~\ref{rem:quadric_desc_par_lambda_0} that we have
\[
		\widehat{\mathrm{Ig}}^{\ord}_{-\lambda_p}\simeq 	\widehat{\mathrm{Ig}}^{\ord}_{\lambda_p}\simeq \widehat{\mathrm{Ig}}^{\ord}_{K,p}\times^{M_{\Int_p}(\Int_p)}\underline{\mathrm{Par}_{\lambda_p}(\Field_p)},
\]
where $\mathrm{Par}_{\lambda_p}(\Field_p)$ is the set of isotropic lines in $V^0_{\Field_p}$.
\end{remark}

\begin{remark}
The cocharacter $\lambda_p$ breaks up $\Lie U^-_{\mu_p}$ and $\Hom(H_{\Int_p}^1,H_{\Int_p}^0)$ compatibly into eigenspaces
\begin{align*}
\Lie U^-_{\mu_p} = \bigoplus_{i=-1}^1\Lie U^-_{\mu_p}(i)\;;\;\Hom(H_{\Int_p}^1,H_{\Int_p}^0) = \bigoplus_{i=-1}^1\Hom(H_{\Int_p}^1,H_{\Int_p}^0)(i),
\end{align*}
where
\[
\Lie U^-_{\mu_p}(i) = \begin{cases}
\Hom(V_{\Int_p}^0(1),V_{\Int_p}^{-1})&\text{if $i=-1$};\\
\Hom(V_{\Int_p}^0(0),V_{\Int_p}^{-1})&\text{if $i=0$};\\
\Hom(V_{\Int_p}^0(-1),V_{\Int_p}^{-1})&\text{if $i=1$}.
\end{cases},
\]
and
\[
\Hom(H_{\Int_p}^1,H_{\Int_p}^0)(i) = \begin{cases}
\Hom(H_{\Int_p}^1(1),H_{\Int_p}^0(0))&\text{if $i=-1$};\\
\Hom(H_{\Int_p}^1(0),H_{\Int_p}^0(0)) \oplus \Hom(H_{\Int_p}^1(1),H_{\Int_p}^0(1))&\text{if $i=0$};\\
\Hom(H_{\Int_p}^1(0),H_{\Int_p}^0(1))&\text{if $i=1$}.
\end{cases}
\]
\end{remark}

\begin{remark}
	\label{rem:ord_def_ring_lambda_p}
Recall the notation and setup from Remark~\ref{rem:ordinary_orthogonal_unipotent_unramified}. Consider the homomorphism
\begin{align}\label{eqn:ord_def_lambda_p_cocharacter}
\Lie U^-_{\mu_p}\oplus \Lie U^-_{\mu_p}\xrightarrow{(\varphi_1,\varphi_2)\mapsto v^-{v_{\varphi_1}}^0\lambda_p(p)-\lambda_p(p)v^-v_{\varphi_2}^0}\Hom(H_{\Int_p}^1,H_{\Int_p}^0)
\end{align}
If we restrict to each of the non-zero weight spaces for the action of $\lambda_p$, then one sees that we obtain
\begin{align*}
	\Lie U^-_{\mu_p}(-1)\oplus \Lie U^-_{\mu_p}(-1)\subset \Hom(H_{\Int_p}^1(1),H_{\Int_p}^0(0))\times \Hom(H_{\Int_p}^1(1),H_{\Int_p}^0(0))\xrightarrow{(f_1,f_2)\mapsto pf_1-f_2}\Hom(H_{\Int_p}^1,H_{\Int_p}^0)\\
	\Lie U^-_{\mu_p}(-1)\oplus \Lie U^-_{\mu_p}(-1)\Hom(H_{\Int_p}^1(0),H_{\Int_p}^0(1))\oplus \Hom(H_{\Int_p}^1(0),H_{\Int_p}^0(1))\xrightarrow{((f_1,f_2)\mapsto f_1-pf_2}\Hom(H_{\Int_p}^1,H_{\Int_p}^0).
\end{align*}
For the zero weight space, the restriction of both maps $\varphi\mapsto v^-{v_\varphi}^0\lambda_p(p)$ and $\varphi\mapsto \lambda_p(p)v^-{v_\varphi}^0$ can be identified with the composition
\begin{align}\label{eqn:lambda_p_weight_0_map}
\Lie U^-_{\mu_p}(0)&\xrightarrow{\varphi\mapsto v^-v^0_\varphi}\Hom(H_{\Int_p}^1(0),H_{\Int_p}^0(0)) \oplus \Hom(H_{\Int_p}^1(1),H_{\Int_p}^0(1))\\
&\xrightarrow{(\mathrm{id},p\cdot \mathrm{id})}\Hom(H_{\Int_p}^1(0),H_{\Int_p}^0(0)) \oplus \Hom(H_{\Int_p}^1(1),H_{\Int_p}^0(1)),\nonumber
\end{align}
which maps onto a direct summand of its target.
\end{remark}

\begin{remark}\label{rem:ord_def_ring_lambda_p_2}
For any algebraically closed field $\kappa$,~\eqref{eqn:ord_def_lambda_p_cocharacter} gives  a map of formal tori over $W(\kappa)$, $\widehat{T}_G\times \widehat{T}_G \to \widehat{T}$, and Proposition~\ref{prop:isogG_ordinary_deformation_rings} now shows that the kernel $\widehat{T}_{G,\lambda_p}$ of this map  is the model for the completion of $\widehat{\mathrm{QIsog}}^{\ord}_{G,\lambda_p}$ at any $\kappa$-valued point. 

The weight decompositions of their cocharacter groups for the action of $\lambda_p$ yield compatible splittings
\[
	\widehat{T}_G = \prod_{i=-1}^1\widehat{T}_G(i)\subset \prod_{i=-1}^1\widehat{T}(i)\subset \widehat{T}
\]
The discussion in Remark~\ref{rem:ord_def_ring_lambda_p} shows that the map $\widehat{T}_G\times \widehat{T}_G \to \widehat{T}$ respects this decomposition, and its restrictions take the following shape:
\begin{align*}
	\widehat{T}_G(-1)\times \widehat{T}_G(-1)&\xrightarrow{(x,y)\mapsto x^py^{-1}}\widehat{T}_G(-1)\subset \widehat{T}(-1);\\
	\widehat{T}_G(0)\times\widehat{T}_G(0)&\xrightarrow{(x,y)\mapsto x-y}\widehat{T}_G(0)\xrightarrow{\varpi}\widehat{T}(0);\\
	\widehat{T}_G(1)\times \widehat{T}_G(1)&\xrightarrow{(x,y)\mapsto xy^{-p}}\widehat{T}_G(1)\subset \widehat{T}(1).
\end{align*}
Here $\varpi$ is the closed immersion of formal tori arising from the map~\eqref{eqn:lambda_p_weight_0_map} on cocharacter groups.

Therefore, we obtain a decomposition 
\[
\widehat{T}_{G,\lambda_p} = \widehat{T}_{G,\lambda_p}(-1)\times \widehat{T}_{G,\lambda_p}(0)\times \widehat{T}_{G,\lambda_p}(1),
\]
with 
\begin{align}\label{eqn:UG_lambda0_decomp}
\widehat{T}_{G,\lambda_p}(i) &= \begin{cases}
\{(x,x^p):\;x\in \widehat{T}_G(-1)\}&\text{if $i=-1$};\\
\{(x,x):\;x\in \widehat{T}_G(0)\}&\text{if $i=0$};\\
\{(x^p,x):\;x\in \widehat{T}_G(1)\}&\text{if $i=1$}.
\end{cases}
\end{align}
The map $s_{\lambda_p}$ (resp. $t_{\lambda_p}$) from $\widehat{T}_{G,\lambda_p}$ to $\widehat{T}_G$ is therefore given by the identity on $\widehat{T}_{G,\lambda_p}(-1)$ and $\widehat{T}_{G,\lambda_p}(0)$ (resp. $\widehat{T}_{G,\lambda_p}(1)$ and $\widehat{T}_{G,\lambda_p}(0)$ ) and by the $p$-power map on  $\widehat{T}_{G,\lambda_p}(1)$ (resp.  $\widehat{T}_{G,\lambda_p}(-1)$).
\end{remark}

\begin{remark}
	\label{rem:ord_def_ring_lambda_character_groups}
Just as in Remark~\ref{rem:ordinary_SKg_character_groups}, via the choice of basis element $v^1$ for $V^1_{\Int_p}$, we can identify the character group of $\widehat{T}_{G,\lambda_p}$ with $\bigoplus_{i=-1}^1V^0_{\Int_p}(i) = V^0_{\Int_p}$, and the source and target maps are given on the level of character groups, and for this splitting, by $(p,1,1)$ and $(1,1,p)$, respectively.
\end{remark}

\subsection{Ordinary correspondences between special cycles}\label{subsec:corr_ordinary}

Here we give the ordinary counterpart to the story from \S\ref{subsec:corr_char_0}.


\begin{remark}
\label{rem:conjugation_isogeny_ordinary_locus}
	Suppose that we have $s,t:\Spf R\to \widehat{\Ss}^{\ord}_{K,\Field_p}$ and $f\in \widehat{\mathrm{QIsog}}^{\ord}_G(s,t)$. Then just as in the generic fiber in Remark~\ref{rem:conjugation_isogeny_generic_fiber} conjugation by $f^{-1}$ induces an isomorphism 
\[
	c(f) \defn f^{-1}\circ (\cdot)\circ f:\End(\mathcal{A}_s)[p^{-1}]\xrightarrow{\simeq}\End(\mathcal{A}_t)[p^{-1}] 
\]
	carrying $V(s)[p^{-1}]$ onto $V(t)[p^{-1}]$. 
\end{remark}

\begin{definition}
	Suppose that $\Xi = (\tilde{\Lambda}\subset\Lambda)$ is a minimal pair. Define a formal stack
\[
	\widehat{\mathrm{QIsog}}^{\ord}_{G,\lambda_p,\Xi}\to \leftcirc \widehat{\mathcal{Z}}^{\ord}_K(\Lambda)\times \leftcirc \widehat{\mathcal{Z}}^{\ord}_K(\tilde{\Lambda})
\]
whose fiber over $((s,\iota),(t,\tilde{\iota}))\in \leftcirc \widehat{\mathcal{Z}}^{\ord}_K(\Lambda)(R)\times \leftcirc \widehat{\mathcal{Z}}^{\ord}_K(\tilde{\Lambda})(R)$ for $R$ $p$-complete is given by
\[
	\widehat{\mathrm{QIsog}}^{\ord}_{G,\lambda_p,\Xi}((s,\iota),(t,\tilde{\iota})) = \{f\in \widehat{\mathrm{QIsog}}^{\ord}_{G,\lambda_p}(s,t):\;c(f)\circ \iota\vert_{\tilde{\Lambda}} = \tilde{\iota}\}.
\]
Write $s^{\ord}_{\lambda_p,\Xi,K}:\widehat{\mathrm{QIsog}}^{\ord}_{G,\lambda_p,\Xi}\to \leftcirc \widehat{\mathcal{Z}}^{\ord}_K(\Lambda)$ and $t^{\ord}_{\lambda_p,\Xi,K}:\widehat{\mathrm{QIsog}}^{\ord}_{G,\lambda_p,\Xi}\to \leftcirc \widehat{\mathcal{Z}}^{\ord}_K(\tilde{\Lambda})$ for the natural maps. For each $g\in \leftcirc\Upsilon(\Lambda)$ (resp. $\tilde{g}\in \leftcirc\Upsilon(\tilde{\Lambda})$), write $s^{\ord}_{\lambda_p,\Xi,K_g^\flat}$ (resp. $t^{\ord}_{\lambda_p,\Xi,K_{\tilde{g}}^\flat}$ for the restriction of $s^{\ord}_{\lambda_p,\Xi,K}$ (resp. $t^{\ord}_{\lambda_p,\Xi,K}$) over the open subspace $\widehat{\Ss}^{\ord}_{K^\flat_g,p}$ (resp. $\widehat{\Ss}^{\ord}_{K^\flat_{\tilde{g}},p}$ )
\end{definition}

\begin{notation}
	Suppose that we have $g\in \leftcirc\Upsilon(\Lambda)$. Set $\widetilde{W}_g = g_p^{-1}\tilde{\Lambda}_{\Int_p}\subset W_g$, and let $U_g\in \mathbb{P}(W_g)(\Int_p)$ be a hyperplane such that $pW_g + U_g = \tilde{W}_g$. Associated with this is the scheme $\mathrm{Par}^\circ_{\lambda_p}(W_g,U_g)$ from Definition~\ref{defn:W_generic_type_U} that parameterizes $W_g$-generic isotropic lines in $V^0_{\Int_p}$ such that $U$ is the kernel of the pairing of $W_g$ with $J$. 
\end{notation}

\begin{remark}
	\label{rem:isog_lambda_g_source}
Suppose that we have an algebraically closed point 
  \[
  	z = ((s_0,\iota),(t_0,\tilde{\iota}),f)\in \widehat{\mathrm{QIsog}}^{\ord}_{G,\lambda_p,\Xi}(\kappa)
  \]
  such that $(s_0,\iota)\in \widehat{\Ss}^{\ord}_{K^\flat_g,p}(\kappa)$ for $g\in \leftcirc\Upsilon(\Lambda)$. We can choose a cocharacter $\mu_g$ as in Remark~\ref{rem:ordinary_SKg}, so that it factors through $G^\flat_{g,\Int_p}$, and a trivialization $W(\kappa)\otimes_{\Int_p}\mathcal{G}_0\xrightarrow{\simeq}\mathcal{G}_{s_0}$ that is a section of $\mathrm{Ig}^{\ord}_{K^\flat_g,p}$ over $s_0$. In particular, this identifies $V^0_{\Int_p}$ with $\bm{V}_{\cris,s_0}^{F=1}$ in such a way that $W_g = g_p^{-1}\Lambda_{\Int_p}$ is identified with the $\Int_p$-submodule generated by the image of $\iota$. Now, conjugation by the crystalline realization of $f$ gives an isomorphism 
\[
  V^0_{\Rat_p}\simeq \bm{V}_{\cris,s_0}^{F=1}[p^{-1}]\xrightarrow{\simeq}\bm{V}_{\cris,t_0}^{F=1}[p^{-1}].
 \]
 By our hypotheses, the pre-image of the lattice $\bm{V}_{\cris,t_0}^{F=1}$ in $V^0_{\Rat_p}$ is identified with a lattice of the form $hV^0_{\Int_p}$ with the following properties:
 \begin{enumerate}
 	\item $h \in M_{\Int_p}(\Int_p)\lambda_p(p)M_{\Int_p}(\Int_p)$;
 	\item $hV^0_{\Int_p}\cap W_g[p^{-1}] = \tilde{W}_g$.
 \end{enumerate}
By (1) of Lemma~\ref{lem:nice_cochar}, such lattices are in canonical bijection with $\mathrm{Par}^\circ_{\lambda_p}(W_g,U_g)(\Field_p)$.
\end{remark}

\begin{remark}
	\label{rem:isog_lambda_tilde_g_target}
Suppose that we have a point $z\in \widehat{\mathrm{QIsog}}^{\ord}_{G,\lambda_p,\Xi}(\kappa)$ as above with $(t_0,\tilde{\iota})\in \widehat{\Ss}^{\ord}_{K^\flat_{\tilde{g}},p}(\kappa)$ for $\tilde{g}\in \leftcirc\Upsilon(\tilde{\Lambda})$. Choose the cocharacter $\mu_{\tilde{g}}$ as in Remark~\ref{rem:ordinary_SKg}, so that it factors through $G^\flat_{\tilde{g},\Int_p}$. Set
\[
  \tilde{W}_{\tilde{g}} = \tilde{g}_p^{-1}\tilde{\Lambda}_{\Int_p}\subset V^0_{\Int_p}.
\]
Choose also a trivialization $W(\kappa)\otimes_{\Int_p}\mathcal{G}_0\xrightarrow{\simeq}\mathcal{G}_{t_0}$ that is a section of $\mathrm{Ig}^{\ord}_{K^\flat_{\tilde{g}},p}$ over $t_0$. Conjugation by the crystalline realization of $f$ now gives us an isomorphism $\bm{V}_{\cris,s_0}^{F=1}[p^{-1}]\xrightarrow{\simeq}V^0_{\Rat_p}$ such that the image of $\bm{V}_{\cris,s_0}^{F=1}$ is of the form $\tilde{h}V^0_{\Int_p}$ with the following properties:
\begin{enumerate}
		\item $\tilde{h}\in M_{\Int_p}(\Int_p)\lambda_p(p)^{-1}M_{\Int_p}(\Int_p)$;
 	\item $\tilde{h}V^0_{\Int_p}\cap \tilde{W}_{\tilde{g}}[p^{-1}] = W_{\tilde{g}}\defn \tilde{g}_p^{-1}\Lambda_{\Int_p}$.
\end{enumerate}
By (2) of Lemma~\ref{lem:nice_cochar}, there is a unique such lattice.
\end{remark}


\begin{notation}
	Set $\mathsf{X}^\flat_g \defn \mathrm{Par}^\circ_{\lambda_g}(W_g,U_g)$  and
\[
	 \widehat{\mathrm{Ig}}^{\ord}_{\mathsf{X}^\flat_g} \defn \widehat{\mathrm{Ig}}^{\ord}_{K^\flat_g,p}\times^{M^\flat_{g,\Int_p}(\Int_p)}\mathsf{X}^\flat_g(\Field_p).
\]
\end{notation}


\begin{proposition}
\label{prop:ord_corr_cycles}
Fix $g\in \leftcirc\Upsilon(\Lambda)$ and $\tilde{g}\in \leftcirc\Upsilon(\tilde{\Lambda})$. Then:
\begin{enumerate}
	\item The map $s^{\ord}_{\lambda_p,\Xi,K^\flat_g}$ factors as
	\[
		\widehat{\mathrm{QIsog}}^{\ord}_{G,\lambda_p,\Xi,g}\to \widehat{\mathrm{Ig}}^{\ord}_{\mathsf{X}^\flat_g}\to \widehat{\Ss}^{\ord}_{K^\flat_g,p}
	\]
	where the first map is a finite flat homeomorphism.
\item The map $t^{\ord}_{\lambda_p,\Xi,K^\flat_{\tilde{g}}}$ is an isomorphism
\end{enumerate}
\end{proposition}	
\begin{proof}
	By Remark~\ref{rem:isogG_lambdap_etale_struct}, we have a commuting diagram
	\[
		\begin{diagram}
		\widehat{\mathrm{QIsog}}^{\ord}_{G,\lambda_p,\Xi}&\rTo&\widehat{\mathrm{Ig}}^{\ord}_{K^\flat_g,p}\times^{M_{g,\Int_p}^\flat(\Int_p)}\mathrm{Par}_{\lambda_g}(\Field_p)\\
		&\rdTo(1,2)_s\ldTo(1,2)\\
		&\widehat{\Ss}^{\ord}_{K^\flat_g,p}
	\end{diagram}
	\]
	Here $\mathrm{Par}_{\lambda_g}(\Field_p)$ is the set of isotropic lines in $V^0_{\Field_p}$. Remark~\ref{rem:isog_lambda_g_source} shows that, on $\kappa$-points for $\kappa$ algebraically closed, the top arrow maps isomorphically onto the $\kappa$-points of
	\[
		\widehat{\mathrm{Ig}}^{\ord}_{K^\flat_g,p}\times^{M_{g,\Int_p}^\flat(\Int_p)}\mathrm{Par}_{\lambda_g}^{\circ}(W_g,U_g)(\Field_p) =  \widehat{\mathrm{Ig}}^{\ord}_{\mathsf{X}^\flat_g}.
	\]
	Since $\widehat{\mathrm{Ig}}^{\ord}_{\mathsf{X}^\flat_g}$ is \'etale over the base, this shows that the top horizontal map actually factors through a map
  \[
	\widehat{\mathrm{QIsog}}^{\ord}_{G,\lambda_p,\Xi,g}\to \widehat{\mathrm{Ig}}^{\ord}_{\mathsf{X}^\flat_g}.
	\]
	that is an isomorphism on algebraically closed points. Similarly, Remark~\ref{rem:isog_lambda_tilde_g_target} shows that the map $t^{\ord}_{\lambda_p,\Xi,K^\flat_{\tilde{g}}}$ is a bijection on $\kappa$-points for $\kappa$ algebraically closed. 

  To finish the proof, we can now work on the level of deformation rings. More precisely, suppose that we are in the situation of Remark~\ref{rem:isog_lambda_g_source}.  If $\widehat{U}_z$ (resp. $\widehat{U}^\flat_{s_0}$, $\widehat{U}^{\flat}_{t_0}$) is the completion of $\widehat{\mathrm{QIsog}}^{\ord}_{G,\lambda_p,\Xi}$  at $z$ (resp. $\widehat{\Ss}^{\ord}_{K^\flat_g,p}$ at $(s_0,\iota)$, resp. $\widehat{\mathcal{Z}}^{\ord}_{K}(\tilde{\Lambda})$ at $(t_0,\tilde\iota)$), we want to show that the source map $\widehat{U}_z\to \widehat{U}^{\flat}_{s_0}$ is a finite flat homeomorphism and that the target map $\widehat{U}_z\to \widehat{U}^\flat_{t_0}$ is an isomorphism. 

  By Remark~\ref{rem:isog_lambda_g_source} again, we can work with the cocharacter $\mu_g$ factoring through $G_{g,\Int_p}^\flat$ and choose the trivializations 
\[
  \mathcal{G}_{t_0}\xleftarrow{\simeq}W(\kappa)\otimes_{\Int_p}\mathcal{G}_0\xrightarrow{\simeq}\mathcal{G}_{s_0}
 \]
 such that the crystalline realization of $f$ from $\bm{V}_{\cris,s_0}[p^{-1}]$ to $\bm{V}_{\cris,t_0}[p^{-1}]$ is carried via these isomorphisms to the map
 \[
 	W(\kappa)\otimes_{\Int_p}V_{\Int_p}[p^{-1}]\xrightarrow{1\otimes \lambda_g(p)}W(\kappa)\otimes_{\Int_p}V_{\Int_p}[p^{-1}],
 \]
 where $\lambda_g$ is a cocharacter of $M_{\Int_p}$ that is conjugate to $\lambda_p$ and whose $-1$-eigenspace maps to an isotropic line in $\mathrm{Par}^\circ_{\lambda_p}(W_g,U_g)(\Int_p)$. That is, if $V_{\Int_p}^0 = \oplus_{i=-1}^1V^0_{\Int_p}(i)$ is the weight space decomposition for $\lambda_g$, then $W_g$ pairs non-degenerately with $V^0_{\Int_p}(-1)$ (equivalently, its projection on $V^0_{\Int_p}(1)$ is surjective), and $V^0_{\Int_p}(-1)$ maps onto a direct summand of $V^0_{\Int_p}/W_g$. Furthermore, $\tilde{W}_g = \lambda_g(p)V^0_{\Int_p}\cap W_g[p^{-1}]$.

 Combining this with Remarks~\ref{rem:ordinary_SKg_character_groups} and~\ref{rem:ord_def_ring_lambda_character_groups}, one sees that the map
 \[
 	\widehat{U}_z\to \widehat{U}^\flat_{s_0}\times \widehat{U}^\flat_{t_0}
 \]
 is isomorphic to a map of formal diagonalizable groups 
 \[
 	\widehat{T}_{\lambda_g}\to \widehat{T}^\flat_1\times \widehat{T}^\flat_2
 \]
 over $W(\kappa)$, where $\widehat{T}^\flat_1$ (resp. $\widehat{T}^\flat_2$) has character group $V^0_{\Int_p}/W_g$ (resp. $V^0_{\Int_p}/\lambda_g(p)^{-1}\tilde{W}_g$) and the character group of	$\widehat{T}_{\lambda_g}$ is the cokernel $M$ of the map
 \begin{align*}
 	V^0_{\Int_p} = \bigoplus_{i=-1}^1V^0_{\Int_p}(i)&\to V^0_{\Int_p}/W_g \oplus V^0_{\Int_p}/\lambda_g(p)^{-1}\tilde{W}_g\\
 	(z_{-1},z_0,z_1)&\mapsto \bigl((z_{-1},z_0,pz_1)+W_g,(pz_{-1},z_0,z_1)+\lambda_g(p)^{-1}\tilde{W}_g\bigr).
 \end{align*}
 I now claim which shows that $V^0_{\Int_p}/W_g$ maps injectively onto an index $p$ subgroup of $M$ while $V^0_{\Int_p}/\lambda_g(p)^{-1}\tilde{W}_g$ maps isomorphically onto $M$. This will imply that the map $\widehat{T}_{\lambda_g}\to \widehat{T}^\flat_1$ is a finite flat homeomorphism and that $\widehat{T}_{\lambda_g}\to \widehat{T}^\flat_2$ is an isomorphism, and so will complete the proof of the proposition.
 
 To prove the claim, note that the the elements of the kernel of $V^0_{\Int_p}/W_g\to M$ have representatives of the form $(z_{-1},z_0,pz_1)$ where $(pz_{-1},z_0,z_1)$ belongs to $\lambda_g(p)^{-1}\tilde{W}_g$. But then
 \[
 	(z_{-1},z_0,pz_1)=\lambda_g(p)\cdot (pz_{-1},z_0,z_1)\in \tilde{W}_g\subset W_g.
 \]
 This shows the injectivity of $V^0_{\Int_p}/W_g\to M$. Similarly, the kernel of $V^0_{\Int_p}/\lambda_g(p)^{-1}\tilde{W}_g\to M$ consists of elements represented by $(pz_{-1},z_0,z_1)$ with $(z_{-1},z_0,pz_1)\in W_g$. But then 
 \[
 	(pz_{-1},z_0,z_1)\in  V^0_{\Int_p}\cap \lambda_g(p)^{-1}W_g = \lambda_p(g)^{-1}\tilde{W}_g.
 \]
 Thus this map is also injective.

 Now, since $W_g$ projects surjectively onto $V^0_{\Int_p}(1)$, every element of $V^0_{\Int_p}/W_g$ can be represented by a tuple of the form $(z_{-1},z_0,0)$ and so has the same image in $M$ as the element $-(pz_{-1},z_0,0)+\lambda_g(p)^{-1}\tilde{W}_g\in V^0_{\Int_p}/\lambda_g(p)^{-1}\tilde{W}_g$. This shows that $V^0_{\Int_p}/\lambda_g(p)^{-1}\tilde{W}_g$ surjects onto $M$, and also shows that the image of $V^0_{\Int_p}/W_g$ has index $p$ in $M$.
\end{proof}

\subsection{Unique specialization of connected components}\label{subsec:spec_conn}

Here we will prove the main technical result of this paper, Proposition~\ref{prop:irred}, and so complete the proof of Theorem~\ref{introthm:main}.

\begin{remark}
	Let $\Xi = (\Lambda\subset \tilde{\Lambda})$ be a minimal pair, and write
\[
\pi^{\ord}_K:s^{\ord}_{\lambda_0,\Xi,K}\circ (t^{\ord}_{\lambda_0,\Xi,K})^{-1}:\leftcirc\widehat{\mathcal{Z}}^{\ord}_K(\tilde{\Lambda})_{\Field_p}\to \leftcirc\widehat{\mathcal{Z}}^{\ord}_K(\Lambda)\;;\; \pi_K:s_{\mu,\Xi,K}\circ (t_{\mu,\Xi,K})^{-1}:\leftcirc Z_K(\tilde{\Lambda})\to \leftcirc Z_K(\Lambda).
\]

We obtain a diagram:
\begin{align}\label{main:eqn:pi0diagram}
\begin{diagram}
    \pi_{0}\bigl(\leftcirc \mathcal{Z}_K(\tilde{\Lambda})_{\overline{\Field}_p}\bigr)&\rTo^{\simeq}&\pi_0\bigl(\leftcirc\mathcal{Z}^{\ord}_K(\tilde{\Lambda})_{\overline{\Field}_p}\bigr)&\rTo&\pi_0\bigl(\leftcirc Z_K(\tilde{\Lambda})_{\overline{\Rat}}\bigr);\\
    &&\dTo_{\pi^{\ord}_{K,*}}&&\dTo^{\simeq}_{\pi_{K,*}}\\
    \pi_{0}\bigl(\leftcirc \mathcal{Z}_K(\Lambda)_{\overline{\Field}_p}\bigr)&\rTo_{\simeq}&\pi_0\bigl(\leftcirc\mathcal{Z}^{\ord}_K(\Lambda)_{\overline{\Field}_p}\bigr)&\rTo&\pi_0\bigl(\leftcirc Z_K(\Lambda)_{\overline{\Rat}}\bigr).
  \end{diagram}
  \end{align}
The horizontal maps are the natural ones, and the ones on the left are bijections because the ordinary locus is dense in $\leftcirc \mathcal{Z}_K(\Lambda)_{\Field_p}$; see (3) of Proposition~\ref{prop:ZLambda_props}. The vertical map on the right hand side is a bijection by Proposition~\ref{prop:char0_corr_cycles}. One can verify that the square on the right commutes using Remarks~\ref{rem:qisog_comparison_global} and~\ref{rem:qisog_points_comparison_types}.
\end{remark}

\begin{notation}
Let $P(\Lambda)$ be the assertion that the lower right horizontal arrow is bijective in the diagram above for all prime-to-$p$ levels $K^p$, and let $Q(\Xi)$ be the assertion that the vertical arrow in the middle is bijective for all prime-to-$p$ levels $K^p$.
\end{notation}

\begin{proposition}\label{prop:p(m)_maximal}
$P(\Lambda)$ is true if $\Lambda$ is \emph{maximal} at $p$. 
\end{proposition}
\begin{proof}
The maximality at $p$ here means (by definition) that, for all $\Lambda'\in \mathsf{L}(\Lambda)$, $[\Lambda':\Lambda]$ is prime to $p$. Therefore, by Remark~\ref{rem:open_closed_prime_to_p}, we see that $\leftcirc \mathcal{Z}_K(\Lambda)_{\Field_p}$ is an open and closed substack of $\mathcal{Z}_K(\Lambda)_{\Field_p}$, and in particular is finite over $\Ss_K$. Moreover, Proposition~\ref{prop:ZLambda_props} tells us that $\leftcirc\mathcal{Z}_K(\Lambda)_{\Int_{(p)}}$ and $\leftcirc\mathcal{Z}_K(\Lambda)_{\overline\Field_p}$ are normal. In particular, the irreducible and connected components of the latter stack agree.

By Remark~\ref{rem:leftcirc_ZLambda_integral_models} and the finiteness of $\leftcirc\mathcal{Z}_K(\Lambda)_{\Int_{(p)}}\to \Ss_{K,(p)}$ , we have a decomposition into normal stacks
\[
\leftcirc\mathcal{Z}_K(\Lambda)_{\Int_{(p)}}\simeq \bigsqcup_{g\in G^\flat(\Rat)\backslash \leftcirc\Xi_K(\Lambda)/K} \Ss_{K^\flat_g,(p)}
\]
finite over $\Ss_{K,(p)}$.

Therefore~\cite[Corollary 4.1.11]{mp:toroidal} shows that every connected component of its special fiber is the specialization of a unique connected component of its generic fiber.
\end{proof}

\begin{proposition}\label{prop:p(m)_to_q(m)}
$P(\Lambda)$ implies $Q(\Xi)$ for all minimal pairs of the form $\Xi = (\tilde{\Lambda}\subset \Lambda)$.
\end{proposition}
\begin{proof}
By Proposition~\ref{prop:ord_corr_cycles}, this reduces to knowing that, for all $g\in \leftcirc\Xi_K(\Lambda)$, the map
\[
\pi_0\left(\widehat{\mathrm{Ig}}^{\ord}_{\mathsf{X}^\flat_g,\overline{\Field}_p}\right)\to \pi_0(\mathcal{S}^{\ord}_{K^\flat_g,\overline{\Field}_p})
\]
is a bijection. We will check this using Proposition~\ref{proposition:chai_hida_application}. Assertions (1)-(3) are taken care of by Lemma~\ref{lem:smooth_grp_schemes}. Assertion (4) holds because we are assuming $P(\Lambda)$ (and by (3) of Proposition~\ref{prop:ZLambda_props}). Our group $G$ is already isotropic over $\Rat$, so Assertion (5) is okay. The last Assertion (6) is also valid by Lemma~\ref{lem:ordinary_SKg_hypersymmetric}.
\end{proof}

\begin{proposition}\label{prop:irred}
For any $\Lambda$, we have a bijection
\[
 \pi_{0}\bigl(\leftcirc \mathcal{Z}_K(\Lambda)_{\overline{\Field}_p}\bigr)\xrightarrow{\simeq}\pi_{0}\bigl(\leftcirc Z_K(\Lambda)_{\overline{\Rat}}\bigr).
\]
\end{proposition}
\begin{proof}
  Using induction on the $p$-adic valuation of $[\Lambda^\vee:\Lambda]$, this follows from the diagram~\eqref{main:eqn:pi0diagram} and Propositions~\ref{prop:p(m)_maximal} and~\ref{prop:p(m)_to_q(m)}.
\end{proof}

\subsection{Application to the moduli of K3 surfaces}
When $\Lambda$ is a rank $1$ lattice spanned by a vector $v$ with $Q(v) = m$, write $\mathcal{Z}_K(m)$ for $\mathcal{Z}_K(\Lambda)$, and similarly for all the sub-loci of the special cycle. For an integer $d\ge 1$, let $\mathsf{M}^{\circ}_{2d,(p)}$ be the moduli stack of primitively polarized K3 surfaces over $\Int_{(p)}$ of degree $2d$ (see~\cite[\S 3]{mp:tatek3}).

\begin{proof}
[Proof of Theorem~\ref{introthm:k3s}]
Let $N$ be the self-dual quadratic $\Int$-lattice $U^{\oplus 3} \oplus E_8^{\oplus 2}$, where $U$ is the hyperbolic plane. Choose a hyperbolic basis $e,f$ for the first copy of $U$. Set
\[
L_d = \langle e-df\rangle^{\perp}\subset N.
\]
This is a quadratic space of signature $(19,2)$. We can choose our quadratic space $V$, and the $\Int$-lattice $V_{\Int}$ such that $V$ has signature $(20,2)$, $V_{\Int_{(p)}}$ is self-dual, and such that there exists an isometric embedding as a direct summand
\[
L_{d}\hookrightarrow V_{\Int}.
\]

Associated with the lattice $V_{\Int}$ and a suitable neat level subgroup $K^p\subset \GSpin(V)(\Adele^p_f)$, we have the integral model $\Ss_{K_pK^p,(p)}$ over $\Int_{(p)}$.

Let $\mathsf{M}^{\mathrm{sm}}_{2d,(p)}$ be the open smooth locus of $\mathsf{M}_{2d,(p)}$: This is a fiber-by-fiber dense subspace. In particular, it suffices to show that $\mathsf{M}^{\mathrm{sm}}_{2d,\overline{\Field}_p}$ is irreducible.

By the theory of~\cite[\S 5]{mp:tatek3}, extended to the case $p=2$ in~\cite[Prop. A. 12]{Kim2016-fb} (see also the erratum at~\cite{mp:erratum}), there is a finite \'etale cover $\tilde{\mathsf{M}}^{\mathrm{sm}}_{2d,K}$ of $\mathsf{M}^{\mathrm{sm}}_{2d,\Int_{(p)}}$, and an \'etale period map
\[
\tilde{\mathsf{M}}^{\mathrm{sm}}_{2d,K}\to \mathcal{Z}_K(2d)_{\Int_{(p)}}
\]
that is in fact an open immersion, since it is one in the generic fiber; see~\cite[Cor. 5.15]{mp:tatek3}. This map lands inside $\leftcirc\mathcal{Z}_K(2d)_{\Int_{(p)}}$, and in fact inside the \emph{primitive} locus $\mathcal{Z}^{\mathrm{pr}}_K(2d)_{\Int_{(p)}}$ where the de Rham realization of the special endomorphism of degree $2d$ generates a direct summand of $\bm{V}_{\dR}$.\footnote{It is the verification of this primitivity when $p=2$ that is the main focus of the erratum~\cite{mp:erratum}.}

Combined with Proposition~\ref{prop:irred}, this shows that every irreducible component of $\tilde{\mathsf{M}}^{\mathrm{sm}}_{2d,K,\overline{\Field}_p}$ is the specialization of a unique irreducible component of $\tilde{\mathsf{M}}^{\mathrm{sm}}_{2d,K,\overline{\Rat}}$. From this, we deduce the same assertion for the fibers of $\mathsf{M}^{\mathrm{sm}}_{2d,(p)}$. However, it is well-known that this moduli stack is irreducible over $\Comp$. For instance, this follows from the Torelli theorem; see the proof of~\cite[Prop. 5.3]{mp:tatek3}.
\end{proof}

\printbibliography

\end{document}

%% file: commandssubsub.tex
\newcommand{\defnword}[1]{\textbf{#1}}
\newcommand{\comment}[1]{}
\newcommand{\on}[1]{\operatorname{#1}}
\newcommand{\bb}[1]{\mathbb{#1}}

\newcommand{\ord}{\on{ord}}

\DeclareSymbolFontAlphabet{\mathbbl}{bbold}
\DeclareSymbolFontAlphabet{\mathbb}{AMSb}

\numberwithin{equation}{subsubsection}
\newtheorem{introthm}{Theorem}

\newtheorem{proposition}[subsubsection]{Proposition}

\newtheorem*{thm*}{Theorem}
\newtheorem{lemma}[subsubsection]{Lemma}

\newtheorem*{lem*}{Lemma}
\newtheorem{corollary}[subsubsection]{Corollary}

\theoremstyle{definition}
\newtheorem{definition}[subsubsection]{Definition}
\newtheorem{construction}[subsubsection]{Construction}
\newtheorem{example}[subsubsection]{Example}

\newtheorem{remark}[subsubsection]{Remark}
\newtheorem{notation}[subsubsection]{Notation}

\newtheorem{assumption}[subsubsection]{Assumption}

\newtheorem*{assump*}{Assumption}

\newarrow{Equals}{=}{=}{}{=}{}
\newarrow{Onto}{-}{-}{-}{-}{>>}

\newcommand{\Real}{\mathbb{R}}
\newcommand{\Int}{\mathbb{Z}}

\newcommand{\Comp}{\mathbb{C}}

\newcommand{\Adele}{\mathbb{A}}
\newcommand{\Field}{\mathbb{F}}
\newcommand{\Rat}{\mathbb{Q}}
\newcommand{\Lie}{\on{Lie}}
\newcommand{\Dieu}{\mathbb{D}}

\newcommand{\defn}{\overset{\mathrm{defn}}{=}}

\newcommand{\leftcirc}{{}^{\circ}}

\newcommand{\im}{\on{im}}

\newcommand{\into}{\hookrightarrow}

\newcommand{\Obj}{\on{Obj}}

\newcommand{\Rg}{\mathscr{O}}
\newcommand{\Reg}[1]{\Rg_{#1}}

\newcommand{\Hom}{\on{Hom}}
\newcommand{\End}{\on{End}}

\newcommand{\Ext}{\on{Ext}}
\newcommand{\Aut}{\on{Aut}}

\newcommand{\Gal}{\on{Gal}}

\newcommand{\Gm}{\mathbb{G}_m}

\newcommand{\Gmh}[1]{\mathbb{G}_{m,#1}}

\newcommand{\cris}{\on{cris}}

\newcommand{\dR}{\on{dR}}

\newcommand{\pow}[1]{[\vert#1\vert]}

\newcommand{\Spec}{\on{Spec}}
\newcommand{\Spf}{\on{Spf}}

\newcommand{\et}{\on{\acute{e}t}}

\newcommand{\mult}{\on{mult}}

\newcommand{\dlog}{\on{dlog}}

\newcommand{\Fil}{\on{Fil}}

\newcommand{\mx}{\mathfrak{m}}

\newcommand{\Sh}{\on{Sh}}
\newcommand{\Ss}{\mathcal{S}}

\newcommand{\Res}{\on{Res}}
\newcommand{\an}{\on{an}}

\newcommand{\ad}{\on{ad}}
\newcommand{\GSp}{\on{GSp}}

\newcommand{\GL}{\on{GL}}

\newcommand{\SO}{\on{SO}}

\newcommand{\Spin}{\on{Spin}}
\newcommand{\GSpin}{\on{GSpin}}
